\documentclass[12pt,a4paper]{amsart}
\usepackage{amssymb,stmaryrd,blkarray,varwidth}
\usepackage[a4paper,left=2.25cm,right=2.25cm,top=3cm,bottom=3cm,headsep=1cm]{geometry}
\usepackage{euler} 
\usepackage{tikz}\usetikzlibrary{graphs,quotes,fit,positioning,matrix,calc,decorations.markings,angles,decorations.pathmorphing,decorations.pathreplacing}
\newcommand\horiz{\tikz{\draw[thick] (0,0) -- (1,0); }}
\newcommand\NW{\tikz{\draw[thick] (0,0) -- (120:1); }}
\newcommand\NE{\tikz{\draw[thick] (0,0) -- (240:1); }}
\tikzset{mynode/.style={circle,draw=black,fill=black,inner sep=1.8pt,outer sep=0pt}}
\tikzset{edgelabel/.style={\mcol,inner sep=0pt}}
\tikzset{invlabel/.style={draw=black,text=black,circle,inner sep=0pt,minimum size=3mm}}
\newcommand\tikzif[2][]{
\tikzifinpicture{#2}{\begin{tikzpicture}[#1]#2\end{tikzpicture}}
}
\tikzset{math mode/.style = {execute at begin node=$, execute at end node=$}}
\def\dr{red!80!black} 
\def\dg{green!80!black} 
\def\db{blue!80!black} 
\tikzset{d/.style={ultra thick}}
\tikzset{dr/.style={draw=\dr,d}}
\tikzset{dg/.style={draw=\dg,d}}
\tikzset{db/.style={draw=\db,d}} 
\def\mcol{black}
\def\m#1{{\color{\mcol}#1}}
\tikzset{rt/.style={text=blue,execute at begin node=$\sf,execute at end node=$}}
\def\rt#1{{\color{blue}\mathsf{#1}}}
\makeatletter
\newcommand{\gettikzxy}[3]{
  \tikz@scan@one@point\pgfutil@firstofone#1\relax
\pgfmathsetmacro{#2}{\the\pgf@x/\linkpatternunit}
\pgfmathsetmacro{#3}{\the\pgf@y/\linkpatternunit}
}
\tikzset{label anchor/.code={%
    \let\tikz@auto@anchor=\pgfutil@empty
    \def\tikz@anchor{#1}
  },
  label anchor/.default=center
}
\makeatother
\tikzset{arrow/.style={postaction={decorate,thick,decoration={markings,mark = at position #1 with {\arrow{>}}}}},arrow/.default=0.5}
\tikzset{invarrow/.style={postaction={decorate,thick,decoration={markings,mark = at position #1 with {\arrow{<}}}}},invarrow/.default=0.5}
\newdimen\linkpatternunit%
\newcount\linkpatternsize%
\newcount\lpsize
\newif\iflinkpatterninverted
\newif\iflinkpatterntikzstarted
\newif\iflinkpatternboxed
\newif\iflinkpatternaxis
\newif\iflinkpatternstraightlines
\newif\iflinkpatternnumbered
\newif\iflinkpatternalias
\newif\iflinkpatternnode
\newif\iflinkpatterncentered
\newcount\linkpatternfused
\pgfkeys{/linkpattern/.cd,centered/.is if=linkpatterncentered,inverted/.is if=linkpatterninverted,numbered/.is if=linkpatternnumbered,tikzstarted/.is if=linkpatterntikzstarted,straight lines/.is if=linkpatternstraightlines,boxed/.is if=linkpatternboxed,axis/.is if=linkpatternaxis,vertexcolor/.store in=\linkpatternvertexcolor,edgecolor/.store in=\linkpatternedgecolor,boxcolor/.code={\linkpatternboxedtrue\def\linkpatternboxcolor{#1}},tikzoptions/.style={every linkpattern/.append style={#1}},size/.code={\linkpatternsize=#1},numbering0/.code={\def\lpnumbering{#1}\def\linkpatternnumbering{#1}},numbering/.style={numbered,numbering0={#1}},unit/.code={\linkpatternunit=#1},height/.store in=\linkpatternheight,shape/.store in=\linkpatternshape,looseness/.store in=\linkpatternlooseness,squareness/.store in=\linkpatternsquareness,extra space/.store in=\linkpatternextraspace,width/.store in=\linkpatternwidth,alias/.is if=linkpatternalias,pos/.store in=\linkpatternpos,labeloptions/.style={labeloptionslist/.append style={#1}},labeloptionslist/.style={inner sep=2pt,font=\scriptsize,execute at begin node=$,execute at end node=$,label anchor={#1+180}},nodeon/.is if=linkpatternnode,node/.style={nodeon,nodeoptionslist/.append style={#1}},nodeoptionslist/.style={},
pipedream/.style={shape=pipedream,looseness=0,straight lines,numbering0=tangle},
tangle/.style={shape=tangle,numbering0=tangle},
every linkpattern/.style={x=\linkpatternunit,y=\linkpatternunit},
fused/.code={\linkpatternfused=#1},
vertex/.style={circle,thin,solid,draw=black,fill=\linkpatternvertexcolor,inner sep=1.5pt,draw opacity=1,transform shape},
edge/.style={very thick,solid,draw=\linkpatternedgecolor,draw opacity=1}}
\linkpatterncenteredfalse%
\linkpatterninvertedfalse%
\linkpatternnumberedfalse%
\linkpatterntikzstartedfalse%
\linkpatternboxedfalse%
\linkpatternaxistrue%
\linkpatternaliastrue%
\linkpatternstraightlinesfalse%
\def\linkpatternlooseness{0.2}
\def\linkpatternsquareness{0.35}
\def\linkpatternvertexcolor{red}%
\def\linkpatternedgecolor{blue}%
\def\linkpatternboxcolor{none}%
\def\linkpatternheight{0}
\def\linkpatternwidth{0}
\def\linkpatternshape{default}
\def\linkpatternnumbering{default}
\def\linkpatternpos{(0,0)}
\def\linkpatternextraspace{0}
\linkpatternnodefalse
\def\firstchar#1#2\empty{#1}%
\def\linkpatterndo#1#2{
\edef\param{\csname linkpattern#2\endcsname}
\edef\firstcharparam{\expandafter\firstchar\param\empty}
\expandafter\ifcat\firstcharparam a
\expandafter\ifx\csname linkpattern#1\param\endcsname\relax
\csname linkpattern#1unknown\endcsname
\else
\csname linkpattern#1\csname linkpattern#2\endcsname\endcsname
\fi
\else
\csname linkpattern#1unknown\endcsname
\fi
}%
\def\linkpatterncoordtangle{\ifnum\x>\lphalfsize\pgfmathparse{\lpsize+1-\x}\xdef\lpcoordx{\pgfmathresult}\xdef\lpcoordy{\lpheight}\xdef\lpangle{270}\else\xdef\lpcoordx{\x}\xdef\lpcoordy{-\lpheight}\xdef\lpangle{90}\fi}
\def\linkpatterncoordpipedream{\ifnum\x>\lphalfsize\pgfmathparse{\lpsize+1-\x-0.5}\xdef\lpcoordx{\pgfmathresult}\xdef\lpcoordy{0}\xdef\lpangle{270}\else\pgfmathparse{0.5-\x}\xdef\lpcoordy{\pgfmathresult}\xdef\lpcoordx{0}\xdef\lpangle{0}\fi}
\def\linkpatterncoordrectangle{
\ifnum\x>\lptqsize
\pgfmathparse{\lpsize+1-\x-0.5}\xdef\lpcoordx{\pgfmathresult}\xdef\lpcoordy{0}\xdef\lpangle{270}
\else\ifnum\x>\lphalfsize
\pgfmathparse{\x-\lptqsize-0.5}\xdef\lpcoordy{\pgfmathresult}\xdef\lpcoordx{\linkpatternwidth}\xdef\lpangle{180}
\else\ifnum\x>\linkpatternheight
\pgfmathparse{\x-\linkpatternheight-0.5}\xdef\lpcoordx{\pgfmathresult}\xdef\lpcoordy{-\linkpatternheight}\xdef\lpangle{90}
\else
\pgfmathparse{0.5-\x}\xdef\lpcoordy{\pgfmathresult}\xdef\lpcoordx{0}\xdef\lpangle{0}
\fi\fi\fi
}%
\def\linkpatternsetsizeunknown{
\global\lpsize=\linkpatternsize
\if\linkpatternheight0
\xdef\maxsep{0}
\foreach \x/\xx in \mylist%
{%
\edef\tempx{\withoutprime{\x}}
\edef\tempxx{\withoutprime{\xx}}
\pgfmathparse{max(\maxsep,abs(\tempx-\tempxx))}
\xdef\maxsep{\pgfmathresult}
}%
\pgfmathparse{0.25+0.8*\linkpatternsquareness*\maxsep}
\xdef\lpheight{\pgfmathresult}
\else
\xdef\lpheight{\linkpatternheight}
\fi
}
\newcount\tempsize
\def\linkpatternrightmostunknown{
\global\lpsize=0
\global\tempsize=0
\foreach\x/\labx in \linkpatternnumbering
{
\edef\tempx{\withoutprime{\x}}
\ifnum\lpsize<\tempx\global\lpsize=\tempx\fi
\global\advance\tempsize by 1
}
\ifnum\tempsize>\lpsize\global\lpsize=\tempsize\fi
}%
\def\linkpatternrightmostdefault{
\global\lpsize=0
\global\tempsize=0
\foreach \x/\y in \mylist
{
\edef\tempx{\withoutprime{\x}}
\ifnum\lpsize<\tempx\global\lpsize=\tempx\fi
\ifx\x\y
\global\advance\tempsize by 1
\else
\edef\tempy{\withoutprime{\y}}
\ifnum\lpsize<\tempy\global\lpsize=\tempy\fi%
\global\advance\tempsize by 2
\fi
}
\ifnum\tempsize>\lpsize\global\lpsize=\tempsize\fi
}%
\def\linkpatternrightmosttangle{
\global\lpsize=0
\global\tempsize=0
\foreach \x/\y in \mylist
{
\edef\tempx{\withoutprime{\x}}
\ifnum\lpsize<\tempx\global\lpsize=\tempx\fi
\ifx\x\y
\global\advance\tempsize by 1
\else
\edef\tempy{\withoutprime{\y}}
\ifnum\lpsize<\tempy\global\lpsize=\tempy\fi%
\global\advance\tempsize by 2
\fi
}
\global\advance\lpsize by\lpsize
\ifnum\tempsize>\lpsize\global\lpsize=\tempsize\fi
}%

\newcommand\linkpattern[2][]{
{
\pgfkeys{/linkpattern/.cd,#1}
\edef\mylist{#2}
\def\primetest##1'{}%
\def\hasaprime##1{\expandafter\primetest##1''}
\def\internalwithoutprime##1'{##1}%
\def\withoutprime##1{\if\hasaprime##1 %
\expandafter\internalwithoutprime##1\else ##1\fi}%
\iflinkpatternnumbered%
\iflinkpatterninverted
\tikzset{/linkpattern/lbl/.style n args={3}{label={[/linkpattern/labeloptionslist=-##1,##3] ##1:##2}}}%
\else%
\tikzset{/linkpattern/lbl/.style n args={3}{label={[/linkpattern/labeloptionslist=##1,##3] ##1:##2}}}%
\fi%
\else%
\tikzset{/linkpattern/lbl/.style={}}%
\fi%
\tikzifinpicture{\linkpatterntikzstartedtrue%
\begin{scope}[shift=\linkpatternpos,/linkpattern/every linkpattern]
}{%
\linkpatterntikzstartedfalse%
\iflinkpatterncentered
\begin{tikzpicture}[baseline=(current  bounding  box.center),/linkpattern/every linkpattern]%
\else
\begin{tikzpicture}[baseline=0,/linkpattern/every linkpattern]%
\fi
}%
\begin{scope}[local bounding box=link pattern box]
\iflinkpatterninverted%
\begin{scope}[yscale=-1]%
\fi%
\linkpatterndo{setsize}{shape}
\ifnum\lpsize=0
\linkpatterndo{rightmost}{numbering}
\fi
\pgfmathtruncatemacro{\lphalfsize}{\lpsize/2}
\linkpatterndo{numbering}{numbering}
\iflinkpatternboxed
\linkpatterndo{drawbox}{shape}
\else
\iflinkpatternaxis
\linkpatterndo{drawaxis}{shape}
\fi
\fi
\foreach\xx/\xlab/\opt in \lpnumbering
{
\ifx\xlab\opt\def\opt{}\fi
\if\hasaprime\xx %
\pgfmathtruncatemacro{\xx}{\lpsize+1-\withoutprime{\xx}}
\fi
\ifnum\linkpatternfused>1
\pgfmathsetmacro{\x}{0.4*(0.5+\linkpatternfused*(0.5+floor((\xx-1)/\linkpatternfused)))+0.6*\xx}
\else
\def\x{\xx}
\fi
\linkpatterndo{coord}{shape}
\iflinkpatternalias\def\xlabb{\xlab}\else\def\xlabb{\xx}\fi
\path (\lpcoordx,\lpcoordy) coordinate[/linkpattern/vertex,/linkpattern/lbl={\lpangle+180}{\xlab}{\opt},alias=v\xlabb] (v\xx) ++(\lpangle:\linkpatternunit) coordinate[alias=vv\xlabb] (vv\xx); 
}
\foreach \a/\b/\c in \mylist
{
\if\hasaprime\a %
\pgfmathtruncatemacro{\a}{\lpsize+1-\withoutprime{\a}}
\fi
\ifx\b\c\def\c{}\fi
\draw[/linkpattern/edge]
\ifx\a\b
(v\a)
\c
--
++(0,\lpheight);
\else
\pgfextra{
\if\hasaprime\b %
\pgfmathtruncatemacro{\b}{\lpsize+1-\withoutprime{\b}}
\fi
\gettikzxy{(v\a)}{\ax}{\ay}
\gettikzxy{(v\b)}{\bx}{\by}
\gettikzxy{(vv\a)}{\axx}{\ayy}
\gettikzxy{(vv\b)}{\bxx}{\byy}
\pgfmathsetmacro{\dist}{sqrt((\ax-\bx)*(\ax-\bx)+(\ay-\by)*(\ay-\by))}
\pgfmathsetmacro{\abx}{(\axx-\ax)*\dist*\linkpatternsquareness+(\bx-\ax)*\linkpatternlooseness)}
\pgfmathsetmacro{\aby}{(\ayy-\ay)*\dist*\linkpatternsquareness+(\by-\ay)*\linkpatternlooseness)}
\pgfmathsetmacro{\bax}{(\bxx-\bx)*\dist*\linkpatternsquareness+(\ax-\bx)*\linkpatternlooseness)}
\pgfmathsetmacro{\bay}{(\byy-\by)*\dist*\linkpatternsquareness+(\ay-\by)*\linkpatternlooseness)}
}
(v\a)
\c
\iflinkpatternstraightlines
\pgfextra{
\pgfmathsetmacro{\t}{((\ax-\bx)*\bay-(\ay-\by)*\bax)/(\aby*\bax-\abx*\bay)}
\pgfmathsetmacro{\abx}{\t*\abx}
\pgfmathsetmacro{\aby}{\t*\aby}
}
[rounded corners=0.2\linkpatternunit] -- ++(\abx,\aby) -- (v\b);
\else
.. controls ++(\abx,\aby) and ++(\bax,\bay) .. 
\fi
(v\b);
\fi
}
\end{scope}
\iflinkpatternnode
\node[fit=(link pattern box),/linkpattern/nodeoptionslist] {};
\fi
\iflinkpatterninverted
\end{scope}
\fi
\iflinkpatterntikzstarted
\end{scope}
\else%
\end{tikzpicture}%
\fi%
}}%
\newcommand\tanglelinkpattern[3][]{%
{
\pgfkeys{/linkpattern/.cd,#1}
\iflinkpatterninverted
\begin{tikzpicture}[/linkpattern/every linkpattern,baseline=\linkpatternunit]%
\else
\begin{tikzpicture}[/linkpattern/every linkpattern,baseline=-\linkpatternunit]%
\fi
\linkpattern[#1,tikzstarted,numbered=false]{#3}
\pgfmathtruncatemacro{\lptempsize}{2*\linkpatternsize}
\iflinkpatterninverted
\begin{scope}[yshift=0.5*\linkpatternunit]
\else
\begin{scope}[yshift=-0.5*\linkpatternunit]
\fi
\linkpattern[tangle,#1,tikzstarted,size=\lptempsize,
numbering=halftangle,
height=0.5]{#2}
\end{scope}
\end{tikzpicture}%
}}
\newcommand\diag[4][]{%
\pgfkeys{/linkpattern/.cd,#1}
\iflinkpatterntikzstarted\else%
\begin{tikzpicture}[scale=0.5]
\fi%
\iflinkpatterninverted%
\begin{scope}[yscale=-1]%
\fi%
\draw (0,0) grid (#2,#3);
\edef\mylist{#4}
\foreach\y/\x/\z in \mylist
{
\ifx\x\z
\draw[decorate,decoration={zigzag,
amplitude=1pt,segment length=5pt}]
(\x-0.5,#3) -- (\x-0.5,\y-0.5) node[circle,fill=black,inner sep=2pt] {} -- (#2,\y-0.5);
\else
\node at (\x-0.5,\y-0.5) {$\z$};
\fi
}
\iflinkpatterninverted
\end{scope}
\fi
\iflinkpatterntikzstarted\else%
\end{tikzpicture}%
\fi%
}
\makeatletter
\tikzset{circle split part fill/.style  args={#1,#2}{%
 alias=tmp@name,
  postaction={%
    insert path={
     \pgfextra{%
     \pgfpointdiff{\pgfpointanchor{\pgf@node@name}{center}}%
                  {\pgfpointanchor{\pgf@node@name}{east}}%
     \pgfmathsetmacro\insiderad{\pgf@x}
      \fill[#1] (\pgf@node@name.base) ([xshift=-\pgflinewidth]\pgf@node@name.east) arc
                          (0:180:\insiderad-\pgflinewidth)--cycle;
      \fill[#2] (\pgf@node@name.base) ([xshift=\pgflinewidth]\pgf@node@name.west)  arc
                           (180:360:\insiderad-\pgflinewidth)--cycle;                    }}}}}  
 \makeatother
\tikzset{bdot/.style={circle,circle split,draw,circle split part fill={black,white},thin,inner sep=1pt}}%
\tikzset{wdot/.style={circle,circle split,draw,circle split part fill={white,black},thin,inner sep=1pt}}%
\newcommand\circlelinkpattern[2][]{
{
\pgfkeys{/linkpattern/.cd,#1}
\iflinkpatterntikzstarted\else%
\begin{tikzpicture}[/linkpattern/every linkpattern]%
\fi%
\iflinkpatterninverted%
\begin{scope}[yscale=-1]%
\fi%
\global\lpsize=\linkpatternsize
\edef\mylist{#2}
\foreach \x/\y in \mylist
{
\ifnum\x>\lpsize\global\lpsize=\x\fi
\ifnum\y>\lpsize\global\lpsize=\y\fi
}
\iflinkpatternaxis
\draw (0,0) circle (1);
\fi
\foreach\x in {1,...,\lpsize}
{
\pgfmathparse{(0.3*floor((\x-1)/\linkpatternfused)+0.7*((\x-0.5)/\linkpatternfused-0.5))*\linkpatternfused*360/\lpsize}
\coordinate[/linkpattern/vertex] (v\x) at (\pgfmathresult:1);
}
\foreach \x/\y/\z in \mylist
{
\ifx\y\z%
\draw[/linkpattern/edge] (v\x) .. controls ($0.5*(v\x)$) and  ($0.5*(v\y)$) .. (v\y);
\else
\draw[/linkpattern/edge] \z (v\x) .. controls ($0.5*(v\x)$) and  ($0.5*(v\y)$) .. (v\y);
\fi
}
\iflinkpatternnumbered%
\pgfmathparse{\lpsize/\linkpatternfused}
\global\lpsize=\pgfmathresult
\def\linkpatternnumbering{1,...,\lpsize}
\newdimen\angle
\foreach\x/\xx/\opt in \linkpatternnumbering
{
  \pgfmathsetmacro{\angle}{360/\lpsize*(\x-1)}
\ifx\xx\opt%
  \node[outer sep=1pt,anchor=180+\angle] at (\angle:1) {$\scriptstyle\xx$}; 
\else
  \node[outer sep=1pt,anchor=180+\angle,\opt] at (\angle:1) {$\scriptstyle\xx$}; 
\fi
}
\fi%
\iflinkpatterninverted%
\end{scope}
\fi%
\iflinkpatterntikzstarted\else%
\end{tikzpicture}%
\fi%
}}%
\newdimen{\loopcellsize}
\tikzset{bgplaq/.style={draw=black,fill=\linkpatternboxcolor}}
\def\plaqwest{}
\def\plaqeast{}
\def\plaqnorth{}
\def\plaqsouth{}
\pgfkeys{/linkpattern/.cd,west/.store in=\plaqwest,east/.store in=\plaqeast,north/.store in=\plaqnorth,south/.store in=\plaqsouth}
\def\plaqname{plaq}
\newcommand\plaq[2][]{
\node[bgplaq,rectangle,draw,use as bounding box,minimum size=\loopcellsize,transform shape] (\plaqname) {};
\pgfkeys{/linkpattern/.cd,#1}
\ifx#2\empty\else
\begin{scope}[x=\loopcellsize,y=\loopcellsize]
\csname plaq#2\endcsname
\end{scope}\fi
}
\tikzset{loop/.code={\def\plaqname{loop-\the\pgfmatrixcurrentrow-\the\pgfmatrixcurrentcolumn}},loop/.append style={matrix,row sep={\loopcellsize,between origins},column sep={\loopcellsize,between origins}}}
\newlength\myshift
\newcommand\getshift{%
\pgfmathsetlength{\myshift}{0.8mm}
}
\def\noparenx#1{%
 \ifx\relax#1
 \else
  \if)#1%
  \else
  \if(#1%
  \else
   #1%
   \fi
   \fi
  \expandafter\noparenx
 \fi
}
\def\noparen#1{\noparenx #1\relax}
\newcommand\rh[5][]{
\tikzif[baseline=0,scale=1.25]{
\getshift
\draw[thick,black] (0,0) coordinate (ad) -- node[edgelabel,left,xshift=\myshift] {$#2$} ++(-60:1) coordinate (ba) -- node[edgelabel,right,xshift=-\myshift] {$#3$} ++(60:1) coordinate (cb) -- node[edgelabel,right,xshift=-\myshift] {$#4$} ++(120:1) coordinate (dc) -- node[edgelabel,left,xshift=\myshift] {$#5$} cycle;
\ifx\&#1\&\else\node[invlabel] at (0.5,0) {$\ss#1$};\fi
}}
\newcommand\eqrh[2]{\rh[-1]{#1}{#2}{#1}{#2}}
\newcommand\uptri[4][]{
\tikzif[baseline=0.34cm,scale=1.25]{
\getshift
\draw[thick,black] (-0.5,0) -- node[edgelabel] (horiz) {$\vphantom{\noparen{#3}}\smash{#3}$} ++(0:1) -- node[edgelabel,right,xshift=-\myshift] (NE) {$#4$} ++(120:1) -- node[edgelabel,left,xshift=\myshift] (NW) {$#2$} ++(240:1) -- cycle; 
\ifx\&#1\&\else\node[invlabel] at (0,0.33) {$\ss#1$};\fi
}}
\newcommand\downtri[4][]{
\tikzif[baseline=-0.54cm,scale=1.25]{
\getshift
\begin{scope}[scale=-1]
\draw[thick,black] (-0.5,0) -- node[edgelabel] (horiz) {$\vphantom{\noparen{#3}}\smash{#3}$} ++(0:1) -- node[edgelabel,left,xshift=\myshift] (NE) {$#4$} ++(120:1) -- node[edgelabel,right,xshift=-\myshift] (NW) {$#2$} ++(240:1) -- cycle; 
\ifx\&#1\&\else\node[invlabel] at (0,0.33) {$\ss#1$};\fi
\end{scope}
}}
\newcommand\Deltatri{{\tikz[baseline=0,scale=0.3]{\uptri{}{}{}}}}
\newcommand\nablatri{{\tikz[baseline=-0.25cm,scale=0.3]{\downtri{}{}{}}}}
\usepackage{hyperref}
\allowdisplaybreaks[1]
\renewcommand\ss{\scriptstyle}
\newcommand\sss{\scriptscriptstyle}
\newcommand\agemo{\!\!\tikz[baseline=-1mm]{\node[rotate=180]{$\omega$}}\!\!}

\DeclareMathOperator{\sign}{sign}

\newcommand\ZZ{{\mathbb Z}}
\newcommand\QQ{{\mathbb Q}}
\newcommand\CC{{\mathbb C}}
\newcommand\PP{{\mathbb P}}

\theoremstyle{plain}
\newtheorem{thm}{Theorem}[section]

\newtheorem{prop}[thm]{Proposition}
\newtheorem{lem}[thm]{Lemma}
\newtheorem{cor}[thm]{Corollary}
\newtheorem{ex}[thm]{Example}
\theoremstyle{definition}
\newtheorem{pty}{Property}
\theoremstyle{remark}
\newtheorem*{rmk*}{Remark}
\title[Schubert puzzles and integrability I]{Schubert puzzles and integrability I: \\ invariant trilinear forms}
\author{Allen Knutson}
\address{Allen Knutson, Cornell University, Ithaca, New York}
\email{allenk@math.cornell.edu}
\author{Paul Zinn-Justin}
\address{Paul Zinn-Justin, School of Mathematics and Statistics, The University of Melbourne, 
Victoria 3010, Australia}
\email{pzinn@unimelb.edu.au}
\thanks{PZJ was supported by ARC grant FT150100232.
}
\date{\today}

\newcommand\rem[2][]{}
\long\def\junk#1{}
\begin{document}

\begin{abstract}
  The {\em puzzle rules}\/ for computing Schubert calculus
  on $d$-step flag manifolds, proven in \cite{KT} for 
  $1$-step, in \cite{BKPT} for $2$-step, and 
  conjectured in \cite{CV} for $3$-step, 
  lead to vector configurations (one vector for each puzzle edge label) 
  that we recognize as the weights of some minuscule representations.
  The $R$-matrices of those representations (which, for $2$-step flag
  manifolds, involve triality of $D_4$) degenerate to give us puzzle
  formul\ae\ for two previously unsolved Schubert calculus problems:
  $K_T(2$-step flag manifolds$)$ and $K(3$-step flag manifolds$)$.
  The $K(3$-step
  flag manifolds$)$ formula, which involves 151 new puzzle pieces,
  implies Buch's correction to the first author's 1999 conjecture for
  $H^*(3$-step flag manifolds$)$. 
  \junk{We conclude with a (computer-based) proof that, under
  certain assumptions, there will be {\em no}\/ puzzle-based
  combinatorial formula for $H^*_T(3$-step$)$ nor for $H^*(4$-step$)$.}
\end{abstract}

\vspace*{-1cm}
\maketitle
\vspace*{-0.5cm}

\tableofcontents

\newcommand\defn[1]{{\bf #1}}
\newcommand\into\hookrightarrow
\newcommand\onto\twoheadrightarrow
\newcommand\otno\twoheadleftarrow
\newcommand\dom\backslash
\newcommand\tensor\otimes
\newcommand\iso\cong
\newcommand\lie[1]{{\mathfrak {#1}}}

\junk{For Allen to do:
  \begin{enumerate}
\junk{  \item \checkmark write up a proof of this single-number thing, for $d\leq 3$.
  \item \checkmark ? Try to process the monomial rule into something readable.
  \item \checkmark Discuss the $E_6$ rep, i.e. lemma~\ref{lem:greend3}.
  \item    \checkmark  Stop claiming manifest duality of $K_T(2$-step$)$
  \item    \checkmark Does the set of labels have a name?
  \item \checkmark Fix the unnumbered eq AFTER EQ \ref{eq:stdtri}!!!
  \item \checkmark? Clarify the example in 2.2 
  \item \checkmark rewrite theorem~\ref{thm:classif}
  \item \checkmark Fix citations (put them in biblio.bib)
  \item \checkmark Unify/standardize the numbering of Dynkin diagrams/irreps  
  \item \checkmark fix the factor of $2$ in $B$ compared to inversion number
\item \checkmark Have $A_d$ generated by $f_0,\ldots,f_d$
\item \checkmark Fix \S\ref{sec:finalproof}
\item \checkmark
  Fill in the part p27 about sum over reduced words after lemma
  \ref{lem:D4a}
\item \checkmark Define $B$ for general $d$
\item \checkmark fill the gaps in sect.~2.1, 2.2.
\item fix Theorem 4, or remove remark after it.
\item read the remaining comments and respond if necessary
}
\end{enumerate}
}

\section{Introduction}
In 1912 Weyl asked what the spectrum of a sum of two Hermitian
matrices could be, knowing the spectra of the matrices individually.
This received two solutions in the 1990s (we recommend the surveys
\cite{FultonSurvey,Zelevinsky}): a geometric one due to Klyachko (extending
work of Helmke--Rosenthal, Hersch--Zwahlen, and Totaro, among others) based on
Schubert calculus of Grassmannians, and a combinatorial one \cite{KTW-II}
in terms of certain triangular tilings called ``puzzles''. 
We take from this coincidence the oracular statement that 
{\em puzzles should be related to Schubert calculus;} hopefully not just for 
ordinary cohomology of Grassmannians, but also for more exotic cohomology
theories of \defn{$d$-step flag manifolds}
$Fl(n_1,\ldots,n_d;\, \CC^n) :=
\{ (0 \leq V_{1} \leq \cdots \leq V_{d} \leq \CC^n)\colon \dim V^i = n_i \}$.

Our general definition of ``Schubert calculus'' is a ring-with-basis problem,
the ring being a cohomology theory applied to a flag manifold, the basis
a tailor-made basis of ``Schubert classes'', and the problem being to
compute the coefficients when expanding a product of two basis elements
into the basis. Any such problem is solvable by giving a presentation of
the ring and locating the Schubert classes in that presentation.
However, it has been a recurrent phenomenon that Schubert structure constants
are {\em nonnegative} (in some sense appropriate to the cohomology theory;
see \cite{BuchKGr,Brion-KPos,Graham,AGM-Kpos}).
Many Schubert calculus problems have been solved nonnegatively with puzzles,
some in no other ways; see \cite{KT, Vakil,
  BKPT, KPur, BuchHT, PY, artic67, artic68}.

\subsection{Puzzles for the $d=1$ case, of Grassmannians}\label{sec:introd1}
We will index the Schubert classes $\{S^\lambda\}$ on $P_- \dom GL_n(\CC)$ 
(precise definitions given in \S\ref{sec:gen}) by strings $\lambda$
in $0,1,\ldots,d$, where the number of $j$s in the string is
$n_{j+1}-n_j$ (by convention $n_0 = 0, n_{1+d} = n$).
In particular those strings 
for the Grassmannian $Gr_k(\CC^n)$ of $k$-planes have%
\footnote{%
  One can get away with the dual convention $k\leftrightarrow n-k$ when dealing
  with only Grassmannians, but this becomes untenable when going to higher $d$.}
content $0^k 1^{n-k}$. We define three \defn{edge labels} $L_1 := \{0,1,10\}$, 
and three \defn{puzzle pieces}:
\begin{equation}\label{eq:d1tri}
  \uptri{0}{0}{0} \qquad \uptri{1}{1}{1} \qquad \uptri{1}{10}{0} 
  \qquad\text{plus all rotations, but not reflections}   
\end{equation}
A \defn{Grassmannian puzzle of size $n$} is a size $n$ triangle (oriented like $\Deltatri$)
decomposed into $n^2$ puzzle pieces, whose boundary labels are only $0$ or $1$,
not $10$. For example, here are the two puzzles whose NW and NE boundaries 
are both labeled $0,1,0,1$ left to right:

\begin{center}\def\ss{}
\begin{tikzpicture}[math mode,nodes={\mcol},x={(-0.577cm,-1cm)},y={(0.577cm,-1cm)},scale=1.5]
\draw[thick] (0,0) -- node[pos=0.5] {\ss 0} ++(0,1); \draw[thick] (0,0) -- node[pos=0.5] {\ss 1} ++(1,0); \draw[thick] (0+1,0) -- node {\ss 10} ++(-1,1); 
\draw[thick] (0,1) -- node[pos=0.5] {\ss 1} ++(0,1); \draw[thick] (0,1) -- node[pos=0.5] {\ss 1} ++(1,0); \draw[thick] (0+1,1) -- node {\ss 1} ++(-1,1); 
\draw[thick] (0,2) -- node[pos=0.5] {\ss 0} ++(0,1); \draw[thick] (0,2) -- node[pos=0.5] {\ss 0} ++(1,0); \draw[thick] (0+1,2) -- node {\ss 0} ++(-1,1); 
\draw[thick] (0,3) -- node[pos=0.5] {\ss 1} ++(0,1); \draw[thick] (0,3) -- node[pos=0.5] {\ss 10} ++(1,0); \draw[thick] (0+1,3) -- node {\ss 0} ++(-1,1); 
\draw[thick] (1,0) -- node[pos=0.5] {\ss 0} ++(0,1); \draw[thick] (1,0) -- node[pos=0.5] {\ss 0} ++(1,0); \draw[thick] (1+1,0) -- node {\ss 0} ++(-1,1); 
\draw[thick] (1,1) -- node[pos=0.5] {\ss 10} ++(0,1); \draw[thick] (1,1) -- node[pos=0.5] {\ss 0} ++(1,0); \draw[thick] (1+1,1) -- node {\ss 1} ++(-1,1); 
\draw[thick] (1,2) -- node[pos=0.5] {\ss 1} ++(0,1); \draw[thick] (1,2) -- node[pos=0.5] {\ss 1} ++(1,0); \draw[thick] (1+1,2) -- node {\ss 1} ++(-1,1); 
\draw[thick] (2,0) -- node[pos=0.5] {\ss 0} ++(0,1); \draw[thick] (2,0) -- node[pos=0.5] {\ss 1} ++(1,0); \draw[thick] (2+1,0) -- node {\ss 10} ++(-1,1); 
\draw[thick] (2,1) -- node[pos=0.5] {\ss 1} ++(0,1); \draw[thick] (2,1) -- node[pos=0.5] {\ss 1} ++(1,0); \draw[thick] (2+1,1) -- node {\ss 1} ++(-1,1); 
\draw[thick] (3,0) -- node[pos=0.5] {\ss 0} ++(0,1); \draw[thick] (3,0) -- node[pos=0.5] {\ss 0} ++(1,0); \draw[thick] (3+1,0) -- node {\ss 0} ++(-1,1); 
\end{tikzpicture}
\qquad
\begin{tikzpicture}[math mode,nodes={\mcol},x={(-0.577cm,-1cm)},y={(0.577cm,-1cm)},scale=1.5]
\draw[thick] (0,0) -- node[pos=0.5] {\ss 0} ++(0,1); \draw[thick] (0,0) -- node[pos=0.5] {\ss 1} ++(1,0); \draw[thick] (0+1,0) -- node {\ss 10} ++(-1,1); 
\draw[thick] (0,1) -- node[pos=0.5] {\ss 1} ++(0,1); \draw[thick] (0,1) -- node[pos=0.5] {\ss 1} ++(1,0); \draw[thick] (0+1,1) -- node {\ss 1} ++(-1,1); 
\draw[thick] (0,2) -- node[pos=0.5] {\ss 0} ++(0,1); \draw[thick] (0,2) -- node[pos=0.5] {\ss 1} ++(1,0); \draw[thick] (0+1,2) -- node {\ss 10} ++(-1,1); 
\draw[thick] (0,3) -- node[pos=0.5] {\ss 1} ++(0,1); \draw[thick] (0,3) -- node[pos=0.5] {\ss 1} ++(1,0); \draw[thick] (0+1,3) -- node {\ss 1} ++(-1,1); 
\draw[thick] (1,0) -- node[pos=0.5] {\ss 0} ++(0,1); \draw[thick] (1,0) -- node[pos=0.5] {\ss 0} ++(1,0); \draw[thick] (1+1,0) -- node {\ss 0} ++(-1,1); 
\draw[thick] (1,1) -- node[pos=0.5] {\ss 1} ++(0,1); \draw[thick] (1,1) -- node[pos=0.5] {\ss 10} ++(1,0); \draw[thick] (1+1,1) -- node {\ss 0} ++(-1,1); 
\draw[thick] (1,2) -- node[pos=0.5] {\ss 0} ++(0,1); \draw[thick] (1,2) -- node[pos=0.5] {\ss 0} ++(1,0); \draw[thick] (1+1,2) -- node {\ss 0} ++(-1,1); 
\draw[thick] (2,0) -- node[pos=0.5] {\ss 1} ++(0,1); \draw[thick] (2,0) -- node[pos=0.5] {\ss 1} ++(1,0); \draw[thick] (2+1,0) -- node {\ss 1} ++(-1,1); 
\draw[thick] (2,1) -- node[pos=0.5] {\ss 0} ++(0,1); \draw[thick] (2,1) -- node[pos=0.5] {\ss 0} ++(1,0); \draw[thick] (2+1,1) -- node {\ss 0} ++(-1,1); 
\draw[thick] (3,0) -- node[pos=0.5] {\ss 10} ++(0,1); \draw[thick] (3,0) -- node[pos=0.5] {\ss 0} ++(1,0); \draw[thick] (3+1,0) -- node {\ss 1} ++(-1,1); 
\end{tikzpicture}
\end{center}

The connection to \defn{Schubert calculus}, the study of the
product of Schubert classes $\{S^\lambda\}$, is simple \cite{KTW-II}: 
\begin{equation}
  \label{eq:LR}
  S^\lambda\, S^\mu \ =\  
  \sum_\nu \#\left(\text{puzzles labeled \tikz[scale=1.8,baseline=0.5cm]{\uptri{\lambda}{\nu}{\mu}}
      left to right (not clockwise) }\right) \ S^\nu 
\end{equation}

\subsection{Overview of the subject and this series of papers}

We remind the reader where the basis of $H^*(Gr_k(\CC^n))$ originates.
To specify a $k$-plane, we can list a basis of $k$ row vectors,
organized into a $k\times n$ matrix of full rank $k$. Since such bases
aren't unique we use row operations (left multiplication by $GL_k(\CC)$)
to put our matrix into row-reduced echelon form (RREF). The benefit is that
each $k$-plane now has a unique description; the drawback is that it
becomes hard to see the topology of the compact manifold $Gr_k(\CC^n)$.
Specifically, if we let
$$ C_\lambda := \text{image}\left(
  \left\{
    \text{RREF matrices with pivots in columns $\lambda$}
  \right\}
  \xrightarrow{\text{row span}}
  Gr_k(\CC^n)
\right)
$$
then $C_\lambda \iso \CC^{\#\{(i<j)\colon i\notin \lambda \ni j\}}$
and $Gr_k(\CC^n) = \coprod_{\lambda \in {[n]\choose k}} C_\lambda$.
(For technical reasons we will prefer to reverse our matrices left-right
before taking row span.)
This even-real-dimensional cell decomposition of $Gr_k(\CC^n)$,
the \defn{Bruhat decomposition}, also arises as an orbit decomposition
(by right-multiplication by upper triangular matrices) and as a
Morse/Bia\l ynicki-Birula decomposition. All of these descriptions
generalize beyond Grassmannians to $d$-step flag manifolds.

Because each cell is invariant under the action of the diagonal matrices
$T$ (acting on these $k\times n$ matrices by right multiplication),
the closures of the cells define a basis of the $T$-equivariant cohomology
$H^*_T(Gr_k(\CC^n))$, not just of ordinary cohomology. However, the
coefficient ring is now $H^*_T(pt) \iso \ZZ[y_1,\ldots,y_n]$, which is where
the multiplicative structure constants will live.
In \cite{KT} the puzzle rule of \S\ref{sec:introd1}
was extended to compute these polynomials,
using one new \defn{equivariant piece} $\rh{1}{0}{1}{0}$.
Since the structure constants to be computed are polynomials,
instead of {\em counting} puzzles (i.e. each contributing $1$)
we {\em sum over} puzzles, where each puzzle contributes a
product of roots $y_i-y_j$ based on the locations of the equivariant pieces.

For example, here are the computations of $S^{010}S^{100}$ and $S^{100}S^{010}$
(which are of course equal, as $H^*_T(Gr_k(\CC^n))$ is commutative): \\
{\def\posa{.5}\def\posb{.5}\def\thescale{1.5}
\begin{tikzpicture}[math mode,nodes={edgelabel},x={(-0.577cm,-1cm)},y={(0.577cm,-1cm)},scale=\thescale]
\draw[thick] (0,0) -- node[pos=\posa] {1} ++(0,1); \draw[thick] (0,0) -- node[pos=\posb] {0} ++(1,0); \node at (0+0.5,0+0.5) {y_1-y_3}; 
\draw[thick] (0,1) -- node[pos=\posa] {0} ++(0,1); \draw[thick] (0,1) -- node[pos=\posb] {0} ++(1,0); \draw[thick] (0+1,1) -- node {0} ++(-1,1); 
\draw[thick] (0,2) -- node[pos=\posa] {0} ++(0,1); \draw[thick] (0,2) -- node[pos=\posb] {0} ++(1,0); \draw[thick] (0+1,2) -- node {0} ++(-1,1); 
\draw[thick] (1,0) -- node[pos=\posa] {1} ++(0,1); \draw[thick] (1,0) -- node[pos=\posb] {1} ++(1,0); \draw[thick] (1+1,0) -- node {1} ++(-1,1); 
\draw[thick] (1,1) -- node[pos=\posa] {0} ++(0,1); \draw[thick] (1,1) -- node[pos=\posb] {0} ++(1,0); \draw[thick] (1+1,1) -- node {0} ++(-1,1); 
\draw[thick] (2,0) -- node[pos=\posa] {10} ++(0,1); \draw[thick] (2,0) -- node[pos=\posb] {0} ++(1,0); \draw[thick] (2+1,0) -- node {1} ++(-1,1); 
\end{tikzpicture}
\raisebox{2cm}{vs}
\begin{tikzpicture}[math mode,nodes={edgelabel},x={(-0.577cm,-1cm)},y={(0.577cm,-1cm)},scale=\thescale]
\draw[thick] (0,0) -- node[pos=\posa] {0} ++(0,1); \draw[thick] (0,0) -- node[pos=\posb] {0} ++(1,0); \draw[thick] (0+1,0) -- node {0} ++(-1,1); 
\draw[thick] (0,1) -- node[pos=\posa] {1} ++(0,1); \draw[thick] (0,1) -- node[pos=\posb] {10} ++(1,0); \draw[thick] (0+1,1) -- node {0} ++(-1,1); 
\draw[thick] (0,2) -- node[pos=\posa] {0} ++(0,1); \draw[thick] (0,2) -- node[pos=\posb] {0} ++(1,0); \draw[thick] (0+1,2) -- node {0} ++(-1,1); 
\draw[thick] (1,0) -- node[pos=\posa] {1} ++(0,1); \draw[thick] (1,0) -- node[pos=\posb] {0} ++(1,0); \node at (1+0.5,0+0.5) {y_1-y_2}; 
\draw[thick] (1,1) -- node[pos=\posa] {0} ++(0,1); \draw[thick] (1,1) -- node[pos=\posb] {0} ++(1,0); \draw[thick] (1+1,1) -- node {0} ++(-1,1); 
\draw[thick] (2,0) -- node[pos=\posa] {1} ++(0,1); \draw[thick] (2,0) -- node[pos=\posb] {1} ++(1,0); \draw[thick] (2+1,0) -- node {1} ++(-1,1); 
\end{tikzpicture}
\begin{tikzpicture}[math mode,nodes={edgelabel},x={(-0.577cm,-1cm)},y={(0.577cm,-1cm)},scale=\thescale]
\draw[thick] (0,0) -- node[pos=\posa] {0} ++(0,1); \draw[thick] (0,0) -- node[pos=\posb] {0} ++(1,0); \draw[thick] (0+1,0) -- node {0} ++(-1,1); 
\draw[thick] (0,1) -- node[pos=\posa] {1} ++(0,1); \draw[thick] (0,1) -- node[pos=\posb] {0} ++(1,0); \node at (0+0.5,1+0.5) {y_2-y_3}; 
\draw[thick] (0,2) -- node[pos=\posa] {0} ++(0,1); \draw[thick] (0,2) -- node[pos=\posb] {0} ++(1,0); \draw[thick] (0+1,2) -- node {0} ++(-1,1); 
\draw[thick] (1,0) -- node[pos=\posa] {0} ++(0,1); \draw[thick] (1,0) -- node[pos=\posb] {0} ++(1,0); \draw[thick] (1+1,0) -- node {0} ++(-1,1); 
\draw[thick] (1,1) -- node[pos=\posa] {1} ++(0,1); \draw[thick] (1,1) -- node[pos=\posb] {10} ++(1,0); \draw[thick] (1+1,1) -- node {0} ++(-1,1); 
\draw[thick] (2,0) -- node[pos=\posa] {1} ++(0,1); \draw[thick] (2,0) -- node[pos=\posb] {1} ++(1,0); \draw[thick] (2+1,0) -- node {1} ++(-1,1); 
\end{tikzpicture}
}

To go from ordinary to equivariant cohomology is but one of the many potential
axes of generalization. There are three others most relevant to the present
series of papers:
\begin{enumerate}
\item {\em $d$-step Schubert calculus,} i.e. on $d$-step flag
  manifolds, as defined above. We do not know any a priori reason that
  the complexity of the subject should increase with $d$, but at least
  for the approach taken in these papers, that has been the case.
\item {\em $K$-theoretic Schubert calculus,} which is sensitive to
  solution sets of all dimensions rather than just top dimension. For
  a simple example, let $L_1 \neq L_2$ be two lines in $\PP^3$ passing
  through a point $q$, hence lying in a common plane $P$. Then the
  space of lines $M$ touching both $L_1$ {\em and} $L_2$ has two
  components (each a $\PP^2$): those touching $q$ vs. those lying in
  $P$, and these two $\PP^2$s intersect in the $\PP^1$ worth of lines
  $M$ that satisfy both conditions. The cohomology class of this union
  is the sum of the classes of the loci from the two cases, but to compute the
  $K$-theory class one must also subtract off the class from the
  $\PP^1$ intersection (negligible in cohomology).
\item {\em Cotangent Schubert calculus,} in which the cycles involved
  are not those of Schubert varieties in flag manifolds but (much more
  complicated) cycles living in the cotangent bundles to the flag manifolds.
  In particular, calculations with these will involve an extra parameter
  coming from the equivariant cohomology of the dilation action
  on the cotangent fibers.
\end{enumerate}

Two other well-known axes of generalization, that we do not address at all,
are {\em quantum cohomology} (though it is deeply connected
in \cite{MO} to cotangent Schubert calculus) and going beyond type
$A$ flag manifolds to other Lie types, e.g. orthogonal flag manifolds
(though a small step in this direction appeared in \cite{artic73}).

It was quite unclear to the authors of \cite{KT} as to what the jigsaw puzzle
edge-matching requirement should have to do with, well, anything!
As such the main proofs were extremely ad hoc.
The situation was cleared up in \cite{artic46,artic68}: the usual rule for
matrix multiplication has a similar matching requirement,
$$ (AB)_{i\ell} = \sum \{A_{ij} B_{kl}\colon (j,k) \text{ s.t. } j = k \}
$$
In the present series of papers we generalize the puzzle rules for
$H^*(Gr_k(\CC^n)),H^*_T(Gr_k(\CC^n))$ essentially
by looking for the right matrices
to multiply. The main algebraic property they will satisfy is the
``Yang-Baxter equation'', and conveniently, solutions to this equation
(called ``$R$-matrices'') have been studied for several decades.
As the Yang-Baxter equation is traditionally used to show that some
large families of matrices commute, the field goes by the name
``quantum integrability'' or ``quantum integrable systems''
(``quantum'' for the vector spaces, ``integrable'' for the common eigenbasis
of the commuting matrices).
\begin{enumerate}
\item In this first paper, and especially in \S\ref{sec:proofs},
  we use algebraic properties of certain specific $R$-matrices
  (associated to certain representations bearing invariant
  trilinear forms) to show that they compute multiplication of some
  rather abstractly defined ``classes''. Those classes are {\em not}
  the Schubert classes; rather, we have to take certain limits to recover
  multiplication of Schubert classes. These limits are easy to handle
  for $d=1,2$ but become very tricky at $d=3$.

  One of the strengths of this framework is that the generalizations
  to $K$-theory come essentially for free. In particular, we {\em discover}
  and prove a puzzle rule for $K(Fl(n_1,n_2,n_3;\, \CC^n))$,
  involving 151 new puzzle pieces.
\item In the second paper, we observe that the classes most naturally
  connected to the $R$-matrices (i.e. without taking the limits)
  are those from cotangent Schubert calculus. This is fundamentally
  based on the work of \cite{MO-qg} realizing the representation theory
  using {\em Nakajima quiver varieties}, a class that includes the cotangent
  bundles of $d$-step flag manifolds. However our approach makes
  crucial use, in an intermediate step,
  of Nakajima quiver varieties that are {\em not} $d$-step flag manifolds.

  One side result, whose statement does not require direct reference to
  cotangent bundles (although the proof does),
  is a formula for the Euler characteristic of the
  intersection of three (transversely placed) Bruhat cells.
  This combinatorial result led us to conjecture a positivity statement
  for such Euler characteristics in general $G/P$ (our theorem only
  holding for $d\leq 3$-step type $A$ flag manifolds).
  This was eventually proven in \cite{SSW}.
\item In the third paper, we push the framework to its limits;
  as we expect combinatorial multiplication whenever we have a map from
  the tensor product of two ``minuscule'' representations to a third
  (generalizing the trilinear form picture from paper \#1),
  we study the few cases thereof systematically. The most basic is
  $\CC^n \tensor \CC^n \to Alt^2\, \CC^n$, which we use to discover and
  solve the case of \defn{separated-descent Schubert calculus}.
  This concerns the pullback map along the inclusion
  $Fl(1,\ldots,n;\, \CC^n)
  \into Fl(1,\ldots,k;\, \CC^n) \times Fl(k,\ldots,n;\, \CC^n)$
  of a full flag manifold into a certain product of partial flag manifolds.
  (Note the single overlap $k$ of subspace dimensions.)
  There is also an infinite family in type $D$, giving us
  \defn{almost separated-descent Schubert calculus} based on 
  $Fl(n_1,\ldots,n_d;\, \CC^n)
  \into Fl(n_1,\ldots,n_{k+1};\, \CC^n) \times Fl(n_k,\ldots,n_d;\, \CC^n)$,
  where the targets share two subspaces. Oddly, the second formula does
  not reduce to the first when $n_k = n_{k+1}$.

  Since our announcement of these rules, another paper \cite{Huang}
  has given an alternative formula for separated-descent Schubert calculus.
\end{enumerate}

\subsection{Studying puzzle boundaries with Green's theorem}

Since we cannot multiply cohomology classes living on different spaces,
it would be odd if the puzzle boundaries $\lambda,\mu,\nu$ 
were strings with different content (numbers of $0$s and $1$s). 
Here is one way to prove that that can't happen:
\newcommand\vf{\vec f}

\begin{lem}\label{lem:greend1}
  To each edge of a $\Deltatri$ triangle in a puzzle $P$, we assign a
  $2$-dimensional vector, rotation-equivariantly:
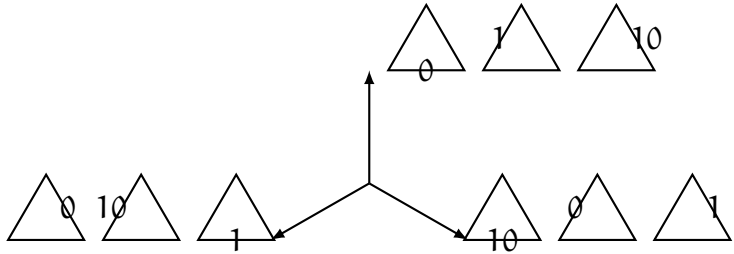
\begin{center}
\begin{tikzpicture}[>=latex]
\draw[thick,->] (0,0) -- (90:1.5);
\draw[thick,->] (0,0) -- (210:1.5);
\draw[thick,->] (0,0) -- (-30:1.5);
\begin{scope}[shift={(-1.25,1.7)}]\uptri{}{0}{}\end{scope}
\begin{scope}[shift={(0,1.7)}]   \uptri{1}{}{}\end{scope}
\begin{scope}[shift={(1.25,1.7)}]\uptri{}{}{10}\end{scope}
\begin{scope}[shift={(2,-1)}]\uptri{}{10}{}\end{scope}
\begin{scope}[shift={(3.25,-1)}]   \uptri{0}{}{}\end{scope}
\begin{scope}[shift={(4.5,-1)}]\uptri{}{}{1}\end{scope}
\begin{scope}[shift={(-2,-1)}]\uptri{}{1}{}\end{scope}
\begin{scope}[shift={(-3.25,-1)}]   \uptri{10}{}{}\end{scope}
\begin{scope}[shift={(-4.5,-1)}]\uptri{}{}{0}\end{scope}
\end{tikzpicture}
\end{center}
The boundary of $P$ has vector sum
  $\vec 0$, which forces the content (the numbers of $0$s and $1$s) on
  the three sides to be the same.
\end{lem}

\begin{proof}
  Every $\Deltatri$ puzzle piece has vector sum $\vec 0$:
  
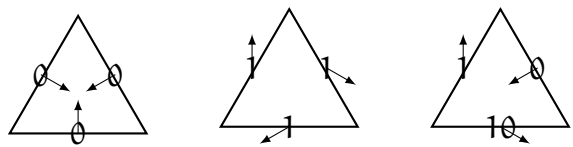
\begin{center}
    \tikz[scale=1.8,>=latex]{
      \uptri{0}{0}{0}
      \draw[->] (horiz.center) -- ++(90:0.25);
      \draw[->] (NE.center) -- ++(210:0.25);
      \draw[->] (NW.center) -- ++(-30:0.25);
    }
    \qquad
    \tikz[scale=1.8,>=latex]{
      \uptri{1}{1}{1}
      \draw[->] (NW.center) -- ++(90:0.25);
      \draw[->] (horiz.center) -- ++(210:0.25);
      \draw[->] (NE.center) -- ++(-30:0.25);
    }
    \qquad
    \tikz[scale=1.8,>=latex]{
      \uptri{1}{10}{0}
      \draw[->] (NW.center) -- ++(90:0.25);
      \draw[->] (horiz.center) -- ++(-30:0.25);
      \draw[->] (NE.center) -- ++(210:0.25);
    }
  \end{center}
  To edges of $\nablatri$ triangles, we assign the negatives of these
  vectors, obtaining a cancelation at each internal edge of $P$.
  Those $\nablatri$ pieces also have vector sum $\vec 0$.
  Then by Green's theorem, the boundary of $P$ has vector sum $\vec 0$.
\end{proof}

\rem{AK adds:}
(If one allows $10$s on the South side of a puzzle but not the other two,
then this argument shows that the number of such $10$s is the number of $1$s
on the Northwest side minus the number of $1$s on the Northeast side.
Such puzzles are given a cohomological interpretation in \cite{artic73}.)

The original result \cite{KTW-II} about ordinary cohomology of
Grassmannians has been generalized (for Grassmannians) to equivariant
cohomology \cite{KT,artic46}, to $K$-theory \cite{Vakil}, and recently to
equivariant $K$-theory \cite{PY,artic68}, in each case adding some pieces. 
Our methods in this paper are closest to those of \cite{artic68},
which recognizes the three vectors above as the
{\em weights\junk{We warn the reader that
    this leads to an unfortunate collision later:
    puzzle labels get weights in one weight lattice, 
    but puzzle pieces are also given weights, in a different weight lattice,
    for an entirely different reason.}
 of the standard representation} $V^{A_2}_{\omega_1}$ of $SL_3(\CC)$
and makes use of the corresponding $R$-matrix and the
Yang--Baxter equation it satisfies.
However, our proof in \S\ref{sec:proofs} is cleaner 
than that of \cite{artic68} even in the $d=1$ case,
through exploitation of the bootstrap equation of the relevant $R$-matrix.

There is another way to label the edges of a $\Deltatri$ and still
get vector sum $\vec 0$: $\uptri{10}{10}{10}$. 
\junk{
  If we take a puzzle piece and rotate the vectors {\em in place,} not
  rotating the triangle, we get
  \[ \uptri{1}{10}{0} \to \uptri{10}{0}{1} \to \uptri{0}{1}{10} \to 
  \uptri{1}{10}{0} \qquad \qquad
  \uptri{0}{0}{0} \to\uptri{1}{1}{1} \to\uptri{10}{10}{10} \to\uptri{0}{0}{0} 
  \]
  suggesting that a new $\uptri{10}{10}{10}$ piece has a role to play.
}
And indeed, if we add this piece {\em but only in the $\Deltatri$ orientation}%
\footnote{If one includes both orientations, even with independent
  weightings, the resulting product rule
  is again commutative associative. We will address this
  in a future publication \cite{artic80}.}
we get a rule for the $K$-theory product, with respect to the
basis of structure sheaves of Schubert varieties \cite[theorem~4.6]{Vakil};
whereas if we add this piece {\em but only in the $\nablatri$ orientation}
we get a rule for the $K$-theory product, with respect to the dual basis 
(the ideal sheaves of boundaries of Schubert varieties) \cite{artic68}.

To be precise, equation (\ref{eq:LR}) is modified in this $K$-theory
product, in that each $K$-piece carries a \defn{fugacity}\footnote{%
  In prior publications it seemed natural to call this a ``weight'',
  but that might cause confusion with the ``weights'' we assign
  to edge labels, and ``fugacity'' is a standard term 
  in the physics literature for such local contributions.}
of $-1$,
and the \defn{fugacity of a puzzle} is defined as the product of the fugacities 
of its pieces. Then we sum the fugacities of the puzzles, 
rather than simply counting them. 
Luckily, in this formula all the puzzle fugacities contributing to a given
coefficient have the same sign so there is no cancelation.

An important difference between $K(Gr_k(\CC^n))$ and $H^*(Gr_k(\CC^n))$ 
is that the latter is a graded ring.
Equation (\ref{eq:LR}) for $H^*$ then implies that 
$\ell(\lambda)+\ell(\mu)=\ell(\nu)$, where 
\[
 \ell(\kappa) \  := \  \#\big\{ i<j\ :\ \kappa_i > \kappa_j \big\} \  
= \frac{1}{2} \deg S^\kappa
\]
is the number of \defn{inversions} of the string $\kappa$.

In \S\ref{ssec:Bmatrix} we define the {\em inversion number} of a
path through a puzzle, and in particular, the inversion number 
of (a clockwise path around) a puzzle piece. If two paths $\gamma$,
$\gamma'$ through a puzzle have the same start and end points,
then the difference of their inversion numbers is proven in
\S\ref{ssec:Bmatrix} to be the sum of the inversion numbers of 
the pieces in between $\gamma$ and $\gamma'$.

It is straightforward to calculate the inversion numbers of puzzle pieces;
the ones in \eqref{eq:d1tri} have vanishing inversion number, whereas
the $K$-pieces have inversion number $1$.
When drawing a piece we'll put its inversion number in the center, e.g. 
$\uptri[1]{10}{10}{10}$, omitting that only when the inversion number is zero.

\subsection{Puzzles for $d=2$ flag manifolds}
\label{ssec:d2intro}
In 1999 the first author trusted the oracle and looked for puzzles to
compute Schubert calculus of $2$-step flag manifolds, where the 
Schubert classes (hence the boundaries of the puzzles) are strings in $0,1,2$. 
Already with puzzles of edge-length $3$, 
one discovers the following eight labels $L_2$ and puzzle pieces:
\begin{equation}\label{eq:d2tri}
\uptri{0}{0}{0} \quad \uptri{1}{1}{1} \quad \uptri{2}{2}{2} 
\qquad \uptri{1}{10}{0} \quad \uptri{2}{20}{0} \quad \uptri{2}{21}{1} 
\qquad \uptri{21}{(21)0}{0} \quad \uptri{2}{2(10)}{10}
\end{equation}
(again permitting rotations but not reflections). 
We will call parenthesized expressions $X$ such as these
(really encoding planar binary trees with labeled leaves) \defn{multinumbers}, 
and call a multinumber $X$ \defn{valid} if it appears in a puzzle rule 
(i.e., on one of the puzzle piece edges here in \eqref{eq:d2tri}, 
or later in \S\ref{ssec:d3puz}).
We emphasize that this is a {\em historical} definition and not a 
mathematical one, with a mathematical retrodiction to come in
proposition \ref{prop:valid}.
Write $|X|$ for the number of digits in $X$ (i.e. $|X|=1$ exactly
for ``single-numbers''). 
Proving the analogue of equation (\ref{eq:LR}) for $d=2$,
and its generalization to equivariant cohomology, 
took until 2014 \cite{BKPT,BuchHT}.

The $d=2$ analogue of lemma~\ref{lem:greend1} is an assignment with
values in $D_4$'s weight lattice of row vectors
$\{(a_1,a_2,a_3,a_4)\ :\ 2a_i\in \ZZ,\ a_i-a_j \in \ZZ\}$.  
We regard the (counterclockwise) $120^\circ$ rotation $\tau$ as acting 
on this lattice on the right by
\[
  \qquad\qquad
\tau = \frac{\ 1\ }{2} 
\begin{bmatrix} -&-&-&- \\ +&-&-&+ \\ +&+&-&- \\ +&-&+&- \end{bmatrix}
\qquad\qquad \text{where $+=+1$, $-=-1$}
\]
and mention that the corresponding automorphism of the Lie algebra 
$D_4$ is outer; it is the triality automorphism.

\begin{lem}\label{lem:greend2}
  To each of the NW labels $L_2$ in the puzzle pieces above, we assign a 
  weight in the $D_4$ weight lattice
  \begin{align*}
    \m2&\mapsto(+,0,0,0)     &\m1&\mapsto(0,+,0,0)&\m0&\mapsto(0,0,+,0)&\m{10}&\mapsto(0,0,0,+)\\
    \m{(21)0}&\mapsto(-,0,0,0)&\m{20}&\mapsto(0,-,0,0)&\m{2(10)}&\mapsto(0,0,-,0)&\m{2 1}&\mapsto(0,0,0,-)
  \end{align*}
  and assign $D_4$ weights to labels on other sides in a
  rotation-equivariant way, using the $4\times 4$ matrix above.

  Then every puzzle piece has vector sum $\vec 0$. By Green's theorem
  the boundary of the whole puzzle also has vector sum $\vec 0$,
  which forces the content 
  (the numbers of $0$s, $1$s, and $2$s) on the three sides to be the same.
\end{lem}

The proof, other than immediate calculation, is exactly as in lemma 
\ref{lem:greend1}. Here is an example of the vectors adding to $\vec 0$,
on the $\uptri{21}{(21)0)}{0}$ puzzle piece:
$$
\begin{array}{ccccl}
  21 \mapsto \begin{pmatrix} 0&0&0&-1 \end{pmatrix} 
&&
  0 &\mapsto& \begin{pmatrix} 0&0&1&0 \end{pmatrix} \tau^2 
\\
  & \Deltatri &&=&  \begin{pmatrix} -1&-1&-1&1 \end{pmatrix}/2 \\
  & 
    \begin{array}{rcl}
      (21)0 &\mapsto&  \begin{pmatrix} -1&0&0&0 \end{pmatrix} \tau \\      
      &   = & \begin{pmatrix} 1&1&1&1 \end{pmatrix}/2   \\
    \end{array} \\
\end{array}
$$

The eight vectors in lemma \ref{lem:greend2}
are the weights of the standard representation 
$V^{D_4}_{\omega_1}$ of $\lie{so}(8)$, 
whose $\ZZ/3$ rotations are the two spin
representations $\tau\cdot V^{D_4}_{\omega_1} \iso V^{D_4}_{\omega_2}$ and
$\tau^2\cdot V^{D_4}_{\omega_1} \iso V^{D_4}_{\omega_3}$.

At this point, the assignments in lemmas~\ref{lem:greend1} and
\ref{lem:greend2} may look quite magical.  In \S\ref{sec:gen} we will
explain a recipe that produces them without effort, at least up to $d\leq 4$.

With the clue provided by lemma~\ref{lem:greend2}, the methods of \cite{artic68}
generalize straightforwardly to give our first major theorem of this paper.
To state it, we first recall

\begin{thm}\label{thm:HTd2}\cite{BuchHT}\cite[for $d=1$]{KT}\label{thm:Buch}
  To compute Schubert calculus in $T$-equivariant cohomology of a 
  $2$-step flag manifold, whose coefficients live in
  $H^*_T(pt) := \ZZ[y_1,\ldots,y_n]$,
  one needs the following \defn{equivariant rhombi}:
  \[
\eqrh{1}{0}\quad\eqrh{2}{0}\quad\eqrh{2}{1}\quad\eqrh{21}{0}\quad
  \eqrh{2}{10}\quad\eqrh{21}{10}\quad\eqrh{2(10)}{0}\quad\eqrh{2}{(21)0}
\]
  These may not be rotated. The \defn{fugacity of an equivariant rhombus}
  is $y_j-y_i \in H^*_T(pt)$, where lines drawn Southwest and Southeast
  from the rhombus hit the bottom edge at positions $i<j$.
  The \defn{fugacity $fug(P)$ of a puzzle $P$} 
  is the product of the fugacities of its rhombi. Then
  \begin{equation}
    \label{eq:LRT}
    S^\lambda\, S^\mu \ =\  
    \sum_\nu \left( \sum_{P\text{ with boundary }\lambda,\mu,\nu} fug(P) \right)
    \ S^\nu 
  \end{equation}
  where as usual the summation is over puzzles $P$ of the form
  $\tikz[scale=1.8,baseline=0.5cm]{\uptri{\lambda}{\nu}{\mu}}$ and all
  strings are read left to right.

  This formula~\eqref{eq:LRT} is manifestly positive in the sense of
  \cite{Graham}.
\end{thm}

The \textcircled{${\ss -1}$}s 
in these pieces are to indicate that they, too, spoil 
lemma~\ref{lem:sympd1}'s inversion count, 
but in the opposite way from $K$-theory; more detail 
will come in \S\ref{ssec:Bmatrix}.

We come to our first major new result of the paper, the
extension of the above to $T$-equivariant $K$-theory,
for which there was no known positive formula (in the
sense of \cite{AGM-Kpos}) or even extant conjecture
for the structure constants.
These live in 
$K_T(pt)\cong \ZZ[u_1^{\pm},\ldots,u_n^{\pm}]$.

\begin{thm}\label{thm:KTd2}
  Equation \eqref{eq:LRT} generalizes
  to compute Schubert calculus in equivariant $K$-theory of a 
  $2$-step flag manifold. One needs the equivariant rhombi
  from theorem~\ref{thm:Buch}, whose fugacities are now $1-u_j/u_i$,
  and also the following \defn{$K$-pieces}:
  \begin{gather*}
    \uptri[1]{10}{10}{10}\qquad\qquad \uptri[1]{20}{20}{20}\qquad\qquad \uptri[1]{21}{21}{21}
    \\
    \uptri[1]{2(10)}{20}{21}\quad 
    \uptri[1]{21}{2(10)}{20}\quad \uptri[1]{20}{21}{2(10)}
    \\
    \uptri[1]{(21)0}{10}{20}\quad     \downtri[1]{20}{(21)0}{10}\quad
        \uptri[1]{10}{20}{(21)0}
    \qquad    \qquad
    \downtri[1]{1}{2(10)}{(21)0}\quad \uptri[1]{(21)0}{1}{2(10)}\quad
    \downtri[1]{2(10)}{(21)0}{1}
    \\
    \downtri[2]{(21)0}{(21)0}{(21)0}
  \end{gather*}
  (no other rotations!) each with fugacity $-1$, or $(-1)^2$ for the last one.

  \junk{below needs more gluing, e.g. discussion of vertical rhombi...  
    Both usual rhombi and $K$-rhombi also acquire an extra monomial fugacity
    (which does not factor as a product over triangles). The power of
    $u_j/u_i$ ($0$ or $1$) in the fugacity can be obtained by the
    following ad hoc rule:}
  Some vertical rhombi $\rh{X}{Y}{Z}{W}$
  also acquire an extra monomial fugacity $u_j/u_i$
  (which does not factor as a product over triangles):  
  a rhombus gets a $u_j/u_i$ if $|W|+|Z| > |X|+|Y|$, 
  or (in the case $|W|+|Z| = |X|+|Y|$) if equal to
  
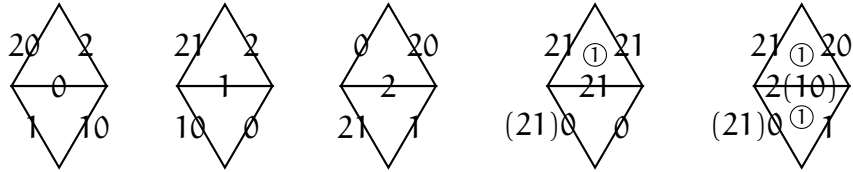
\[
    \tikz[baseline=0,scale=1.25]{\uptri{20}{0}{2}\downtri{10}{}{1}}
    \qquad
    \tikz[baseline=0,scale=1.25]{\uptri{21}{1}{2}\downtri{0}{}{10}}
    \qquad
    \tikz[baseline=0,scale=1.25]{\uptri{0}{2}{20}\downtri{1}{}{21}}
    \qquad
    \tikz[baseline=0,scale=1.25]{\uptri[1]{21}{21}{21}\downtri{0}{}{(21)0}}
    \qquad
    \tikz[baseline=0,scale=1.25]{\uptri[1]{21}{2(10)}{20}\downtri[1]{1}{}{(21)0}}
  \]
  This formula for the Schubert structure constants is manifestly positive in
  the sense of \cite{AGM-Kpos}.
\end{thm}

\junk{if the sum of lengths of top
  multi-numbers is greater than that of bottom multi-numbers -- in
  case of equality, if it satisfies $\max(\text{NE})>\max(\text{SE})$.}

In particular, rotation-invariant rhombi (such as equivariant rhombi)
get no $u_j/u_i$, and among a rhombus and its
$180^\circ$ rotation, assuming they're distinct and valid (i.e., not
involving unrotatable $K$-pieces), exactly one gets a $u_j/u_i$ factor.

\begin{ex}\def\ss{}
$c^{0102,0201}_{0210}$ is computed from one puzzle
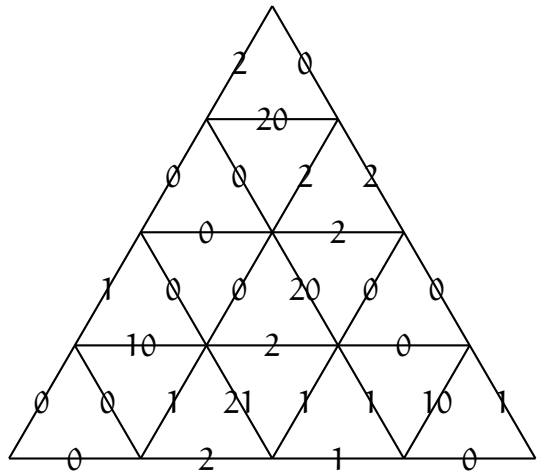
\begin{center}
\begin{tikzpicture}[math mode,nodes={\mcol},x={(-0.577cm,-1cm)},y={(0.577cm,-1cm)},scale=1.5]
\draw[thick] (0,0) -- node[pos=0.5] {\ss 0} ++(0,1); \draw[thick] (0,0) -- node[pos=0.5] {\ss 2} ++(1,0); \draw[thick] (0+1,0) -- node {\ss 20} ++(-1,1); 
\draw[thick] (0,1) -- node[pos=0.5] {\ss 2} ++(0,1); \draw[thick] (0,1) -- node[pos=0.5] {\ss 2} ++(1,0); \draw[thick] (0+1,1) -- node {\ss 2} ++(-1,1); 
\draw[thick] (0,2) -- node[pos=0.5] {\ss 0} ++(0,1); \draw[thick] (0,2) -- node[pos=0.5] {\ss 0} ++(1,0); \draw[thick] (0+1,2) -- node {\ss 0} ++(-1,1); 
\draw[thick] (0,3) -- node[pos=0.5] {\ss 1} ++(0,1); \draw[thick] (0,3) -- node[pos=0.5] {\ss 10} ++(1,0); \draw[thick] (0+1,3) -- node {\ss 0} ++(-1,1); 
\draw[thick] (1,0) -- node[pos=0.5] {\ss 0} ++(0,1); \draw[thick] (1,0) -- node[pos=0.5] {\ss 0} ++(1,0); \draw[thick] (1+1,0) -- node {\ss 0} ++(-1,1); 
\draw[thick] (1,1) -- node[pos=0.5] {\ss 20} ++(0,1); \draw[thick] (1,1) -- node[pos=0.5] {\ss 0} ++(1,0); \draw[thick] (1+1,1) -- node {\ss 2} ++(-1,1); 
\draw[thick] (1,2) -- node[pos=0.5] {\ss 1} ++(0,1); \draw[thick] (1,2) -- node[pos=0.5] {\ss 1} ++(1,0); \draw[thick] (1+1,2) -- node {\ss 1} ++(-1,1); 
\draw[thick] (2,0) -- node[pos=0.5] {\ss 0} ++(0,1); \draw[thick] (2,0) -- node[pos=0.5] {\ss 1} ++(1,0); \draw[thick] (2+1,0) -- node {\ss 10} ++(-1,1); 
\draw[thick] (2,1) -- node[pos=0.5] {\ss 21} ++(0,1); \draw[thick] (2,1) -- node[pos=0.5] {\ss 1} ++(1,0); \draw[thick] (2+1,1) -- node {\ss 2} ++(-1,1); 
\draw[thick] (3,0) -- node[pos=0.5] {\ss 0} ++(0,1); \draw[thick] (3,0) -- node[pos=0.5] {\ss 0} ++(1,0); \draw[thick] (3+1,0) -- node {\ss 0} ++(-1,1); 
\end{tikzpicture}
\end{center}
whose fugacity is $u_3/u_2$, whereas the opposite multiplication
$c^{0201,0102}_{0210}$ comes from two puzzles
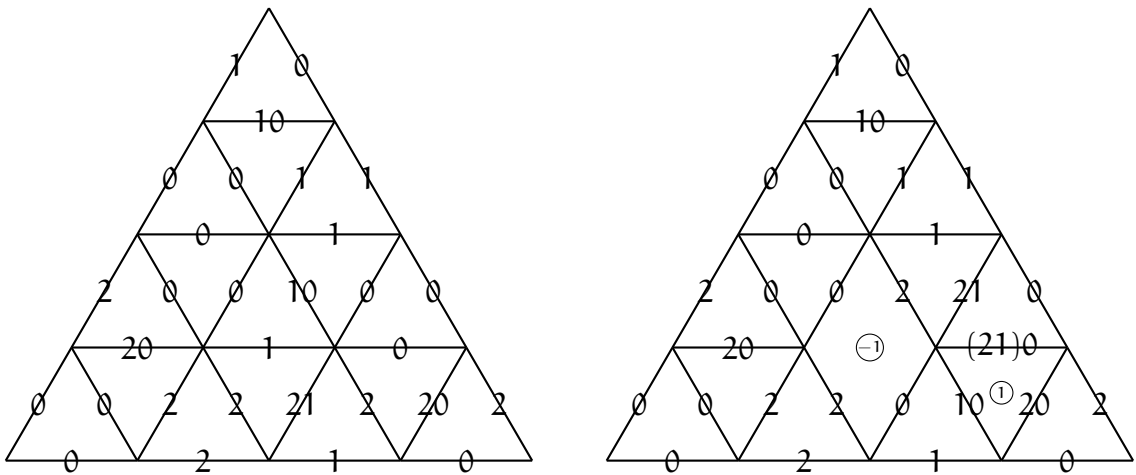
\begin{center}
\begin{tikzpicture}[math mode,nodes={\mcol},x={(-0.577cm,-1cm)},y={(0.577cm,-1cm)},scale=1.5]
\draw[thick] (0,0) -- node[pos=0.5] {\ss 0} ++(0,1); \draw[thick] (0,0) -- node[pos=0.5] {\ss 1} ++(1,0); \draw[thick] (0+1,0) -- node {\ss 10} ++(-1,1); 
\draw[thick] (0,1) -- node[pos=0.5] {\ss 1} ++(0,1); \draw[thick] (0,1) -- node[pos=0.5] {\ss 1} ++(1,0); \draw[thick] (0+1,1) -- node {\ss 1} ++(-1,1); 
\draw[thick] (0,2) -- node[pos=0.5] {\ss 0} ++(0,1); \draw[thick] (0,2) -- node[pos=0.5] {\ss 0} ++(1,0); \draw[thick] (0+1,2) -- node {\ss 0} ++(-1,1); 
\draw[thick] (0,3) -- node[pos=0.5] {\ss 2} ++(0,1); \draw[thick] (0,3) -- node[pos=0.5] {\ss 20} ++(1,0); \draw[thick] (0+1,3) -- node {\ss 0} ++(-1,1); 
\draw[thick] (1,0) -- node[pos=0.5] {\ss 0} ++(0,1); \draw[thick] (1,0) -- node[pos=0.5] {\ss 0} ++(1,0); \draw[thick] (1+1,0) -- node {\ss 0} ++(-1,1); 
\draw[thick] (1,1) -- node[pos=0.5] {\ss 10} ++(0,1); \draw[thick] (1,1) -- node[pos=0.5] {\ss 0} ++(1,0); \draw[thick] (1+1,1) -- node {\ss 1} ++(-1,1); 
\draw[thick] (1,2) -- node[pos=0.5] {\ss 2} ++(0,1); \draw[thick] (1,2) -- node[pos=0.5] {\ss 21} ++(1,0); \draw[thick] (1+1,2) -- node {\ss 1} ++(-1,1); 
\draw[thick] (2,0) -- node[pos=0.5] {\ss 0} ++(0,1); \draw[thick] (2,0) -- node[pos=0.5] {\ss 2} ++(1,0); \draw[thick] (2+1,0) -- node {\ss 20} ++(-1,1); 
\draw[thick] (2,1) -- node[pos=0.5] {\ss 2} ++(0,1); \draw[thick] (2,1) -- node[pos=0.5] {\ss 2} ++(1,0); \draw[thick] (2+1,1) -- node {\ss 2} ++(-1,1); 
\draw[thick] (3,0) -- node[pos=0.5] {\ss 0} ++(0,1); \draw[thick] (3,0) -- node[pos=0.5] {\ss 0} ++(1,0); \draw[thick] (3+1,0) -- node {\ss 0} ++(-1,1); 
\end{tikzpicture}
\qquad
\begin{tikzpicture}[math mode,nodes={\mcol},x={(-0.577cm,-1cm)},y={(0.577cm,-1cm)},scale=1.5]
\draw[thick] (0,0) -- node[pos=0.5] {\ss 0} ++(0,1); \draw[thick] (0,0) -- node[pos=0.5] {\ss 1} ++(1,0); \draw[thick] (0+1,0) -- node {\ss 10} ++(-1,1); 
\draw[thick] (0,1) -- node[pos=0.5] {\ss 1} ++(0,1); \draw[thick] (0,1) -- node[pos=0.5] {\ss 1} ++(1,0); \draw[thick] (0+1,1) -- node {\ss 1} ++(-1,1); 
\draw[thick] (0,2) -- node[pos=0.5] {\ss 0} ++(0,1); \draw[thick] (0,2) -- node[pos=0.5] {\ss 21} ++(1,0); \draw[thick] (0+1,2) -- node {\ss (21)0} ++(-1,1); 
\node[invlabel,text=black] at (0+0.7,2+0.7) {\sss 1};
\draw[thick] (0,3) -- node[pos=0.5] {\ss 2} ++(0,1); \draw[thick] (0,3) -- node[pos=0.5] {\ss 20} ++(1,0); \draw[thick] (0+1,3) -- node {\ss 0} ++(-1,1); 
\draw[thick] (1,0) -- node[pos=0.5] {\ss 0} ++(0,1); \draw[thick] (1,0) -- node[pos=0.5] {\ss 0} ++(1,0); \draw[thick] (1+1,0) -- node {\ss 0} ++(-1,1); 
\draw[thick] (1,1) -- node[pos=0.5] {\ss 2} ++(0,1); \draw[thick] (1,1) -- node[pos=0.5] {\ss 0} ++(1,0); 
\node[invlabel,text=black] at (1+0.5,1+0.5) {\sss -1};
\draw[thick] (1,2) -- node[pos=0.5] {\ss 10} ++(0,1); \draw[thick] (1,2) -- node[pos=0.5] {\ss 0} ++(1,0); \draw[thick] (1+1,2) -- node {\ss 1} ++(-1,1); 
\draw[thick] (2,0) -- node[pos=0.5] {\ss 0} ++(0,1); \draw[thick] (2,0) -- node[pos=0.5] {\ss 2} ++(1,0); \draw[thick] (2+1,0) -- node {\ss 20} ++(-1,1); 
\draw[thick] (2,1) -- node[pos=0.5] {\ss 2} ++(0,1); \draw[thick] (2,1) -- node[pos=0.5] {\ss 2} ++(1,0); \draw[thick] (2+1,1) -- node {\ss 2} ++(-1,1); 
\draw[thick] (3,0) -- node[pos=0.5] {\ss 0} ++(0,1); \draw[thick] (3,0) -- node[pos=0.5] {\ss 0} ++(1,0); \draw[thick] (3+1,0) -- node {\ss 0} ++(-1,1); 
\end{tikzpicture}
\end{center}
whose fugacities are $1$ and $(-1)(1-u_3/u_2)$, respectively.
\end{ex}

Theorem~\ref{thm:HTd2} is recovered by writing $u_i=1-y_i$ and
expanding at first nontrivial order in the $y_i$.

Define \defn{dual Schubert classes} $\{S_\lambda\}$ by 
$\left< S^\lambda S_\mu\right>=\delta^\lambda_\mu$,
where $\langle \rangle$ is the $K$-theoretic analogue of integration,
definable by $\langle S^\lambda \rangle = 1\ \forall \lambda$.
Then theorem~\ref{thm:KTd2} also provides a dual puzzle rule
\[
S_\lambda\, S_\mu \ =\  
    \sum_\nu \left( \sum_{P\text{ with boundary }\lambda,\mu,\nu} fug(P) \right)
    \ S_\nu 
  \]
where now the puzzles $P$ are of the form $\tikz[scale=1.8,baseline=-0.8cm]{\downtri{\lambda}{\nu}{\mu}}$
and boundary labels are again read left-to-right.

For future reference, we comment that the
tensor square of $V^{A_2}_{\omega_1} \iso \tau\cdot V^{A_2}_{\omega_1}$ is
\begin{align*}
(V^{A_2}_{\omega_1})^{\tensor 2} 
&\iso V^{A_2}_{2\omega_1} \oplus (V^{A_2}_{\omega_1})^*
\\
\intertext{and similarly, the tensor product of two of $D_4$'s minuscule 
fundamental representations is}
V^{D_4}_{\omega_1} \tensor V^{D_4}_{\omega_2}
&\iso V^{D_4}_{\omega_1+\omega_2} \oplus (V^{D_4}_{\omega_3})^* 
&&\text{(though note, $(V^{D_4}_{\omega_3})^* \iso V^{D_4}_{\omega_3}$)}
\end{align*}
We get away with tensoring $V^{A_2}_{\omega_1}$ with ``itself'' because
the action of $\ZZ/3$ on the $A_2$ lattice is by an inner automorphism,
but with $D_4$ we must face up to the three different representations
associated with the three edge orientations.

Our rules for $K_T(Gr_k(\CC^n))$ and $H^*_T(2$-step$)$ have a 
duality symmetry: one may reflect a puzzle left-right while 
replacing each number $i$ with $d-i$. This $K_T(2$-step$)$ rule does not,
directly, enjoy this symmetry (even though the structure constants do); 
rather this gives us a second puzzle rule for $K_T(2$-step$)$. 

\subsection{Puzzles for $d=3$ flag manifolds}\label{ssec:d3puz}
The $d=3$ labels and pieces turn out to be (up to rotation)
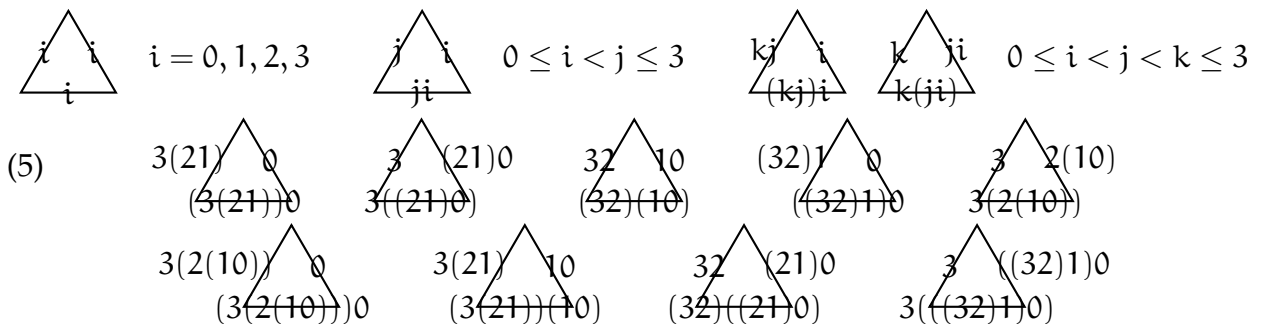
\begin{gather}\notag
  \uptri{i}{i}{i} \quad i=0,1,2,3 \qquad \uptri{j}{ji}{i} \quad 0\leq i<j \leq 3
  \qquad
\uptri{kj}{(kj)i}{i}\quad \uptri{k}{k(ji)}{ji} \quad 0\leq i<j<k \leq 3
\\\label{eq:d3tri}
\uptri{3(21)}{(3(21))0}{0}\qquad \uptri{3}{3((21)0)}{(21)0}
\qquad \uptri{32}{(32)(10)}{10}
\qquad \uptri{(32)1}{((32)1)0}{0}\qquad \uptri{3}{3(2(10))}{2(10)}
\\\notag
 \uptri{3(2(10))}{(3(2(10)))0}{0} \qquad
\uptri{3(21)}{(3(21))(10)}{10} \qquad
\uptri{32}{(32)((21)0)}{(21)0} \qquad
\uptri{3}{3(((32)1)0)}{((32)1)0}
\end{gather}
where the bottom four are Buch's correction \cite{CV} to the first author's
incomplete 1999 conjecture. 

\begin{lem}\label{lem:greend3}
  There is an action of $\ZZ/3$ on $E_6$'s weight lattice, 
  and a map from the $27$ edge labels above to the weights of the
  standard representation $V^{E_6}_{\omega_1}$ of $E_6$ 
  (then extended to other edge orientations $\ZZ/3$-equivariantly), 
  such that each puzzle piece has vector sum $\vec 0$.
  By Green's theorem
  the boundary of the whole puzzle does too, which forces the content 
  (the numbers of $0$s, $1$s, $2$s, and $3$s) on the three sides to be the same.
\end{lem}

Rather than writing out the vectors, we give the correspondence with 
the crystal of the $27$-dimensional representation, whose edges are labeled
by the simple roots of $E_6$:
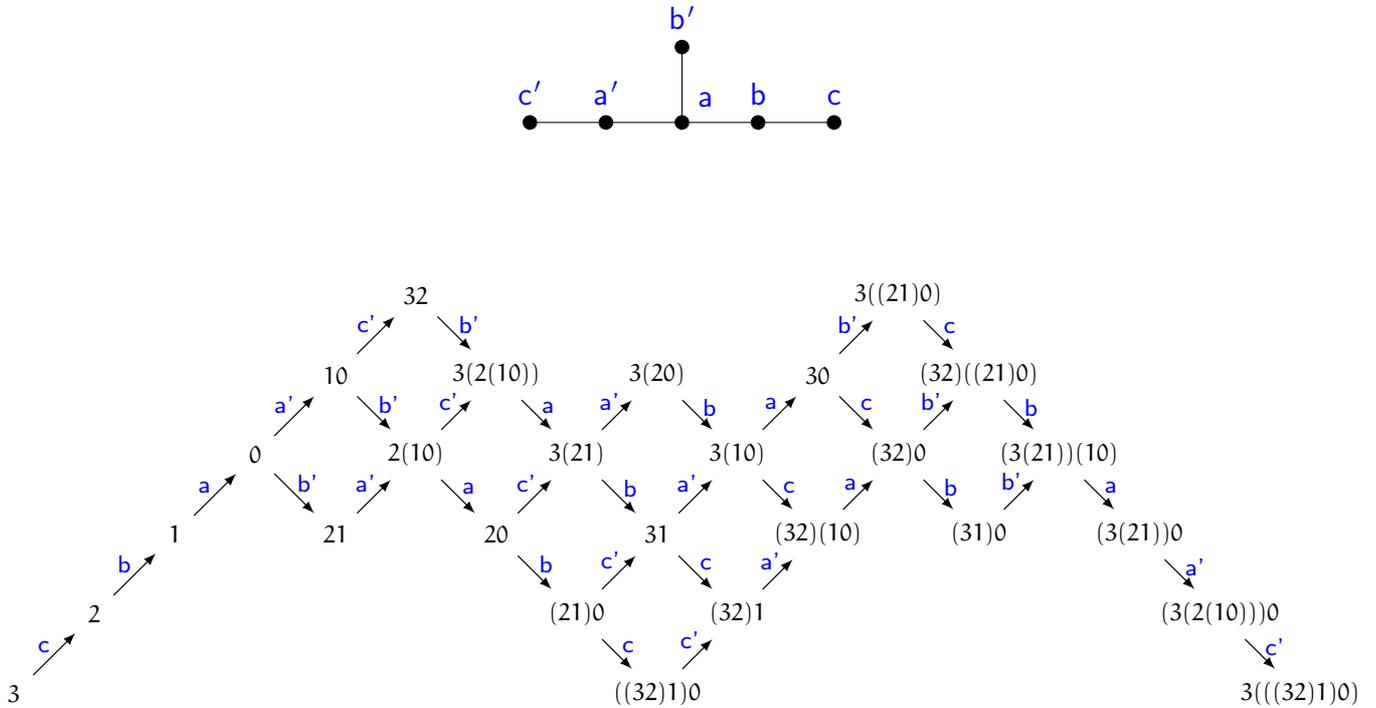
\[
  \tikz[every label/.style={rt,above}]
  {\draw (0,0) node[mynode,label=c'] {} -- ++(1,0)
    node[mynode,label=a'] {} -- ++(1,0)
    node[mynode,label={[name=aa]above right:a}] (a) {} -- ++(1,0)
    node[mynode,label={[name=bb]b}] (b) {} -- ++(1,0)
    node[mynode,label={[name=cc]c}] (c) {}; \draw (2,0) --
    (2,1) node[mynode,label=b'] {};
}
\]

\[
\tikz{\useasboundingbox (-7.5,-0.2) rectangle (7.5,4.7);
\node[rotate=45,scale=1.15]{
\begin{tikzpicture}[>=latex,nodes={rotate=-45,font=\tiny,\mcol},on grid,math mode,every edge quotes/.style={rt,auto,pos=0.25}]
\coordinate (prev) at (0,0);
\foreach \pos/\name/\alt
in {right/3,right/2,right/1,right/0,right/10,right/32,
below/3(2(10))/3o2o10cc,left/2(10)/2o10c,left/21,below right/20,
right/3(21)/3o21c,right/3(20)/3o20c,below/3(10)/3o10c,
left/31,left/(21)0/o21c0,below/((32)1)0/oo32c1c0,right/(32)1/o32c1,
right/(32)(10)/o32co10c,right/(32)0/o32c0,above/30,right/3((21)0)/3oo21c0c,
below/(32)((21)0)/o32coo21c0c,below/(3(21))(10)/o3o21cco10c,
left/(31)0/o31c0,below right/(3(21))0/o3o21cc0,
below/(3(2(10)))0/o3o2o10ccc0,below/3(((32)1)0)/3ooo32c1c0c}
{ \node ({\alt}) [\pos=1.3cm and 1.3cm of prev,alias=prev] {\name}; }
\path[->] 
(3) edge["c"] (2)
(2) edge["b"] (1)
(1) edge["a"] (0)
(0) edge["a'"] (10) edge["b'"] (21)
(10) edge["c'"] (32) edge["b'"] (2o10c)
(32) edge["b'"] (3o2o10cc)
(21) edge["a'"] (2o10c)
(2o10c) edge["c'"] (3o2o10cc) edge["a"] (20)
(3o2o10cc) edge["a"] (3o21c)
(20) edge["c'"] (3o21c) edge["b"] (o21c0)
(3o21c) edge["a'"] (3o20c) edge["b"] (31)
(3o20c) edge["b"] (3o10c)
(o21c0) edge["c'"] (31) edge["c"] (oo32c1c0)
(31) edge["a'"] (3o10c) edge["c"] (o32c1)
(3o10c) edge["a"] (30) edge["c"] (o32co10c)
(30) edge["b'"] (3oo21c0c) edge["c"] (o32c0)
(3oo21c0c) edge["c"] (o32coo21c0c)
(oo32c1c0) edge["c'"] (o32c1)
(o32c1) edge["a'"] (o32co10c)
(o32co10c) edge["a"] (o32c0)
(o32c0) edge["b'"] (o32coo21c0c) edge["b"] (o31c0)
(o32coo21c0c) edge["b"] (o3o21cco10c)
(o31c0) edge["b'"] (o3o21cco10c)
(o3o21cco10c) edge["a"] (o3o21cc0)
(o3o21cc0) edge["a'"] (o3o2o10ccc0)
(o3o2o10ccc0) edge["c'"] (3ooo32c1c0c)
;
\end{tikzpicture}
};
}
\]
In this case the techniques of \cite{artic46,artic68} do {\em not}\/ work right out of
the box; during the degeneration there is a divergence, due to
the proliferation of equivariant rhombi with negative inversion number.
It also seems related to the fact that
\[
 (V^{E_6}_{\omega_1})^{\tensor 2} \iso 
V^{E_6}_{2\omega_1} \oplus (V^{E_6}_{\omega_1})^* \oplus
V^{E_6}_{\omega_3}
\]
has three summands (the Cartan product, the invariant trilinear form,
and another) unlike the $d=1,2$ cases that had only two summands.

We can partially rescue the argument, but at the cost of $T$-equivariance.
This gives our other main result:
\begin{thm}\label{thm:d3}
  Buch's correction holds: the puzzle pieces above (and their
  rotations) correctly compute Schubert calculus in the ordinary
  cohomology of $3$-step flag manifolds.

  One can also obtain formul\ae\ for Schubert calculus in 
  the $K$-theory of $3$-step flag manifolds, with respect to either
  the standard basis (of structure sheaves) or the dual basis 
  (of ideal sheaves of the boundaries). This requires another 151 pieces,
  which we list in appendix \ref{app:d3}.
  The resulting formula is manifestly positive in the sense of
  \cite{Brion-KPos}.
\end{thm}

\begin{ex}
$c^{103213,103213}_{323011}=4$ (in nonequivariant $K$-theory),
as counted by the following puzzles:
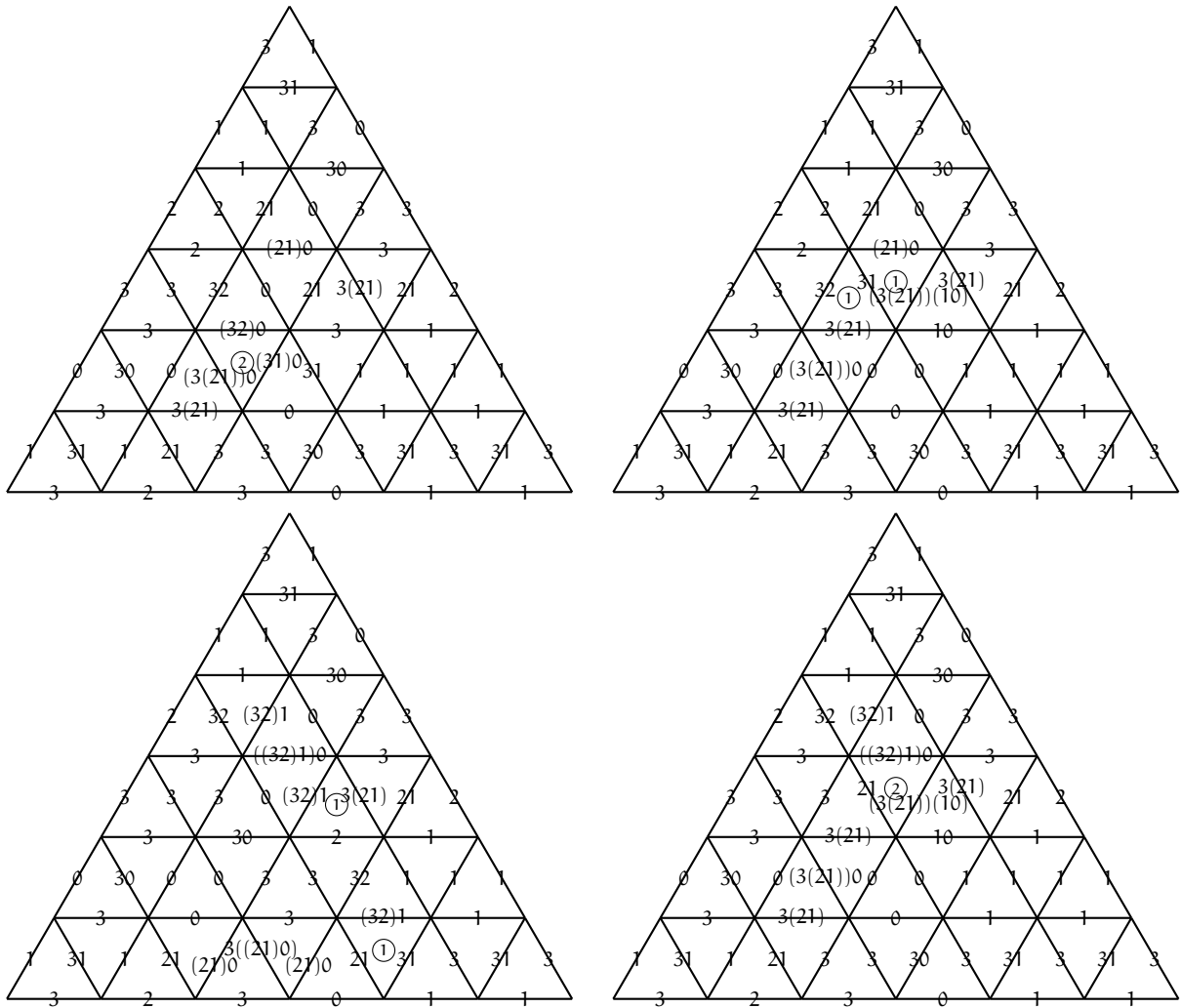
\begin{center}
\begin{tikzpicture}[math mode,nodes={\mcol},x={(-0.577cm,-1cm)},y={(0.577cm,-1cm)},scale=1.1]
\draw[thick] (0,0) -- node[pos=0.5] {\ss 1} ++(0,1); \draw[thick] (0,0) -- node[pos=0.5] {\ss 3} ++(1,0); \draw[thick] (0+1,0) -- node {\ss 31} ++(-1,1); 
\draw[thick] (0,1) -- node[pos=0.5] {\ss 0} ++(0,1); \draw[thick] (0,1) -- node[pos=0.5] {\ss 3} ++(1,0); \draw[thick] (0+1,1) -- node {\ss 30} ++(-1,1); 
\draw[thick] (0,2) -- node[pos=0.5] {\ss 3} ++(0,1); \draw[thick] (0,2) -- node[pos=0.5] {\ss 3} ++(1,0); \draw[thick] (0+1,2) -- node {\ss 3} ++(-1,1); 
\draw[thick] (0,3) -- node[pos=0.5] {\ss 2} ++(0,1); \draw[thick] (0,3) -- node[pos=0.5] {\ss 21} ++(1,0); \draw[thick] (0+1,3) -- node {\ss 1} ++(-1,1); 
\draw[thick] (0,4) -- node[pos=0.5] {\ss 1} ++(0,1); \draw[thick] (0,4) -- node[pos=0.5] {\ss 1} ++(1,0); \draw[thick] (0+1,4) -- node {\ss 1} ++(-1,1); 
\draw[thick] (0,5) -- node[pos=0.5] {\ss 3} ++(0,1); \draw[thick] (0,5) -- node[pos=0.5] {\ss 31} ++(1,0); \draw[thick] (0+1,5) -- node {\ss 1} ++(-1,1); 
\draw[thick] (1,0) -- node[pos=0.5] {\ss 1} ++(0,1); \draw[thick] (1,0) -- node[pos=0.5] {\ss 1} ++(1,0); \draw[thick] (1+1,0) -- node {\ss 1} ++(-1,1); 
\draw[thick] (1,1) -- node[pos=0.5] {\ss 0} ++(0,1); \draw[thick] (1,1) -- node[pos=0.5] {\ss 21} ++(1,0); \draw[thick] (1+1,1) -- node {\ss (21)0} ++(-1,1); 
\draw[thick] (1,2) -- node[pos=0.5] {\ss 3(21)} ++(0,1); \draw[thick] (1,2) -- node[pos=0.5] {\ss 21} ++(1,0); \draw[thick] (1+1,2) -- node {\ss 3} ++(-1,1); 
\draw[thick] (1,3) -- node[pos=0.5] {\ss 1} ++(0,1); \draw[thick] (1,3) -- node[pos=0.5] {\ss 1} ++(1,0); \draw[thick] (1+1,3) -- node {\ss 1} ++(-1,1); 
\draw[thick] (1,4) -- node[pos=0.5] {\ss 3} ++(0,1); \draw[thick] (1,4) -- node[pos=0.5] {\ss 31} ++(1,0); \draw[thick] (1+1,4) -- node {\ss 1} ++(-1,1); 
\draw[thick] (2,0) -- node[pos=0.5] {\ss 2} ++(0,1); \draw[thick] (2,0) -- node[pos=0.5] {\ss 2} ++(1,0); \draw[thick] (2+1,0) -- node {\ss 2} ++(-1,1); 
\draw[thick] (2,1) -- node[pos=0.5] {\ss 0} ++(0,1); \draw[thick] (2,1) -- node[pos=0.5] {\ss 32} ++(1,0); \draw[thick] (2+1,1) -- node {\ss (32)0} ++(-1,1); 
\node[invlabel,text=black] at (2+0.7,1+0.7) {\sss 2};
\draw[thick] (2,2) -- node[pos=0.5] {\ss 31} ++(0,1); \draw[thick] (2,2) -- node[pos=0.4] {\ss\ \ (31)0} ++(1,0); \draw[thick] (2+1,2) -- node {\ss 0} ++(-1,1); 
\draw[thick] (2,3) -- node[pos=0.5] {\ss 3} ++(0,1); \draw[thick] (2,3) -- node[pos=0.5] {\ss 30} ++(1,0); \draw[thick] (2+1,3) -- node {\ss 0} ++(-1,1); 
\draw[thick] (3,0) -- node[pos=0.5] {\ss 3} ++(0,1); \draw[thick] (3,0) -- node[pos=0.5] {\ss 3} ++(1,0); \draw[thick] (3+1,0) -- node {\ss 3} ++(-1,1); 
\draw[thick] (3,1) -- node[pos=0.6] {\ss (3(21))0\ } ++(0,1); \draw[thick] (3,1) -- node[pos=0.5] {\ss 0} ++(1,0); \draw[thick] (3+1,1) -- node {\ss 3(21)} ++(-1,1); 
\draw[thick] (3,2) -- node[pos=0.5] {\ss 3} ++(0,1); \draw[thick] (3,2) -- node[pos=0.5] {\ss 3} ++(1,0); \draw[thick] (3+1,2) -- node {\ss 3} ++(-1,1); 
\draw[thick] (4,0) -- node[pos=0.5] {\ss 30} ++(0,1); \draw[thick] (4,0) -- node[pos=0.5] {\ss 0} ++(1,0); \draw[thick] (4+1,0) -- node {\ss 3} ++(-1,1); 
\draw[thick] (4,1) -- node[pos=0.5] {\ss 21} ++(0,1); \draw[thick] (4,1) -- node[pos=0.5] {\ss 1} ++(1,0); \draw[thick] (4+1,1) -- node {\ss 2} ++(-1,1); 
\draw[thick] (5,0) -- node[pos=0.5] {\ss 31} ++(0,1); \draw[thick] (5,0) -- node[pos=0.5] {\ss 1} ++(1,0); \draw[thick] (5+1,0) -- node {\ss 3} ++(-1,1); 
\end{tikzpicture}
\quad
\begin{tikzpicture}[math mode,nodes={\mcol},x={(-0.577cm,-1cm)},y={(0.577cm,-1cm)},scale=1.1]
\draw[thick] (0,0) -- node[pos=0.5] {\ss 1} ++(0,1); \draw[thick] (0,0) -- node[pos=0.5] {\ss 3} ++(1,0); \draw[thick] (0+1,0) -- node {\ss 31} ++(-1,1); 
\draw[thick] (0,1) -- node[pos=0.5] {\ss 0} ++(0,1); \draw[thick] (0,1) -- node[pos=0.5] {\ss 3} ++(1,0); \draw[thick] (0+1,1) -- node {\ss 30} ++(-1,1); 
\draw[thick] (0,2) -- node[pos=0.5] {\ss 3} ++(0,1); \draw[thick] (0,2) -- node[pos=0.5] {\ss 3} ++(1,0); \draw[thick] (0+1,2) -- node {\ss 3} ++(-1,1); 
\draw[thick] (0,3) -- node[pos=0.5] {\ss 2} ++(0,1); \draw[thick] (0,3) -- node[pos=0.5] {\ss 21} ++(1,0); \draw[thick] (0+1,3) -- node {\ss 1} ++(-1,1); 
\draw[thick] (0,4) -- node[pos=0.5] {\ss 1} ++(0,1); \draw[thick] (0,4) -- node[pos=0.5] {\ss 1} ++(1,0); \draw[thick] (0+1,4) -- node {\ss 1} ++(-1,1); 
\draw[thick] (0,5) -- node[pos=0.5] {\ss 3} ++(0,1); \draw[thick] (0,5) -- node[pos=0.5] {\ss 31} ++(1,0); \draw[thick] (0+1,5) -- node {\ss 1} ++(-1,1); 
\draw[thick] (1,0) -- node[pos=0.5] {\ss 1} ++(0,1); \draw[thick] (1,0) -- node[pos=0.5] {\ss 1} ++(1,0); \draw[thick] (1+1,0) -- node {\ss 1} ++(-1,1); 
\draw[thick] (1,1) -- node[pos=0.5] {\ss 0} ++(0,1); \draw[thick] (1,1) -- node[pos=0.5] {\ss 21} ++(1,0); \draw[thick] (1+1,1) -- node {\ss (21)0} ++(-1,1); 
\node[invlabel,text=black] at (1+0.7,1+0.7) {\sss 1};
\draw[thick] (1,2) -- node[pos=0.4] {\ss 3(21)} ++(0,1); \draw[thick] (1,2) -- node[pos=0.6] {\ss\ (3(21))(10)} ++(1,0); \draw[thick] (1+1,2) -- node {\ss 10} ++(-1,1); 
\draw[thick] (1,3) -- node[pos=0.5] {\ss 1} ++(0,1); \draw[thick] (1,3) -- node[pos=0.5] {\ss 1} ++(1,0); \draw[thick] (1+1,3) -- node {\ss 1} ++(-1,1); 
\draw[thick] (1,4) -- node[pos=0.5] {\ss 3} ++(0,1); \draw[thick] (1,4) -- node[pos=0.5] {\ss 31} ++(1,0); \draw[thick] (1+1,4) -- node {\ss 1} ++(-1,1); 
\draw[thick] (2,0) -- node[pos=0.5] {\ss 2} ++(0,1); \draw[thick] (2,0) -- node[pos=0.5] {\ss 2} ++(1,0); \draw[thick] (2+1,0) -- node {\ss 2} ++(-1,1); 
\draw[thick] (2,1) -- node[pos=0.4] {\ss 31} ++(0,1); \draw[thick] (2,1) -- node[pos=0.5] {\ss 32} ++(1,0); \draw[thick] (2+1,1) -- node {\ss 3(21)} ++(-1,1); 
\node[invlabel,text=black] at (2+0.3,1+0.3) {\sss 1};
\draw[thick] (2,2) -- node[pos=0.5] {\ss 0} ++(0,1); \draw[thick] (2,2) -- node[pos=0.5] {\ss 0} ++(1,0); \draw[thick] (2+1,2) -- node {\ss 0} ++(-1,1); 
\draw[thick] (2,3) -- node[pos=0.5] {\ss 3} ++(0,1); \draw[thick] (2,3) -- node[pos=0.5] {\ss 30} ++(1,0); \draw[thick] (2+1,3) -- node {\ss 0} ++(-1,1); 
\draw[thick] (3,0) -- node[pos=0.5] {\ss 3} ++(0,1); \draw[thick] (3,0) -- node[pos=0.5] {\ss 3} ++(1,0); \draw[thick] (3+1,0) -- node {\ss 3} ++(-1,1); 
\draw[thick] (3,1) -- node[pos=0.5] {\ss (3(21))0} ++(0,1); \draw[thick] (3,1) -- node[pos=0.5] {\ss 0} ++(1,0); \draw[thick] (3+1,1) -- node {\ss 3(21)} ++(-1,1); 
\draw[thick] (3,2) -- node[pos=0.5] {\ss 3} ++(0,1); \draw[thick] (3,2) -- node[pos=0.5] {\ss 3} ++(1,0); \draw[thick] (3+1,2) -- node {\ss 3} ++(-1,1); 
\draw[thick] (4,0) -- node[pos=0.5] {\ss 30} ++(0,1); \draw[thick] (4,0) -- node[pos=0.5] {\ss 0} ++(1,0); \draw[thick] (4+1,0) -- node {\ss 3} ++(-1,1); 
\draw[thick] (4,1) -- node[pos=0.5] {\ss 21} ++(0,1); \draw[thick] (4,1) -- node[pos=0.5] {\ss 1} ++(1,0); \draw[thick] (4+1,1) -- node {\ss 2} ++(-1,1); 
\draw[thick] (5,0) -- node[pos=0.5] {\ss 31} ++(0,1); \draw[thick] (5,0) -- node[pos=0.5] {\ss 1} ++(1,0); \draw[thick] (5+1,0) -- node {\ss 3} ++(-1,1); 
\end{tikzpicture}
\quad
\begin{tikzpicture}[math mode,nodes={\mcol},x={(-0.577cm,-1cm)},y={(0.577cm,-1cm)},scale=1.1]
\draw[thick] (0,0) -- node[pos=0.5] {\ss 1} ++(0,1); \draw[thick] (0,0) -- node[pos=0.5] {\ss 3} ++(1,0); \draw[thick] (0+1,0) -- node {\ss 31} ++(-1,1); 
\draw[thick] (0,1) -- node[pos=0.5] {\ss 0} ++(0,1); \draw[thick] (0,1) -- node[pos=0.5] {\ss 3} ++(1,0); \draw[thick] (0+1,1) -- node {\ss 30} ++(-1,1); 
\draw[thick] (0,2) -- node[pos=0.5] {\ss 3} ++(0,1); \draw[thick] (0,2) -- node[pos=0.5] {\ss 3} ++(1,0); \draw[thick] (0+1,2) -- node {\ss 3} ++(-1,1); 
\draw[thick] (0,3) -- node[pos=0.5] {\ss 2} ++(0,1); \draw[thick] (0,3) -- node[pos=0.5] {\ss 21} ++(1,0); \draw[thick] (0+1,3) -- node {\ss 1} ++(-1,1); 
\draw[thick] (0,4) -- node[pos=0.5] {\ss 1} ++(0,1); \draw[thick] (0,4) -- node[pos=0.5] {\ss 1} ++(1,0); \draw[thick] (0+1,4) -- node {\ss 1} ++(-1,1); 
\draw[thick] (0,5) -- node[pos=0.5] {\ss 3} ++(0,1); \draw[thick] (0,5) -- node[pos=0.5] {\ss 31} ++(1,0); \draw[thick] (0+1,5) -- node {\ss 1} ++(-1,1); 
\draw[thick] (1,0) -- node[pos=0.5] {\ss 1} ++(0,1); \draw[thick] (1,0) -- node[pos=0.5] {\ss 1} ++(1,0); \draw[thick] (1+1,0) -- node {\ss 1} ++(-1,1); 
\draw[thick] (1,1) -- node[pos=0.5] {\ss 0} ++(0,1); \draw[thick] (1,1) -- node[pos=0.5] {\ss (32)1} ++(1,0); \draw[thick] (1+1,1) -- node {\ss ((32)1)0} ++(-1,1); 
\draw[thick] (1,2) -- node[pos=0.5] {\ss\ 3(21)} ++(0,1); \draw[thick] (1,2) -- node[pos=0.5] {\ss (32)1\ \ } ++(1,0); \draw[thick] (1+1,2) -- node {\ss 2} ++(-1,1); 
\node[invlabel,text=black] at (1+0.3,2+0.3) {\sss 1};
\draw[thick] (1,3) -- node[pos=0.5] {\ss 1} ++(0,1); \draw[thick] (1,3) -- node[pos=0.5] {\ss 32} ++(1,0); \draw[thick] (1+1,3) -- node {\ss (32)1} ++(-1,1); 
\node[invlabel,text=black] at (1+0.7,3+0.7) {\sss 1};
\draw[thick] (1,4) -- node[pos=0.5] {\ss 3} ++(0,1); \draw[thick] (1,4) -- node[pos=0.5] {\ss 31} ++(1,0); \draw[thick] (1+1,4) -- node {\ss 1} ++(-1,1); 
\draw[thick] (2,0) -- node[pos=0.5] {\ss 32} ++(0,1); \draw[thick] (2,0) -- node[pos=0.5] {\ss 2} ++(1,0); \draw[thick] (2+1,0) -- node {\ss 3} ++(-1,1); 
\draw[thick] (2,1) -- node[pos=0.5] {\ss 0} ++(0,1); \draw[thick] (2,1) -- node[pos=0.5] {\ss 3} ++(1,0); \draw[thick] (2+1,1) -- node {\ss 30} ++(-1,1); 
\draw[thick] (2,2) -- node[pos=0.5] {\ss 3} ++(0,1); \draw[thick] (2,2) -- node[pos=0.5] {\ss 3} ++(1,0); \draw[thick] (2+1,2) -- node {\ss 3} ++(-1,1); 
\draw[thick] (2,3) -- node[pos=0.5] {\ss 21} ++(0,1); \draw[thick] (2,3) -- node[pos=0.6] {\ss (21)0} ++(1,0); \draw[thick] (2+1,3) -- node {\ss 0} ++(-1,1); 
\draw[thick] (3,0) -- node[pos=0.5] {\ss 3} ++(0,1); \draw[thick] (3,0) -- node[pos=0.5] {\ss 3} ++(1,0); \draw[thick] (3+1,0) -- node {\ss 3} ++(-1,1); 
\draw[thick] (3,1) -- node[pos=0.5] {\ss 0} ++(0,1); \draw[thick] (3,1) -- node[pos=0.5] {\ss 0} ++(1,0); \draw[thick] (3+1,1) -- node {\ss 0} ++(-1,1); 
\draw[thick] (3,2) -- node[pos=0.4] {\ss 3((21)0)} ++(0,1); \draw[thick] (3,2) -- node[pos=0.6] {\ss (21)0} ++(1,0); \draw[thick] (3+1,2) -- node {\ss 3} ++(-1,1); 
\draw[thick] (4,0) -- node[pos=0.5] {\ss 30} ++(0,1); \draw[thick] (4,0) -- node[pos=0.5] {\ss 0} ++(1,0); \draw[thick] (4+1,0) -- node {\ss 3} ++(-1,1); 
\draw[thick] (4,1) -- node[pos=0.5] {\ss 21} ++(0,1); \draw[thick] (4,1) -- node[pos=0.5] {\ss 1} ++(1,0); \draw[thick] (4+1,1) -- node {\ss 2} ++(-1,1); 
\draw[thick] (5,0) -- node[pos=0.5] {\ss 31} ++(0,1); \draw[thick] (5,0) -- node[pos=0.5] {\ss 1} ++(1,0); \draw[thick] (5+1,0) -- node {\ss 3} ++(-1,1); 
\end{tikzpicture}
\quad
\begin{tikzpicture}[math mode,nodes={\mcol},x={(-0.577cm,-1cm)},y={(0.577cm,-1cm)},scale=1.1]
\draw[thick] (0,0) -- node[pos=0.5] {\ss 1} ++(0,1); \draw[thick] (0,0) -- node[pos=0.5] {\ss 3} ++(1,0); \draw[thick] (0+1,0) -- node {\ss 31} ++(-1,1); 
\draw[thick] (0,1) -- node[pos=0.5] {\ss 0} ++(0,1); \draw[thick] (0,1) -- node[pos=0.5] {\ss 3} ++(1,0); \draw[thick] (0+1,1) -- node {\ss 30} ++(-1,1); 
\draw[thick] (0,2) -- node[pos=0.5] {\ss 3} ++(0,1); \draw[thick] (0,2) -- node[pos=0.5] {\ss 3} ++(1,0); \draw[thick] (0+1,2) -- node {\ss 3} ++(-1,1); 
\draw[thick] (0,3) -- node[pos=0.5] {\ss 2} ++(0,1); \draw[thick] (0,3) -- node[pos=0.5] {\ss 21} ++(1,0); \draw[thick] (0+1,3) -- node {\ss 1} ++(-1,1); 
\draw[thick] (0,4) -- node[pos=0.5] {\ss 1} ++(0,1); \draw[thick] (0,4) -- node[pos=0.5] {\ss 1} ++(1,0); \draw[thick] (0+1,4) -- node {\ss 1} ++(-1,1); 
\draw[thick] (0,5) -- node[pos=0.5] {\ss 3} ++(0,1); \draw[thick] (0,5) -- node[pos=0.5] {\ss 31} ++(1,0); \draw[thick] (0+1,5) -- node {\ss 1} ++(-1,1); 
\draw[thick] (1,0) -- node[pos=0.5] {\ss 1} ++(0,1); \draw[thick] (1,0) -- node[pos=0.5] {\ss 1} ++(1,0); \draw[thick] (1+1,0) -- node {\ss 1} ++(-1,1); 
\draw[thick] (1,1) -- node[pos=0.5] {\ss 0} ++(0,1); \draw[thick] (1,1) -- node[pos=0.5] {\ss (32)1} ++(1,0); \draw[thick] (1+1,1) -- node {\ss ((32)1)0} ++(-1,1); 
\node[invlabel,text=black] at (1+0.7,1+0.7) {\sss 2};
\draw[thick] (1,2) -- node[pos=0.4] {\ss 3(21)} ++(0,1); \draw[thick] (1,2) -- node[pos=0.6] {\ss\ (3(21))(10)} ++(1,0); \draw[thick] (1+1,2) -- node {\ss 10} ++(-1,1); 
\draw[thick] (1,3) -- node[pos=0.5] {\ss 1} ++(0,1); \draw[thick] (1,3) -- node[pos=0.5] {\ss 1} ++(1,0); \draw[thick] (1+1,3) -- node {\ss 1} ++(-1,1); 
\draw[thick] (1,4) -- node[pos=0.5] {\ss 3} ++(0,1); \draw[thick] (1,4) -- node[pos=0.5] {\ss 31} ++(1,0); \draw[thick] (1+1,4) -- node {\ss 1} ++(-1,1); 
\draw[thick] (2,0) -- node[pos=0.5] {\ss 32} ++(0,1); \draw[thick] (2,0) -- node[pos=0.5] {\ss 2} ++(1,0); \draw[thick] (2+1,0) -- node {\ss 3} ++(-1,1); 
\draw[thick] (2,1) -- node[pos=0.4] {\ss 21} ++(0,1); \draw[thick] (2,1) -- node[pos=0.5] {\ss 3} ++(1,0); \draw[thick] (2+1,1) -- node {\ss 3(21)} ++(-1,1); 
\draw[thick] (2,2) -- node[pos=0.5] {\ss 0} ++(0,1); \draw[thick] (2,2) -- node[pos=0.5] {\ss 0} ++(1,0); \draw[thick] (2+1,2) -- node {\ss 0} ++(-1,1); 
\draw[thick] (2,3) -- node[pos=0.5] {\ss 3} ++(0,1); \draw[thick] (2,3) -- node[pos=0.5] {\ss 30} ++(1,0); \draw[thick] (2+1,3) -- node {\ss 0} ++(-1,1); 
\draw[thick] (3,0) -- node[pos=0.5] {\ss 3} ++(0,1); \draw[thick] (3,0) -- node[pos=0.5] {\ss 3} ++(1,0); \draw[thick] (3+1,0) -- node {\ss 3} ++(-1,1); 
\draw[thick] (3,1) -- node[pos=0.5] {\ss (3(21))0} ++(0,1); \draw[thick] (3,1) -- node[pos=0.5] {\ss 0} ++(1,0); \draw[thick] (3+1,1) -- node {\ss 3(21)} ++(-1,1); 
\draw[thick] (3,2) -- node[pos=0.5] {\ss 3} ++(0,1); \draw[thick] (3,2) -- node[pos=0.5] {\ss 3} ++(1,0); \draw[thick] (3+1,2) -- node {\ss 3} ++(-1,1); 
\draw[thick] (4,0) -- node[pos=0.5] {\ss 30} ++(0,1); \draw[thick] (4,0) -- node[pos=0.5] {\ss 0} ++(1,0); \draw[thick] (4+1,0) -- node {\ss 3} ++(-1,1); 
\draw[thick] (4,1) -- node[pos=0.5] {\ss 21} ++(0,1); \draw[thick] (4,1) -- node[pos=0.5] {\ss 1} ++(1,0); \draw[thick] (4+1,1) -- node {\ss 2} ++(-1,1); 
\draw[thick] (5,0) -- node[pos=0.5] {\ss 31} ++(0,1); \draw[thick] (5,0) -- node[pos=0.5] {\ss 1} ++(1,0); \draw[thick] (5+1,0) -- node {\ss 3} ++(-1,1); 
\end{tikzpicture}
\end{center}
\end{ex}

\junk{recycle the first part of sect 4 here?}The situation is even worse for $d=4$, in the sense that even nonequivariant
triangles with negative inversion number proliferate, 
leading to more divergences. We shall not formulate any puzzle rule
for $d\ge 4$ in the present paper, \rem{AK adds} leaving that to
\cite{artic80}. (But even there, the rule will not be manifestly positive.)

In all three cases $d=1,2,3$, 
as observed above, there exists a trilinear invariant form
acting on a tensor product of three minuscule representations of $A_2,D_4,E_6$, 
related by an automorphism of order $3$ 
(whose action on the weight lattice is given
by $\tau$). In fact, we shall see that the nonequivariant 
puzzle rules are nothing but a diagrammatic description of 
(a degenerate, and in $K$-theory, $q$-deformed, version of) 
this trilinear invariant form.

\subsection{Plan of the paper}
In \S\ref{sec:gen}, after providing basic definitions (\S\ref{ssec:schub}),
we lay out the general construction of the
vectors $\vf_X$ we assign to multinumbers $X$
(\S\ref{ssec:X}), from which we derive
the $A_2,D_4,E_6$ representations above (\S\ref{ssec:weight}).
We also explain the counting of inversions (\S\ref{ssec:Bmatrix}). 
It is worth emphasizing that these sections \S\ref{ssec:X}-\S\ref{ssec:weight}
are only included to demystify the origins of the representations
we then exploit in \S\ref{sec:proofs}.
\S\ref{sec:proofs} is concerned with the proofs of the main results:
first, a general framework is introduced, leading to our main theorem 
(\S\ref{ssec:mainthms}), which is then applied to $d=2$ equivariant Schubert
calculus in \S\ref{sec:d2}, and to $d=3$ nonequivariant Schubert calculus
in \S\ref{sec:d3}, thus proving
theorems~\ref{thm:KTd2} and \ref{thm:d3}. 
\junk{Finally, in \S\ref{sec:neg} we discuss the difficulties associated
with extending our method to higher $d$, and
prove, subject to certain assumptions,
that there are {\em no}\/ possible 
puzzle rules for $H^*_T(3$-step$)$ or $H^*(4$-step$)$.}

\rem{AK adds}
\subsection{The next two papers in this series}
One of the mysterious aspects of the framework in this paper is that
the quantum-group calculations naturally produce a {\em deformation of}
Schubert classes. In \cite{artic80} we give a direct cohomological
interpretation of these deformations, connecting them to the
``stable classes'' of \cite{MO-qg} and to Chern-Schwartz-MacPherson classes.
In addition, we interpret the majority of the puzzle calculation as
occurring on {\em quiver varieties} that, for $d>1$, are not cotangent
bundles to flag varieties! In \cite{artic83} we will use this quiver
variety framework to suggest and solve some additional
``separated descent'' Schubert problems.

\subsection*{Acknowledgments} 
We are grateful to Sergey Fomin, Gleb Koshevoy, 
Bernard Leclerc, David E Speyer, Terry Tao, 
and Michael Wheeler for useful discussions,
and to Iva Halacheva for catching unpleasantly important mistakes in the 
principal statements.
AK is especially grateful to Anders Buch for persevering
with $3$-step puzzles when AK had given up hope.

\section{The general construction}\label{sec:gen}
\subsection{Schubert classes}\label{ssec:schub}
Let $B_\pm \leq GL_n(\CC)$ denote the upper/lower triangular matrices,
with intersection $T$ the diagonal matrices,
and let $P_- \geq B_-$ be a standard parabolic with Levi factor $\prod_{i=0}^d GL_{p_i}(\CC)$.  
Our Schubert cycles $X^\sigma := P_- \dom \overline{P_- \sigma B_+}$
are right $B_+$-orbit closures on $P_- \dom GL_n(\CC)$, 
indexed by $\prod_{i=0}^d \mathcal{S}_{p_i} \dom \mathcal{S}_n$,
which we identify with strings of length $n$ in an ordered alphabet
with multiplicities given by the $(p_i)$; explicitly, the identification takes any representative
$\sigma\in \mathcal S_n$ of a right coset to the string $\lambda=\omega_{\sigma(1)},\ldots,\omega_{\sigma(n)}$, 
where $\omega$ is the unique weakly increasing string with $p_i$ letters $i=0,\ldots,d$.
The codimension of the cycle is given by the number of inversions in the string.
Schubert cycles $X^\lambda$ define classes $S^\lambda$ in cohomology, 
$T$-equivariant cohomology, $K$-theory, and $T$-equivariant $K$-theory.

In each case, these classes form a basis for whatever sort of 
cohomology group, considered as a module over the same cohomology of a point: 
$H(pt) \iso K(pt) \iso \ZZ$, $H_T(pt) \iso \ZZ[y_1,\ldots,y_n]$, 
$K_T(pt) \iso \ZZ[u_1^\pm,\ldots,u_n^\pm]$.
Following the literature on Schubert and Grothendieck polynomials (see in particular \cite{KM-Schubert2,KM-Schubert}),
the $u_i$ are the usual $K$-theory equivariant parameters,
but the $y_i$ are the {\em opposite}\/ of the cohomology equivariant parameters.
Our goal is to compute as in \eqref{eq:LRT} the coefficients, living in this base ring, of
the expansion in the basis of the product of two basis elements.

One benefit of working equivariantly is that we can replace 
equation \eqref{eq:LRT} with its restriction to $T$-fixed points,
as this restriction operation is well-known to be injective.
(This goes back to Segal's thesis \cite[proposition 2.1]{Segal},
which computes the kernel of any such map to be the $K_T$-torsion;
since $K_T(P_-\dom G)$ is a free $K_T$-module there is no torsion).
The equation becomes a {\em list}\/ of equations between elements of 
the base ring $K_T$, one for each $T$-fixed point $\sigma$:
\begin{equation}
  \label{eq:LRTrestricted}
  S^\lambda|_\sigma\, S^\mu|_\sigma \ =\  
  \sum_\nu \left( \sum_{P\text{ with boundary }\lambda,\mu,\nu} fug(P) \right)
  \ S^\nu|_\sigma
\end{equation}

These restrictions $S^\pi|_\sigma$ themselves have interpretations as state sums,
much like puzzles, and in \S \ref{ssec:mainthms} it is this equation
\eqref{eq:LRTrestricted}
we will find most amenable to the framework of quantum integrable systems.

\subsection{Functoriality}\label{ssec:funct}
Consider $d$-step flags $V_1 \leq V_2 \leq \ldots \leq V_d \leq \CC^n,
\dim V_i = n_i$, and {\em don't} assume strict containments. 
Then if some $n_j = n_{j+1}$, this ``$d$-step flag'' is really only
a $d'$-step flag for some $d'<d$, and the Schubert classes on this
$d$-step flag manifold are indexed by words that skip the letter $j$ entirely.
We hereafter make the inductive assumption that the puzzle rule we seek 
for $d$-step flag manifolds should specialize (when $n_j = n_{j+1}$)
to the one for $(d-1)$-step, under the 
label bijection 
$\{0,\ldots,d\} \setminus \{j\} \iso \{0,\ldots,d-1\}$. 

In this way, for each $d>0$ we get a lower bound on the set of
expected multinumber\footnote{Recall these from \S\ref{ssec:d2intro}: they
  are fully parenthesized expressions of numbers $0\ldots d$,
  such as $((43)(20))$.}
 labels, derived from each $d'<d$. For example,
at $d=2$ we expect to see $10$, $20$, $21$ (in addition to 
the numbers $0,1,2$) derived from $d=1$, but the $2(10)$, $(21)0$
labels are new. At $d=3$ we expect to see $ji$ for $3\geq j>i\geq 0$
from $d=1$, and $k(ji)$, $(kj)i$ for $3\geq k>j>i\geq 0$ from $d=2$,
but e.g. the $(32)(10)$ label is new.

More generally, by forgetting each of the $d$ subspaces we get maps
from the $d$-step flag manifolds to $d$ different $(d-1)$-step flag manifolds, 
and the above concerns the cases where these maps are isomorphisms. 
Now we consider the richer situation where they need not be.

If $Q_- \geq P_-$, then we have a projection
$P_-\dom GL_n(\CC) \onto Q_-\dom GL_n(\CC)$, and pullback maps
on these cohomology theories. Under these maps Schubert classes
pull back to Schubert classes (and products to products), taking
a string with content $(q_i)$ to the unique string with content $(p_i)$
that refines it and has the same number of inversions. For example,
\[
\begin{array}{rcl}
  \begin{array}{ccccc}   2&+&(3+2)&+&(1+3) \\   aa && fffgg && jkkk  \end{array}
 &\onto&
 \begin{array}{ccccc} 2&+&5&+&4   \\ aa && fffff && jjjj  \end{array}\\ 
 f\ j\ a\ k\ f\ a\ k\ f\ g\ k\ g&\mapsfrom&f\ j\ a\ j\ f\ a\ j\ f\ f\ j\ f
\end{array}
\]
where the top two rows indicate the coarsening ($Q_- \geq P_-$)
and the bottom row gives an example of a refinement.
This rule tells us what to do with the boundary labels on a $(Q_-)-$puzzle
to make them into the boundary of a $(P_-)-$puzzle computing the
same coefficient. So we can naturally hope that the refinement rule extends
to the interior of the puzzle as well, giving a correspondence on puzzles,
not just on their boundaries.

Each $T$-fixed point on $P_-\dom GL_n(\CC)$ is uniquely of the form
$P_-\dom P_- w$ for $w \in \prod_i \mathcal{S}_{p_i} \dom \mathcal{S}_n$;
for convenience we write this set $W_P\dom W$.
Let $\lambda$ be an element of $W_P\dom W$, and $\lambda'$ its image
in $W_Q\dom W$. Let $\mu \in W_Q\dom W$, and $\tilde\mu \in W_P\dom W$
its refinement as discussed above. (So $(\tilde\mu)' = \mu$,
but $\widetilde{(\lambda')}$ isn't necessarily $\lambda$.)
Then considering the functoriality of the composite
\[
pt \quad\mapsto\quad P_-\dom P_- \lambda \quad\mapsto\quad 
Q_-\dom Q_- \lambda'
\]
gets us a match on restrictions, $S^{\tilde\mu}|_\lambda = S^{\mu}|_{\lambda'}$.
In particular, if $\lambda_1,\lambda_2 \in W_P\dom W$ lie over the
same $\lambda' \in W_Q\dom W$, 
then $S^{\tilde\mu}|_{\lambda_1} = S^{\tilde\mu}|_{\lambda_2}$.

\junk{
  A consequence of this functoriality is that to prove our theorems
  for $d\leq 3$ it will suffice to prove them in the $d=3$ cases.}

\subsection{The vectors \texorpdfstring{$\vf_X$}{f\_X}}\label{ssec:X}
Fix $d$, the number of steps in the flag. 
At $d=1,2$ we learn there are two types of puzzle pieces,
up to rotation:
\begin{equation}\label{eq:stdtri}
\uptri{i}{i}{i} \quad i = 0,\ldots,d, \qquad\text{and}\qquad
\uptri{Y}{YX}{X} \quad \text{ for certain pairs of labels $X,Y$} 
\end{equation}
Let $\Lambda := A_2^*$ denote the $2$-dimensional weight lattice of $A_2$, 
on which $\tau$ acts by $120^\circ$ counterclockwise rotation. 
Let $\vf \in \Lambda$ denote the first fundamental weight,
so $\vf,\tau \vf$ are a basis and $\vf + \tau \vf + \tau^2 \vf = \vec 0$.

Now we consider $(\vf_i,\tau \vf_i)_{i=0,\ldots,d}$ as 
a basis of $\Lambda^{1+d}$.
We associate $\vf_i$ (resp.\ $\tau \vf_i, \tau^2 \vf_i$) to the edge label
$i$ on the NW (resp.\ S, NE) of a $\Deltatri$ triangle.
For $YX$ any multinumber
(not necessarily valid) we define
\begin{equation}\label{eq:YX}
\vf_{YX} := -\tau \vf_X - \tau^2 \vf_Y
\end{equation}
with the salutary effect that the vectors on the edges of a puzzle piece
$\uptri{Y}{YX}{X}$ (and its rotations) add to $\vec 0$.
For $d=1,2,3,$ these vectors $\{\vf_X\ :\ X$ a valid multinumber$\}$
are {\em not} yet quite the ones we used
in lemmas~\ref{lem:greend1}, \ref{lem:greend2}, \ref{lem:greend3},
since they live in $2(1+d)$-dimensional space instead of $2d$-dimensional.
They are quite close to the ``auras'' used in \cite{BuchHT} for
similar Green's theorem bookkeeping, though he worked in 
$\CC^{1+d}$ and didn't worry about integrality.

\subsection{The Gram matrix}\label{ssec:Gram}
In order to make use of the $R$-matrices of irreducible representations of
quantized affine algebras, our weights $(\vf_X)$ will need 
to be the weights of a representation. In the simplest possible case,
the representation is \defn{minuscule}, meaning it has 
only extremal weights, which are therefore all of the same norm.
So we make the following guess:
\begin{quotation}
  There should be a $\tau$-invariant symmetric form on $\Lambda^{1+d}$ 
  with respect to which, for each valid multinumber $X$, 
  we have $\left|\vf_X\right|^2 = 2$.
\end{quotation}

In particular, on each individual $\Lambda$ the form is determined by
the above condition (as we will now recalculate), but the different
$\Lambda$ factors will {\em not} be orthogonal to one another.

\rem{AK adds}
The derivation from here through \S \ref{ssec:weight} is not strictly
necessary to the results to come -- we could just declare ``we'll use the
following minuscule representations of $A_2,D_4,E_6$'' and check that
their $R$-matrices satisfy the conditions coming in \S\ref{sec:proofs}.
(In particular, that's essentially what we'll do in \cite{artic83}.)
We include the derivation to make the introduction of these $A_2,D_4,E_6$
representations seem less like a magic trick, but technically the
reader could jump to \S\ref{ssec:Bmatrix} now, at least on a first reading.

\subsubsection{$d=0$} Here 
$$ |\vf_0|^2 = |\tau\vf_0|^2 = 2
= |\tau^2\vf_0|^2 = \big|-\vf_0-\tau\vf_0\big|^2 
= |\vf_0|^2 + 2\langle \vf_0,\tau \vf_0\rangle + |\tau\vf_0|^2
= 4 + 2\langle \vf_0,\tau \vf_0\rangle.
$$ 
Hence the Gram matrix w.r.t. the ordered basis $(\vf_0,\tau \vf_0)$
is the $A_2$ Cartan matrix $\begin{bmatrix}  2&-1 \\ -1&2\end{bmatrix}$.

When we go to $d>0$, the $2\times 2$ blocks on the diagonal will
each look like this.

\subsubsection{$d=1$}
To begin with, the Gram matrix 
in the basis $\vf_0, \tau \vf_0, \vf_1, \tau \vf_1$ is
\[G_1 = 
\begin{bmatrix}
 2&-1& &A&B \\ -1&2 & & C&D \\ \\ A&C& & 2&-1 \\ B&D&&-1&2
  \end{bmatrix}\]
Since $10$ is a valid multinumber, 
\[ 2 = |\vf_{10}|^2 
= \left|-\tau\vf_0-\tau^2\vf_1\right|^2
= \left|\vf_0 + \tau\vf_1\right|^2 = 2 + 2\langle \vf_0,\tau \vf_1\rangle + 2 \]
so as in $d=0$, we learn that a certain off-diagonal entry 
(in this case, $B$) is $-1$.

The required $\tau$-invariance gives us
\begin{align*}
   A = \langle \vf_0, \vf_1 \rangle&= \langle \tau\vf_0, \tau\vf_1 \rangle = D\\
   &= \langle \tau^2\vf_0, \tau^2\vf_1 \rangle  
   = \langle -\vf_0-\tau\vf_0, -\vf_1-\tau\vf_1 \rangle 
   = \langle \vf_0+\tau\vf_0, \vf_1+\tau\vf_1 \rangle \\
   &= \langle \vf_0,\vf_1 \rangle + \langle \vf_0, \tau\vf_1\rangle
   + \langle \tau \vf_0,\vf_1 \rangle + \langle \tau\vf_0, \tau\vf_1\rangle \\
   &= A + B + C + D = 2A+C-1
\end{align*}
so $C = 1-A$.
The off-diagonal block begins to look like the diagonal: to emphasize
the similarity we will take $A = 2-a$, and obtain
\[
G_1 =
\begin{bmatrix} 
2&-1& &2-a&-1 \\ 
-1&2 & & a-1&2-a \\ \\ 
2-a& a-1& & 2&-1 \\
-1&2-a&&-1&2  
\end{bmatrix}
\]

When we go to $d>1$, the $2\times 2$ blocks above the diagonal will each 
have the same form as this off-diagonal upper block.

\subsubsection{$d=2$}
From the analysis above, and the valid multinumbers $10$, $20$, $21$,
we already know the Gram matrix looks like
\[
G_2 =
\begin{bmatrix}
  2&-1&& 2-a&-1&& 2-b&-1 \\
  -1&2&& a-1 &2-a&&  b-1&2-b \\ \\
 2-a&a-1&&  2&-1&&  2-c&-1 \\
 -1&2-a&& -1&2&&  c-1 &2-c \\ \\
 2-b&b-1&&  2-c&c-1&&  2&-1 \\
 -1&2-b&& -1&2-c&& -1&2
\end{bmatrix}
\]
Now we use
\[
\vf_{(21)0}
= -\tau^2\vf_{21}-\tau\vf_0
= -\tau^2(-\tau^2\vf_2-\tau\vf_1)-\tau\vf_0 
= \tau\vf_2 + \vf_1 -\tau\vf_0 
\]
so that
\begin{align*}
  |\vf_{(21)0}|^2 &= \left|\tau\vf_2 + \vf_1 -\tau\vf_0 \right|^2 \\
  &= |\vf_2|^2 + |\vf_1|^2 + |\vf_0|^2 +2\left(
      \langle \tau\vf_2,\vf_1 \rangle 
    + \langle \tau\vf_2,-\tau\vf_0 \rangle 
    + \langle \vf_1,-\tau\vf_0 \rangle \right) \\
  2 &= 6 + 2(-1 - (2-b) -(a-1)) = 2 + 2(b-a)
\end{align*}
to learn $a=b$. 
The validity of $2(10)$ gives us $b=c$ by a similar computation.
In short, the $2\times 2$ blocks above the diagonal are all equal;
the only freedom left is in the parameter $a$.

\subsubsection{$d\geq 3$}
The multinumbers derived using \S\ref{ssec:funct}'s functoriality
from $d\leq 2$ are already enough to tie down the matrix,
except for this single parameter $a$ in the off-diagonal blocks. 
For $d=3$ it remains to verify that
the other multinumbers $X$ from \S\ref{ssec:d3puz}
already have $|\vf_X|^2 = 2$, and indeed they do, putting no condition on $a$.

Hereafter $G_d$ refers to this $2(d+1)\times 2(d+1)$ matrix,
with $a$ its only free parameter.

\setcounter{MaxMatrixCols}{20}

\subsection{Signature, and special values of the parameter $a$}
\label{ssec:signature}
This matrix $G_d$ has rank $2$ at $a=0$, i.e. its determinant is a multiple 
of $a^{2d}$. Mysteriously, this multiple seems to be periodic in $d$ 
of period $6$ for $d>0$, much as in \cite[example 1.6]{BarotGeissZelevinsky}.
The multiples (which are necessarily squares, since the $\tau$-invariance
leads to duplicity in the eigenvalues) are
\[
\begin{matrix}
  &d && 1 & 2 & 3 & 4 & 5 & 6 & \cdots \\
  &a^{-2d} \det G_d && (a-3)^2 & 3(a-2)^2 & (2a-3)^2 & 3(a-1)^2 & a^2 & 3
  &\cdots \\ \\
  \text{so we define}& a_d &:=& 3 & 2 & 3/2 & 1 & 0 & \infty & \cdots 
\end{matrix}
\]
We checked this $6$-periodicity to $d=25$, but since we will soon restrict to
$d$ at most $4$, we didn't pursue it further.

For each $d>0$ to at least $25$ (and probably forever),
there is an interval $[0,a_d]$ of values of $a$ outside which 
$G_d$ is indefinite. 
For $d=1,2,3,4$, the matrix is positive definite for $a$ in 
the open interval $(0,a_d)$.
For $d=5$ there is no open interval, since $a_d = 0$. 
For $d=7$ (and $d\not\equiv 5 \bmod 6$, at least up to $25$)
there is again an interval, but $G_d$ is not definite in there.

\begin{prop}\label{prop:valid}
  For $d=1,2,3$, the only integer vectors in $\Lambda^{1+d}$ with norm-square $2$
  (independent of $a$!) are $\left\{\pm \tau^\bullet \vf_X
    \text{ for valid multinumbers }X\right\}$, 
  i.e. those associated to a label on a side of a $\Deltatri$ or $\nablatri$. 
\end{prop}

We omit the proof as this is a quick computer search, 
since for $a \in (0,a_d)$, the positive definiteness says we only have
to check the integral vectors inside a certain compact ball. 
For $d=4$ there are $240$ such vectors, arranged in the root system of
$E_8$, about which more in a moment.

Note that our previous definition of ``valid'' multinumber $X$ was
historical, based on whether $X$ had occurred in a proven
or conjectural puzzle rule. Now we can use proposition \ref{prop:valid}
to give an alternate definition: those $X$ such that $|\vf_X|^2 = 2$.

\subsection{Deriving roots and weights from the Gram matrix}\label{ssec:weight}
Let $K_d$ be the kernel of the summation map $\Lambda^{1+d} \onto \Lambda$, 
with evident basis $(\alpha_i := \vf_{i-1} - \vf_{i},  \tau\alpha_i)_{i=1,\ldots,d}$.
The Gram matrix on $K_d$ with respect to this basis\footnote{%
  The basis $\left( (-1)^i \alpha_i, (-1)^i \tau^2\alpha_i\right)$
  makes the Gram matrix Toeplitz, but this basis seemed more natural.}
comes out to be the Northwest $2d\times 2d$ block of 
\[ a\text{ times }
\begin{bmatrix}
  2  &{-1}& & {-1}& & & & &   \\
  {-1}&  2& &1& {-1}& & & &  \\ \\
    {-1}& 1& &2& {-1}&& {-1}& & &   \\
    &   {-1}&& {-1} & 2& & 1& {-1}& &   \\ \\
  & & & {-1}& 1&& 2& {-1}& & {-1}&   \\
  & & & & {-1}&& {-1}& 2& & 1&  {-1} &\\ \\
  & & & & & & {-1}& 1&& 2&  {-1}& {\ddots}\\
  & & & & & & & {-1}&& {-1}&  2 & \ddots \\
  &&&&& & &&&\ddots &\ddots & \ddots
\end{bmatrix}
\]

\junk{
We will use a slightly different basis
$\left( (-1)^i \alpha_i, (-1)^i \tau^2\alpha_i\right)$
in order to have the best looking Gram matrix on $K_d$;
it comes out to be the Northwest $2d\times 2d$ block of the Toeplitz matrix
\[ a\text{ times }
\begin{bmatrix}
  2&{-1}& {1}& & & & &   \\
  {-1}& 2& -1& {1}& & & &  \\
  {1}& -1& 2& {-1}& {1}& & &   \\
  &   {1}& {-1}& 2& -1& {1}& &   \\
  & & {1}& -1& 2& {-1}& {1}&   \\
  & & & {1}& {-1}& 2& -1&  {1} &\\
  & & & & {1}& -1& 2&  {-1}& {\ddots}\\
  & & & & & {1}& {-1}&  2 & \ddots \\
  &&&&& & \ddots &\ddots & \ddots
\end{bmatrix}
\]
}

\begin{thm}\label{thm:classif}
  Consider the matrix $G_4$ from \S \ref{ssec:Gram} on $\Lambda^{1+4}$.
  Recall from \S\ref{ssec:signature} that it is positive definite
  for $a \in (0,1)$; when we set $a=1$ it makes $G_4$ positive
  {\em semi}definite.  The kernel $Rad(G_4|_{a=1})$ is generated by 
  $\vec\nu := \sum_{i=0}^4 (-\tau^2)^i \vf_i$ and $\tau \vec\nu$, and the 
  composite map $K_4 \into \Lambda^{1+4} \onto \Lambda^{1+4} / Rad(G_4|_{a=1})$
  is an isomorphism of lattices. Moreover, both are isomorphic to
  the $E_8$ lattice.

  There is an ordered system of simple roots for this $E_8$ such that
  (in their $K_4$ appearance)
  \begin{itemize}
  \item the first $2d$ of them generate $K_d \leq K_4$, for $d=1,2,3,4$; and
  \item the $(2d-1)$st root in the list is $\alpha_d := \vf_{d-1} - \vf_{d}$,
    and for $d\leq 3$, $\vf_d$ is the corresponding fundamental weight
    of $K_d$. 
  \end{itemize}
  Having picked out this simple system, we return to general values of $a$.
  Then 
  \begin{enumerate}
  \item From this simple system we see that the symmetric forms on
    $K_1,K_2,K_3,K_4$ 
    are ($a$ times) those on the root lattices $A_2,D_4,E_6,E_8$.  
    \footnote{%
      David Speyer points out that the corresponding cluster varieties
      are $Gr_3(\CC^n)$ for $n=5,6,7,8$; for $n>8$ these varieties are
      still cluster, but of infinite type \cite{JScott}.  This seems
      likely to be connected to the fact that these matrices are
      almost (and have the same determinants as) the crossing matrices
      of the standard projections of the torus knots $T(3,n)$
      \cite{oeis131027}, from whose projections are built these
      cluster varieties in \cite{ShendeTWZ}.  For another mysterious
      connection, note that the number of positive roots of
      $A_2,D_4,E_6,E_8$ is three times the number of indecomposable
      preprojective modules for the $A_1,A_2,A_3,A_4$ quivers
      \cite{GLSprep}, a bit of numerology that does not seem to have
      been observed before.  }
  \item\label{item:AdFace}
    The weights $\vf_0,\ldots,\vf_d$ are vertices of the 
    convex hull $\{\vf_X\ :\ X$ valid$\}$, and form one (simplicial) face
    of that polytope.
  \item For $d\leq 3$, the weights
    $\{\vf_X\ :\ X$ only uses numbers $i\leq d\}$
    are those of the $K_d$ minuscule irrep with highest weight $\vf_d$.
  \item The weights 
    $\{\vf_X\ :\ X$ only uses numbers $i\leq 4$ and is derived
    from \S\ref{ssec:funct}'s functoriality$\}$ are
    {\em some of} those of the $K_4 \iso E_8$ adjoint representation.
  \end{enumerate}
\end{thm}

\begin{proof}
  We first check that $K_4$'s symmetric form is ($a$ times) that of $E_8$.
  It is integral, unimodular ($\det = 1$), even (even-diagonal),
  positive definite, and dimension $8$; Witt proved in 1941 that
  the only such lattice is $E_8$ (our reference is \cite{SPLAG}). 
  The computation of the radical is straightforward.

  Here is a simple system satisfying the conditions, and its induced
  Dynkin diagram.\footnote{To check this, let the two matrices above 
    be $G_4$ and $S$, and observe that $2-SG_4S^T$ is 
    the adjacency matrix for the graph pictured.}
  \[
\begin{blockarray}{lrrrrrrrr}
 &  \alpha_1 &\tau \alpha_1 &   \alpha_2 &\tau \alpha_2 &
  \alpha_3 &\tau \alpha_3 &  \alpha_4 &\tau \alpha_4 \\
    \begin{block}{l(rrrrrrrr)}             
\rt{a} & 1& &&&&&&&&  \\
\rt{a'}&                -1& -1& &&&&&&  \\
\rt{b} & 0& 0 & 1 &&&&&& \\
\rt{b'}&                -1& 0& -1& -1  &&&&& \\
\rt{c} & 0& 0& 0& 0& 1 &&&& \\
\rt{c'}&                0& 1& -1& 0& -1& -1&&& \\
\rt{d} & 0& 0& 0& 0& 0& 0& 1 && \\
\rt{d'}&            0& 1& -1& 1& -2& 0& -2& -1& \\
   \end{block} 
  \end{blockarray}
\qquad
  \tikz[every label/.style={rt,above}] {\draw 
    (0,0)   node[mynode,label=c'] {} -- 
    ++(1,0) node[mynode,label=a'] {} --
    ++(1,0) node[mynode,label={[name=aa]above right:a}] {} --
    ++(1,0) node[mynode,label=b] {} -- 
    ++(1,0) node[mynode,label=c] {} -- 
    ++(1,0) node[mynode,label={[name=dd]d}] {} -- 
    ++(1,0) node[mynode,label=d'] {}; 
    \draw (2,0) -- 
    (2,1)   node[mynode,label=b'] {};}
  \]
Checking the first two conditions is straightforward from the Gram matrices.
  The third is much more tedious, and for $d=3$ results in the crystal 
  pictured after lemma \ref{lem:greend3}. 
  (Note that the crystal of a minuscule representation is particularly
  simple -- given two weights $\lambda$, $\lambda+\alpha$ where
  $\alpha$ is a simple root, there is a directed edge 
  $\lambda \to \lambda+\alpha$.)
  \junk{
    Within that crystal we find the
    $(a,a',b,b')$-subcrystal with high weight $\vf_2$, 
    and the $(a,a')$-subcrystal with high weight $\vf_1$.
    (The parenthetical comment came from brute force computer search.)
  }
  \junk{AK: I think what remains is to write the simple roots as 
    combinations of $(\vf_i)$, up to multiples of $\nu$ and $\tau\nu$.}
\end{proof}

In fact the conditions on the simple system in theorem \ref{thm:classif} 
make the choice of the first seven simple roots unique; 
alas, there is no choice for the eighth that would make $\vf_4$
a fundamental weight.

For $d=1,2,3,4$, if instead of $a=1$ we set $a$ equal to the $a_d$ from
\S\ref{ssec:signature} ($3,2,1\frac{1}{2},1$ respectively), 
then the form on $\Lambda^{1+d}$ becomes
positive semidefinite with two null directions,
and its quotient is the weight lattice of $A_2,D_4,E_6,E_8$ respectively.
This is how the vectors in lemmas \ref{lem:greend1}-\ref{lem:greend3} arise.

Theorem~\ref{thm:classif} is {\em very}\/ suggestive that $4$-step 
Schubert calculus might be approached using edge labels based on 
the roots of $E_8$. However, as briefly indicated in the introduction,
the situation is somewhat complicated, and will be discussed in
\rem{AK changes} 
the next article \cite{artic80} in this series.

\begin{rmk*}
  We thank the referee for correcting our calculation of $a^{-2\cdot 6}\det G_6$.
  The limit $\lim_{a\to\infty} G_6/a$ exists and (as in $d=1,2,3,4$) 
  has nullity $2$, with kernel $Rad(G_6/a|_{a\to\infty})$ generated by 
  $\vec\nu := \sum_{i=0}^6 (-\tau^2)^i \vf_i$ and $\tau \vec\nu$.
  As in $d=4$, the Gram matrix for $K_6$ is unimodular;
  unlike the $d\leq 4$ cases the resulting form has (two) negative eigenvalues
  hence is abstractly isomorphic to $E_8 
  \oplus {\tiny   \begin{bmatrix}     0&1 \\ 1&0  \end{bmatrix}}
  \oplus {\tiny   \begin{bmatrix}     0&1 \\ 1&0  \end{bmatrix}}$
  (via standard results from \cite{SPLAG}).
\end{rmk*}

\subsection{Inversion numbers of paths}\label{ssec:Bmatrix}
This subsection does not particularly follow the flow from the previous
subsections, but its results are of intrinsic interest. Besides this,
we need it for the proofs in the next section.

Fix an {\em anti}symmetric form $B$ on $\Lambda^{1+d}$, to be specified soon,
but whose actual shape is not yet important.

Let $P$ be a puzzle and $\gamma$ an oriented (possibly self-intersecting,
even edge-repeating)
path through $P$'s edges, thought of as a sequence of steps from one
vertex to the next. 
For each step $\gamma_i$ in this oriented path, we 
consider it as part of a triangle to its left,
and associate a vector $\vf(\gamma_i) = \pm \tau^k \vf_X$ 
depending on that triangle and label (as in lemmata \ref{lem:greend1}-
\ref{lem:greend3}). Define the \defn{$B$-inversion number of $\gamma$}
as $\frac{1}{2} \sum_{i<j} B(\vf(\gamma_i),\vf(\gamma_j))$.
It is something like the area of a surface bounded by $\gamma$:

\begin{lem}\label{lem:Binv}
  If the $N$-step path $\gamma$ repeats a vertex, i.e. 
  after step $b$ finds itself in the same location as before step $a$,
  then we can break $\gamma$ into the path $\gamma_{[1,a)\coprod (b,N]}$
  and the loop $\gamma_{[a,b]}$. In this situation, the $B$-inversion number
  of $\gamma$ is the sum of the inversion numbers of the path and the loop.

  Meanwhile, define the $B$-inversion number of a puzzle piece $p$ as
  the $B$-inversion number of a closed loop that traverses $\partial p$
  counterclockwise. If $\gamma_{[a,b]}$ doesn't repeat vertices other than 
  its first matching its last, then the $B$-inversion number 
  of $\gamma_{(a,b]}$ is the sum of the $B$-inversion numbers of
  the puzzle pieces it encircles, times $-1$ if $\gamma_{[a,b]}$ is
  clockwise.
\end{lem}

We emphasize that these results are independent of the antisymmetric form $B$,
which is why we have not distracted the reader with its specific form yet.

\begin{proof}
  \begin{eqnarray*}
    \sum_{i<j} B(\vf(\gamma_i),\vf(\gamma_j)) 
    &=&    \sum_{i<j<a} B(\vf(\gamma_i),\vf(\gamma_j)) 
    + \sum_{i<a\leq j\leq b} B(\vf(\gamma_i),\vf(\gamma_j)) 
    + \sum_{i<a, b<j} B(\vf(\gamma_i),\vf(\gamma_j)) \\
    &+& \sum_{a\leq i< j\leq b} B(\vf(\gamma_i),\vf(\gamma_j))  
    + \sum_{a\leq i\leq b< j} B(\vf(\gamma_i),\vf(\gamma_j)) 
    + \sum_{b < i < j} B(\vf(\gamma_i),\vf(\gamma_j)) \\
    &=&    \sum_{i<j<a} B(\vf(\gamma_i),\vf(\gamma_j)) 
    + \sum_{i<a} B\left( \vf(\gamma_i), \sum_{a\leq j\leq b} \vf(\gamma_j) \right) 
    + \sum_{i<a, b<j} B(\vf(\gamma_i),\vf(\gamma_j)) \\
    &+& \sum_{a\leq i< j\leq b} B(\vf(\gamma_i),\vf(\gamma_j))  
    + \sum_{ b< j} B \left(\sum_{a\leq i\leq b} \vf(\gamma_i),\vf(\gamma_j) \right) 
    + \sum_{b < i < j} B(\vf(\gamma_i),\vf(\gamma_j))
  \end{eqnarray*}
  The Green's theorem property of the vectors $\vf(\gamma_j)$, in the
  closed loop $\gamma_{[a,b]}$, says that 
  $\sum_{a\leq j\leq b} \vf(\gamma_j) = \vec 0$. Canceling those two terms,
  \begin{eqnarray*}
    \sum_{i<j} B(\vf(\gamma_i),\vf(\gamma_j)) 
    &=&    \sum_{i<j<a} B(\vf(\gamma_i),\vf(\gamma_j)) 
    + \sum_{i<a, b<j} B(\vf(\gamma_i),\vf(\gamma_j)) 
    + \sum_{b < i < j} B(\vf(\gamma_i),\vf(\gamma_j))  \\
    &+& \sum_{a\leq i< j\leq b} B(\vf(\gamma_i),\vf(\gamma_j))  
  \end{eqnarray*}
  giving the stated addition of inversion numbers.

  In the simplest case of the above, $b=a+1$, i.e. the $(a+1)$st step retraces
  the $a$th step, and the inversion number of the $2$-step closed loop is $0$.
  Consequently we can consider these paths as equivalent
  modulo insertion or removal of such retracings. 

  The second paragraph of the lemma is tautological if $\gamma$ bounds
  only one puzzle piece. Otherwise, we can insert extra steps into $\gamma$
  as follows. Just after its first step, along an edge of some piece $p$
  in the region encircled by $\gamma$, add extra steps into $\gamma$
  to encircle the piece $p$ entirely, and then retrace those new steps. 
  Apply the first paragraph of the lemma to break off that piece, and
  use induction based on the number of puzzle pieces enclosed by $\gamma$.
\end{proof}

(What's ``really'' going on is that any $2$-form $B$ on $\Lambda^{1+d}$
with constant coefficients is exact, i.e. is of the form $d\alpha$ 
for some $1$-form $\alpha$ with linear coefficients, and the inversion
number is the integral over the triangle $0\leq i\leq j\leq N$
of the pullback of $B$. We could rewrite the inversion number
as the integral of the pullback of $\alpha$ to the boundary of that triangle,
but this didn't seem useful.)

\junk{
  The next simplest application comes when $\gamma_a,\gamma_{a+1}$
  are two sides of a small $\Deltatri$ or $\nablatri$. We first use the lemma 
  to insert two more steps, jogging back and forth along the third
  side of the triangle. Now by a second use of the lemma,
  the triangle can be broken off as a closed loop.
}
We now get specific about $B$ on $\Lambda^{1+d}$,
defining it on our bases as 
\begin{equation}\label{eq:defB}
B(\tau^r \vf_i,\tau^s \vf_j)=
\sign(i-j)
\begin{cases}
1& s \equiv r \pmod 3\\
-1& s \equiv r+\sign(i-j)\pmod 3\\
0& \text{otherwise}
\end{cases}
\end{equation}
(as usual, $\sign(x>0):=1, \sign(x<0):=-1, \sign(0):=0$).
We use \defn{inversion number} to mean $B$-inversion number with this $B$.
The factor of $1/2$ in the definition of inversion number 
could have been subsumed into our $B$, or, 
we could have left it out in which case each puzzle piece 
(or more precisely, its boundary traversed clockwise)
would have even inversion number.
\junk{
  which we can encode as a directed graph on the vertex set 
  $\{\vf_i, \tau \vf_i\}_{i=0,\ldots,d}$: 
  \begin{itemize}
  \item each $\vf_i$ points to $\vf_j$ with $j<i$, so by $\tau$-invariance,
  \item each $\tau\vf_i$ points to $\tau \vf_j$ if $j<i$, and
  \item each $\tau\vf_i$ points to $\vf_j$ with $j>i$.
  \end{itemize}
}

\begin{lem}\label{lem:sympd1}
  Fix a partial flag manifold
  $\{V_1 \leq V_2 \leq \ldots \leq V_d \leq \CC^n:\ \dim V_i = n_i\}$,
  therefore of (complex) dimension $D :=\sum_{i<j} (n_i - n_{i-1})(n_j - n_{j-1})$.

  Let $P$ be a puzzle with boundaries $\lambda,\mu,\nu$ as usual, 
  and $\gamma,\gamma'$ paths through $P$ from the SW
  corner to the SE corner. 
  \begin{itemize}
  \item If $\gamma$ is the straight path across the bottom, then the
    inversion number of $\gamma$ is $\ell(\nu) - \frac{D}{2}$.
  \item If $\gamma'$ follows the NW then NE sides of $P$, then the
    inversion number of $\gamma'$ is $\ell(\lambda) + \ell(\mu) - \frac{D}{2}$.
  \item If one path $\gamma'$ is always weakly above another path $\gamma$
    (e.g. like the two just mentioned),
    then the difference in their inversion numbers is the sum of
    the inversion numbers of the pieces in between them.
  \end{itemize}
\end{lem}

\begin{proof}
  The inversion number of the path along the Southern edge is
  \[ \frac{1}{2}   \sum_{1 \leq i<j \neq n} B(\vf(\gamma_i),\vf(\gamma_j))
  = \frac{1}{2} \sum_{1 \leq i<j \neq n} \sign(\nu_i-\nu_j) 
  = \ell(\nu) - \frac{D}{2} 
  \]
  Similarly, the inversion number of the path along the NW then NE edge
  is $\ell(\lambda)+\ell(\mu) - D$ plus the cross-terms
  from the first half of the path with the second half.
  The only nonvanishing cross-terms 
  $\frac{1}{2} B(\tau^2 \vf_i, \tau \vf_j)$ 
  come from pairs $i<j$ with $i$ on the NW and $j$ on the NE.
  There are $D$ such pairs, so the inversion number of the second path is 
  $\ell(\lambda)+\ell(\mu)) - \frac{D}{2}$.

  The last statement follows directly from the second half of
  lemma \ref{lem:Binv}, applied to the closed counterclockwise loop  
  ``$\gamma$-then-$\gamma'$-backwards''.
\end{proof}

\section{Proofs of the main theorems}\label{sec:proofs}
The proofs of theorems~\ref{thm:KTd2} and \ref{thm:d3} follow the same
general philosophy as in \cite{artic46} and \cite{artic68}: they consist
in finding an appropriate ``quantum integrable system'' (that is, in the 
present context, a set of fugacities collectively satisfying
the Yang--Baxter equation), from which can be built both 
the Schubert basis of $K_T(G/P)$ (or $K(G/P)$ in the 3-step case) 
and the structure constants of that basis.
However, we use here a much more straightforward approach 
based on the bootstrap equation \eqref{eq:qtri}-\eqref{eq:qtrirev}, 
sidestepping some of the difficulties found in these articles (choice
of representative, issues of stability).
In appendix~\ref{app:cyclic}, we briefly sketch
an alternative route which is closer to \cite{artic46,artic68}.
In the whole of this section, $d\le 3$.

\subsection{Puzzles via the calculus of tensors}\label{ssec:tensors}
We begin with a brief attempt to bring Schubert calculators into the
integrable-systems tent, which essentially requires a proper 
linear algebra understanding of puzzles and their matching rule. 
Return to equation (\ref{eq:LRT}) of theorem~\ref{thm:KTd2}.
The requirement of matching between an edge label on a $\Deltatri$, and that on 
a neighboring $\nablatri$, we will reinterpret as a dot product 
$\langle \vec v_X, \vec v^Y \rangle = \delta_X^Y$ between elements 
of dual bases.
To set this up, we need to assign two complex vector spaces in involution,%
\junk{There is a subtle issue later, when these spaces are modules
  over a noncocommutative algebra and we must distinguish between
  dual and predual, which we ignore for the moment.}
one to the South side of $\Deltatri$ and one to the North side of $\nablatri$,
with dual bases indexed by our edge label set $L_d$ (and similarly 
choose pairs of dual vector spaces for the other two orientations).
For the moment, let's call these vector spaces $N,S,NW,SE,NE,SW$.

\junk{\subsubsection{Nonequivariant puzzles}
We start with nonequivariant Schubert calculus, both because 
it only involves triangles and because it doesn't have parameters.}
Let us start with nonequivariant, $K$-theoretic, Schubert calculus.
Define a tensor $U \in S\tensor NE\tensor NW\cong \text{Hom}(SE\otimes SW,S)$
\[
U := \sum_{X,Y,Z} fug\left( \uptri{Z}{X}{Y} \right) 
\vec v_X\tensor \vec v_Y\tensor \vec v_Z
 \quad \in S\tensor NE\otimes NW
\]
where summation is over valid pieces $\uptri{Z}{X}{Y}$ (or equivalently,
we declare the fugacity of an invalid piece to be zero).

Define $D \in SW\tensor SE\tensor N\cong \text{Hom}(S,SW\tensor SE)$ using $\nablatri$ pieces similarly,
but using dual basis elements $\vec v^X\tensor \vec v^Y\tensor \vec v^Z$.

\newcommand\calD{{\mathcal D}}
Now consider a diagram $\calD$ made of $N_\Delta$ $\Deltatri$s and
$N_\nabla$ $\nablatri$s, some edges shared, some edges labeled.
To such a diagram we assign a tensor $fug(\calD)$ as follows: start with
$U^{\tensor N_\Delta} \tensor D^{\tensor N_\nabla}$, contract the dual
vector spaces at each shared unlabeled edge, and contract with the dual basis
vector at each labeled edge. The result is a tensor $fug(\calD)$ living in the
tensor product of the vector spaces of all unmatched, unlabeled edges
of $\calD$. (Observe that if $\calD$ is a just a puzzle piece,
then this definition gives the usual fugacity $\pm 1$ for the piece,
and gives $0$ if applied to labeled triangles that aren't valid puzzle pieces.
Or if $\calD$ is an unlabeled $\Deltatri$, then $fug(\calD) = U$.)

The basic case to keep in mind is $\calD$ a size $n$ triangle made of
$n^2$ little $\Deltatri$s and $\nablatri$s, with boundary labeled by
$\lambda,\mu,\nu$ as usual. Now all internal edges are matched 
and all external edges are labeled, so the resulting tensor is 
just a number -- exactly the coefficient appearing in the
nonequivariant puzzle rule, the $d=1$ case being equation (\ref{eq:LR}).

Because we use the same labels $L_d$ for our basis and dual basis,
if $E$ is a shared unlabeled edge of $\calD$, then 

\begin{equation}\label{eq:labeledges}
  fug(\calD) = \sum_{X \in L_d} fug(
  \text{$\calD$ with edge $E$ labeled $X$}
  )
\end{equation}

\junk{\subsubsection{Equivariant puzzles}
We now explain the role of the multiplicative parameters that enter
the fugacities.  Instead of assigning a vector space to each unlabeled edge, 
we assign a module over a certain algebra, and the parameter is
involved in specifying the module structure.} 
To involve equivariant pieces, 
we also define a tensor $\check R$ living in 
$ SW \tensor SE \tensor NE \tensor NW \iso \text{Hom}(SE\tensor SW,SW\tensor SE)$,
constructed as a sum 
\[
 \check R := \sum_{W,X,Y,Z} fug\left( \rh{X}{Y}{Z}{W} \right) 
\vec v_X\tensor \vec v_Y\tensor \vec v^Z \tensor \vec v^W
\]
where 
$\rh{X}{Y}{Z}{W}$ can be filled in with either a rhombus or two triangles.
The fugacities depend on the location of the rhombus (ultimately, they 
will depend on the equivariant parameters).

Finally, each vector space associated to edges will be endowed with the action
of a {\em Hopf algebra}, in such a way that $U$, $D$ and $\check R$ are
invariant under that action (i.e., they live in the trivial representation);
in particular, $U$ and $D$ are the trilinear invariant forms
advertised in the paper's title. 
\junk{A subtlety here is that the lack of cocommutativity means that
  we should preserve cyclic ordering of edges around vertices (as we
  have implicitly done so far); and the lack of involutivity of the
  antipode means that the notion of dual representation is not
  uniquely defined, which forces us to co\"orient all edges, as will
  be done in the next section.}

\subsection{The required properties: quantum integrability}
\label{sec:genproof}
It is convenient to introduce the ``dual'' graphical representation
which is more traditional in quantum integrable systems: we denote a rhombus
with its four edges labeled by $L_d$
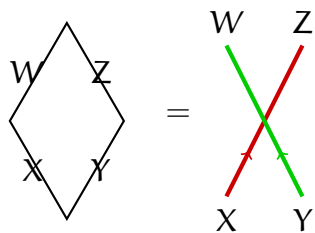
\[
\tikz[baseline=0,scale=1.5]{
\rh{X}{Y}{Z}{W}
} \quad
=
\tikz[baseline=0,xscale=0.5]
{
\draw[invarrow=0.75,dr] (-1,-1) node[below] {$\m X$} -- (1,1) node[above] {$\m Z$};
\draw[invarrow=0.75,dg] (1,-1) node[below] {$\m Y$} -- (-1,1) node[above] {$\m W$};
}
\]
The colors of lines are a reminder of the direction of the corresponding
edge of the rhombus (red lines go SouthWest, green lines SouthEast, and blue lines will go South). This information may seem redundant since lines
themselves have a given direction; however it is sometimes useful to
deform the lines (in a way that is harder to realize with rhombi), and
then their direction may vary (but not their color). We will say more about this after proposition~\ref{prop:key}. Representation-theoretically, this corresponds to the fact that lines of a given color carry a given representation of the underlying Hopf algebra.

The graphical convention 
is that a drawn rhombus\,/\,crossing {\em represents its own fugacity,} and 
that when rhombi are glued together, the labels of the internal edges
are summed over as in (\ref{eq:labeledges}), 
or in the dual language, when lines emerging from crossings are reconnected,
the labels of these lines should match and are summed over.

We then attach a parameter (usually called the spectral parameter) to
each line, in such a way that the fugacity of the rhombus\,/\,crossing
depends on the ratio\footnote{%
  We work in the multiplicative group, or equivalently with
  trigonometric solutions of the Yang-Baxter equation, since our goal is 
  to compute in $K$-theory. The corresponding additive group\,/\,rational 
  solutions compute Schubert calculus in ordinary cohomology.}
of these two parameters:
\begin{equation}\label{eq:rhombus}
  \tikz[baseline=0,xscale=0.5]
  {
    \draw[invarrow=0.75,dr] (-1,-1) node[below] {$\m X$} -- node[pos=0.75,right,black] {$\ss u''$} (1,1) node[above] {$\m Z$};
    \draw[invarrow=0.75,dg] (1,-1) node[below] {$\m Y$} -- node[pos=0.75,left,black] {$\ss u'$} (-1,1) node[above] {$\m W$};
  }
  = \check R(u)_{XY}^{WZ} \ \in\ \QQ(u',u'',\ldots)
  \qquad
  u=u''/u'
\end{equation}
\junk{corrected the ring since the spectral parameters will typically
be any of these variables, not just $u''/u'$, and
there are other spectator variables like $q$, hence the vague statement below}
The fugacity lives in some appropriate ring
of rational functions in $u'$, $u''$ and possibly other indeterminates.

We impose an important condition on the fugacity of \eqref{eq:rhombus}: 
that there is a special value $\alpha/\beta$ of the ratio of spectral
parameters such that the fugacity ``factorizes'':
\begin{pty}\label{pty:factor}
\begin{equation}\label{eq:factor}
\tikz[baseline=0,xscale=0.5]
{
\draw[invarrow=0.75,dr] (-1,-1) node[below] {$\m X$} -- node[pos=0.75,right,black] {$\ss \alpha u$} (1,1) node[above] {$\m Z$};
\draw[invarrow=0.75,dg] (1,-1) node[below] {$\m Y$} -- node[pos=0.75,left,black] {$\ss \beta u$} (-1,1) node[above] {$\m W$};
}
=
\tikz[baseline=0]
{
\draw[invarrow=0.6,dr] (-0.5,-1) node[below] {$\m X$} -- node[pos=0.6,left,black] {$\ss \alpha u$} (0,-0.5);
\draw[invarrow=0.6,dr] (0,0.5) --  node[pos=0.6,right,black] {$\ss \alpha u$} (0.5,1) node[above] {$\m Z$};
\draw[invarrow=0.6,dg] (0.5,-1) node[below] {$\m Y$} -- node[pos=0.6,right,black] {$\ss \beta u$} (0,-0.5);
\draw[invarrow=0.6,dg] (0,0.5) -- node[pos=0.6,left,black] {$\ss \beta u$} (-0.5,1) node[above] {$\m W$};
\draw[invarrow=0.6,db] (0,-0.5) -- node[pos=0.6,left,black] {$\ss u$} (0,0.5);
}
\end{equation}
where the blue line, or equivalently the diagonal of the rhombus, carries a label in $L_d$. 
\end{pty}
Equivalently, in the original graphical language, the r.h.s.\ can be pictured as
\tikz[baseline=-1mm]{\uptri{W}{}{Z}\downtri{Y}{}{X}}.
\junk{it's very tempting to change the orientation of the blue line to
  get back to our usual trilinear form, but one needs to be super
  careful about dual vs predual. Moreover, this is the natural direction
  if one thinks that puzzles are about pushforward of CSM classes along
  diagonal inclusion}
This is the way that the ``invariant trilinear form'' appears in this integrable setting.
The assignment of the parameter $u$ to the intermediate blue line is at the moment purely a definition,
but it will be useful below.
The constants $\alpha$ and $\beta$ could be set to $1$ by a redefinition of the spectral parameters, but as we shall see this would not match with usual quantum group conventions.

We need more general types of crossings, where two lines with arbitrary colors can cross each other, e.g.,
\tikz[baseline=0,xscale=0.5]
{
\draw[invarrow=0.75,db] (-1,-1) -- node[pos=0.75,right,black] {$\ss u''$} (1,1);
\draw[invarrow=0.75,dr] (1,-1) -- node[pos=0.75,left,black] {$\ss u'$} (-1,1);
}.\footnote{Many of these $R$-matrices were defined in the study of the
  cohomology of $d=1$ in \cite{artic46}; in particular, crossings of
  identical colors were called $R^*$ there.}
In the argument from proposition~\ref{prop:key}, it will become clear that
the same-color crossings are closely related to the $(150^\circ,30^\circ)$ 
rhombi in \cite{Pur} and more distantly related to the gashes in \cite{KT,BKPT}.

We also associate to each labeled line a weight:
the weight of a green (resp.\ red, blue) line labeled $X$ is
$\vec f_X$ (resp.\ $\tau^2 \vec f_X$, $-\tau \vec f_X$).
The minus sign for blue lines comes from their opposite orientation at
the trivalent vertices.

We then require the following conditions:
\begin{pty}\label{pty:ybe}
  In the pictures below, black lines can have arbitrary (independent)
  colors, and all lines can have arbitrary spectral parameters (as
  long as they match between l.h.s.\ and r.h.s.).
\begin{itemize}
\item The fugacity of any vertex is nonzero only if weight is conserved,
i.e., if the sum of incoming weights is equal to the sum of outgoing weights.
\item Yang--Baxter equation:
\begin{equation}\label{eq:ybe}
\begin{tikzpicture}[baseline=-3pt,y=2cm]
\draw[d,arrow=0.1,arrow=0.4,arrow=0.7,rounded corners] (-0.5,0.5) -- (0.75,0) -- (1.5,-0.5) (0.5,0.5) -- (0.25,0) -- (0.5,-0.5) (1.5,0.5) -- (0.75,0) -- (-0.5,-0.5);
\end{tikzpicture}
=
\begin{tikzpicture}[baseline=-3pt,y=2cm]
\draw[d,arrow=0.1,arrow=0.4,arrow=0.7,rounded corners] (-0.5,0.5) -- (0.25,0) -- (1.5,-0.5) (0.5,0.5) -- (0.75,0) -- (0.5,-0.5) (1.5,0.5) -- (0.25,0) -- (-0.5,-0.5);
\end{tikzpicture}
\end{equation}
\item Bootstrap equations: 
\junk{JBZ suggested calling them fusion equations. but there are so many uses of the word fusion in MP...}
\begin{align}\label{eq:qtri}
&\tikz[baseline=0]{
\draw[invarrow=0.5,db] (0,-1.5) -- (0,0);
\draw[invarrow=0.5,dg] (0,0) -- (-1,1);
\draw[invarrow=0.5,dr] (0,0) -- (1,1);
\draw[invarrow=0.5,rounded corners,d] (-1,-1.5) -- (-1,-0.5) -- (1.5,0) -- (1.5,1);
}
=
\tikz[baseline=0]{
\draw[invarrow=0.6,db] (0,-1.5) -- (0,-0.5);
\draw[invarrow=0.6,dg] (0,-0.5) -- (-1,1);
\draw[invarrow=0.75,dr] (0,-0.5) -- (1,1);
\draw[invarrow=0.6,rounded corners,d] (-1,-1.5) -- (-1,0) -- (1.5,0.5) -- (1.5,1);
}
&
&\tikz[baseline=0,xscale=-1]{
\draw[invarrow=0.5,db] (0,-1.5) -- (0,0);
\draw[invarrow=0.5,dr] (0,0) -- (-1,1);
\draw[invarrow=0.5,dg] (0,0) -- (1,1);
\draw[invarrow=0.5,rounded corners,d] (-1,-1.5) -- (-1,-0.5) -- (1.5,0) -- (1.5,1);
}
=
\tikz[baseline=0,xscale=-1]{
\draw[invarrow=0.6,db] (0,-1.5) -- (0,-0.5);
\draw[invarrow=0.6,dr] (0,-0.5) -- (-1,1);
\draw[invarrow=0.75,dg] (0,-0.5) -- (1,1);
\draw[invarrow=0.6,rounded corners,d] (-1,-1.5) -- (-1,0) -- (1.5,0.5) -- (1.5,1);
}
\\[4mm]\label{eq:qtrirev}
&\tikz[baseline=0,scale=-1]{
\draw[arrow=0.5,db] (0,-1.5) -- (0,0);
\draw[arrow=0.5,dg] (0,0) -- (-1,1);
\draw[arrow=0.5,dr] (0,0) -- (1,1);
\draw[arrow=0.5,rounded corners,d] (-1,-1.5) -- (-1,-0.5) -- (1.5,0) -- (1.5,1);
}
=
\tikz[baseline=0,scale=-1]{
\draw[arrow=0.6,db] (0,-1.5) -- (0,-0.5);
\draw[arrow=0.6,dg] (0,-0.5) -- (-1,1);
\draw[arrow=0.75,dr] (0,-0.5) -- (1,1);
\draw[arrow=0.6,rounded corners,d] (-1,-1.5) -- (-1,0) -- (1.5,0.5) -- (1.5,1);
}&
&\tikz[baseline=0,yscale=-1]{
\draw[arrow=0.5,db] (0,-1.5) -- (0,0);
\draw[arrow=0.5,dr] (0,0) -- (-1,1);
\draw[arrow=0.5,dg] (0,0) -- (1,1);
\draw[arrow=0.5,rounded corners,d] (-1,-1.5) -- (-1,-0.5) -- (1.5,0) -- (1.5,1);
}
=
\tikz[baseline=0,yscale=-1]{
\draw[arrow=0.6,db] (0,-1.5) -- (0,-0.5);
\draw[arrow=0.6,dr] (0,-0.5) -- (-1,1);
\draw[arrow=0.75,dg] (0,-0.5) -- (1,1);
\draw[arrow=0.6,rounded corners,d] (-1,-1.5) -- (-1,0) -- (1.5,0.5) -- (1.5,1);
}
\end{align}
\item Unitarity equation: \junk{Do we really want to emphasize this? 
    Insofar it doesn't hold for our degenerations of the standard $R$-matrix.
of course it does.}
\begin{equation}\label{eq:unit}
\begin{tikzpicture}[baseline=-3pt]
\draw[arrow=0.07,arrow=0.57,rounded corners,d] (-0.5,1) -- (0.5,0) -- (-0.5,-1) (0.5,1) -- (-0.5,0) -- (0.5,-1);
\end{tikzpicture}
=
\begin{tikzpicture}[baseline=-3pt]
\draw[arrow=0.1,arrow=0.6,rounded corners,d] (-0.5,1) -- (-0.5,-1) (0.5,1) -- (0.5,-1);
\end{tikzpicture}
\end{equation}
\item Value at equal spectral parameters: (here, the lines {\em must}\/ have the same color
for the equality to make sense, i.e., for the outgoing lines to have the same color on both sides)
\begin{equation}\label{eq:equal}
\begin{tikzpicture}[baseline=-3pt,yscale=1.5]
\draw[invarrow=0.75,d] (-0.5,-0.5) -- node[right,pos=0.75] {$\ss u$} (0.5,0.5);
\draw[invarrow=0.75,d] (0.5,-0.5) -- node[left,pos=0.75] {$\ss u$} (-0.5,0.5);
\end{tikzpicture}
=
\begin{tikzpicture}[baseline=-3pt,yscale=1.5]
\draw[invarrow=0.7,rounded corners=4mm,d] (-0.5,-0.5) -- (0,0) -- node[left] {$\ss u$} (-0.5,0.5);
\draw[invarrow=0.7,rounded corners=4mm,d] (0.5,-0.5) -- (0,0) -- node[right] {$\ss u$} (0.5,0.5);
\end{tikzpicture}
\end{equation}
\end{itemize}
\end{pty}
Morally, properties \eqref{eq:qtri}--\eqref{eq:unit} mean that one can 
freely slide lines across vertices/intersections of other lines.
For example, combining say the second equation of \eqref{eq:qtri} 
with \eqref{eq:unit}, we can also write
\begin{equation}\label{eq:qtri2}
\tikz[baseline=0]{
\draw[invarrow=0.4,db] (0,-1.5) -- (0,0);
\draw[invarrow=0.7,dg] (0,0) -- (-1,1);
\draw[invarrow=0.5,dr] (0,0) -- (1,1);
\draw[invarrow=0.7,rounded corners,d] (1,-1.5) -- (1,-1) -- (-0.75,-0.25) -- (-0.25,1);
}
=
\tikz[baseline=0]{
\draw[invarrow=0.4,db] (0,-1.5) -- (0,0);
\draw[invarrow=0.7,dg] (0,0) -- (-1,1);
\draw[invarrow=0.7,dr] (0,0) -- (1,1);
\draw[invarrow=0.4,rounded corners,d] (1,-1.5) -- (1,0) -- (-0.25,1);
}
\end{equation}
which is an equality that will be used in what follows.

We also impose the following normalization condition on the fugacities:
\begin{pty}\label{pty:norm}
For any $X,Y\in L_d$ such that 
$\langle \vf_X, \tau \vf_Y \rangle = -1$,
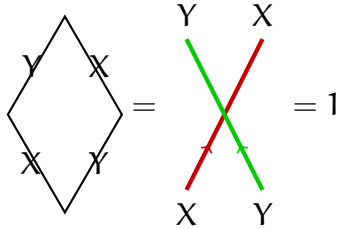
\[
\tikz[baseline=0,scale=1.5]{
\rh{X}{Y}{X}{Y}
}
=
\tikz[baseline=0,xscale=0.5]
{
\draw[invarrow=0.75,dr] (-1,-1) node[below] {$\m X$} -- (1,1) node[above] {$\m X$};
\draw[invarrow=0.75,dg] (1,-1) node[below] {$\m Y$} -- (-1,1) node[above] {$\m Y$};
}
=1
\]
(independent of $u$).
\end{pty}

\junk{Note that the condition
there exists a $Z$ such that $\tau^2\vec f_X+\vec f_Y+\tau\vec f_Z=0$,
is such a $Z$ exists, it is unique {minuscule}}
There is a remaining gauge freedom on the fugacities, namely, one can
multiply any vertex by the product of $\prod_{\text{edges}} \gamma(\text{color},\text{label})^{\pm1}$
where the $\gamma$s are arbitrary parameters and the sign of the exponent depends on the orientation of the edge;
e.g.,
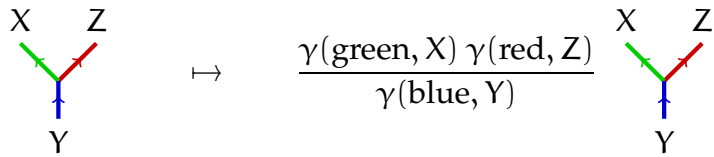
\[
\begin{tikzpicture}[baseline=0.5cm]
\draw[invarrow=0.6,dr] (0,0.5) -- (0.5,1) node[above] {$\m Z$};
\draw[invarrow=0.6,dg] (0,0.5) -- (-0.5,1) node[above] {$\m X$};
\draw[invarrow=0.6,db] (0,0) node[below] {$\m Y$} -- (0,0.5);
\end{tikzpicture}
\qquad \mapsto \qquad
\frac{\gamma(\text{green},X)\ \gamma(\text{red},Z)}
{\gamma(\text{blue},Y)} \ 
\begin{tikzpicture}[baseline=0.5cm]
\draw[invarrow=0.6,dr] (0,0.5) -- (0.5,1) node[above] {$\m Z$};
\draw[invarrow=0.6,dg] (0,0.5) -- (-0.5,1) node[above] {$\m X$};
\draw[invarrow=0.6,db] (0,0) node[below] {$\m Y$} -- (0,0.5);
\end{tikzpicture}
\]
These rescalings preserve properties~\ref{pty:factor}--\ref{pty:norm}.
We now get rid of this freedom by imposing
\begin{equation}\label{eq:norm}
\tikz[baseline=0.34cm,scale=1.5]{\uptri{Y}{YX}{X} } =
\tikz[baseline=0.34cm,scale=1.5]{\uptri{i}{i}{i} } =
1
\qquad \text{likewise for all rotations $\Deltatri$ and $\nablatri$}
\end{equation}
where $YX$ varies over valid multinumbers and $i$ varies over $0,\ldots,d$.
Indeed, we can use the $\Deltatri$ triangles alone (say)
to fix the $\gamma(\text{color},X)$
inductively in $|X|$ (number of digits of $X$), and then
the $\nablatri$ triangle condition follows from property~\ref{pty:norm}
(noting that the weight conservation from property~\ref{pty:ybe}
for the triangle 
\uptri{X}{Y}{Z} implies $\langle \vf_X, \tau \vf_Y \rangle = -1$).
The only remaining gauge freedom in the choice of basis vectors
comes in choosing a weight vector for each weight $\vf_i, \tau \vf_i$,
$i=0,\ldots,d$. These $2(d+1)$ remaining scale parameters can be blamed on our
$2(d+1)$-torus whose weight lattice is $\Lambda^{1+d}$. 

\junk{is it completely obvious, that even equivariantly, this fixes
  the normalizations of weight vectors?}

Finally, one more graphical notation is needed.
Given a positive integer $n$, a permutation $\sigma\in \mathcal S_n$,
and a color $\in\{\text{red},\text{green},\text{blue}\}$, 
we consider a wiring diagram for $\sigma$ made of $n$ lines in that color, 
with parameters $u_1,\ldots,u_n$, e.g.
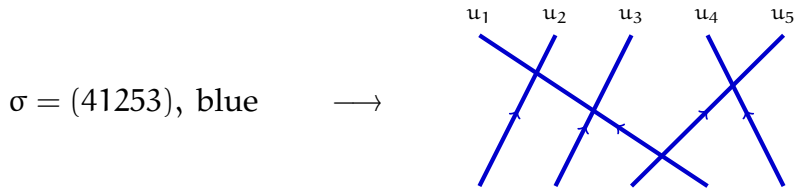
\[
\sigma=(41253),\text{ blue} \qquad\longrightarrow\qquad
\tikz[baseline=1cm]{
\draw[db,invarrow=0.5]
(1,0) -- (2,2) 
node[above,black] {$\ss u_2$};
\draw[db,invarrow=0.4]
(2,0) -- (3,2) 
node[above,black] {$\ss u_3$};
\draw[db,invarrow=0.5]
(3,0) -- (5,2) 
node[above,black] {$\ss u_5$};
\draw[db,invarrow=0.4]
(4,0) -- (1,2) 
node[above,black] {$\ss u_1$};
\draw[db,invarrow=0.5]
(5,0) -- (4,2) 
node[above,black] {$\ss u_4$};
}
\]
This results in an endomorphism of $S^{\tensor n}$ (in the blue case;
or of $SW^{\tensor n},SE^{\tensor n}$ in the red and green cases), which,
thanks to the Yang--Baxter equation \eqref{eq:ybe} and
the unitarity equation \eqref{eq:unit}, 
is independent of the choice of wiring diagram for the permutation.
Here the convention is that the numbering of the spectral parameters is the increasing one at the top,
and $u_{\sigma^{-1}(1)},u_{\sigma^{-1}(2)},\ldots,u_{\sigma^{-1}(n)}$ at the bottom.
In the dual picture, we'll denote the endomorphism by a rectangle
labeled $\sigma$, the orientation of the rectangle determining as
usual the color of the corresponding lines.

Note that if $\sigma\in W$ is a representative of a $n$-string $\sigma'$ in the letters $\{0,\ldots,d\}$ (viewed as an element of $W_P\dom W$),
then any diagram of $\sigma$ ``sorts'' $\sigma'$, i.e.
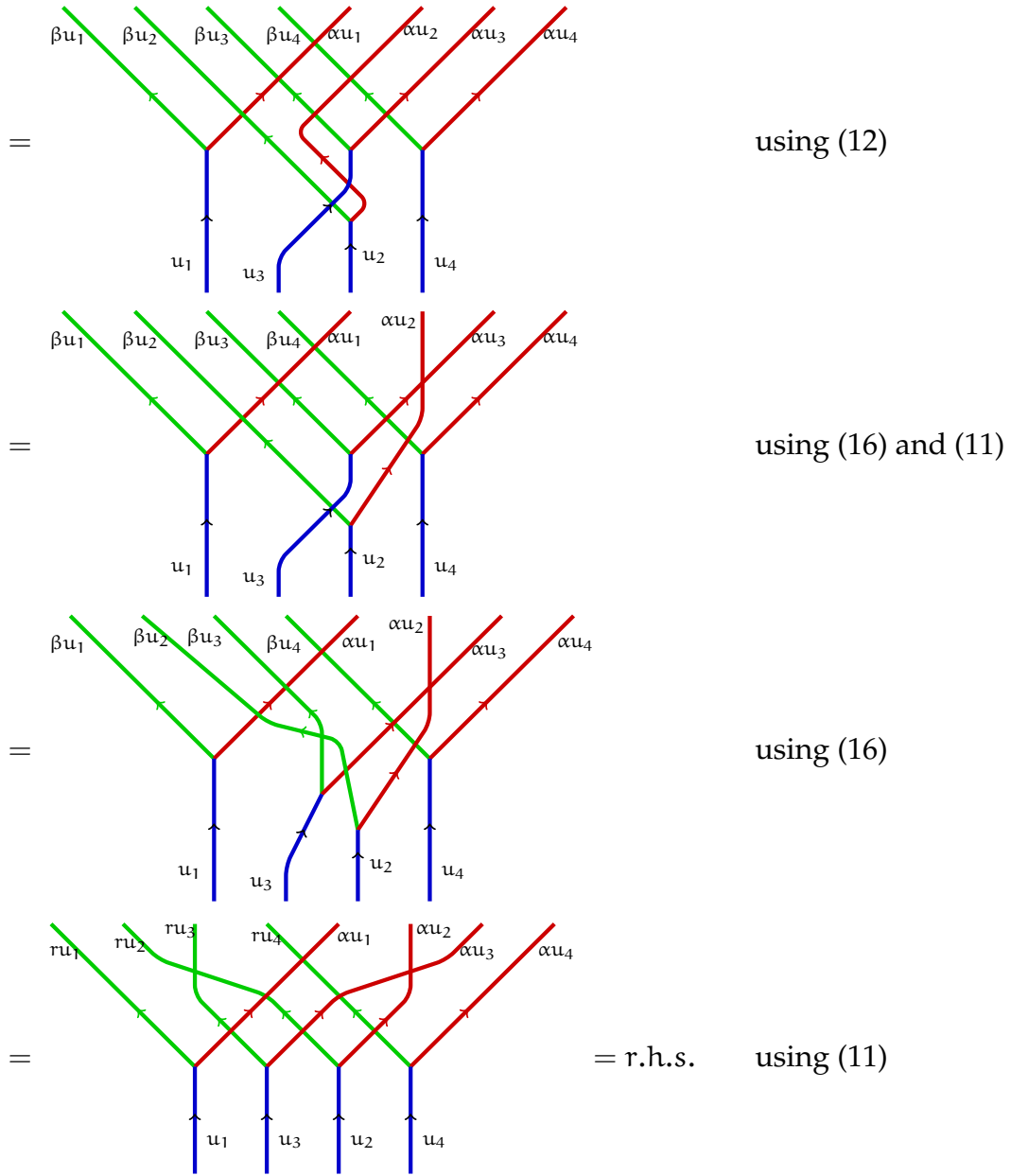
\[
  \tikz[baseline=1cm]{
\node[above] at (0,2) {$\sigma'=$};
\draw[db,invarrow=0.5] (1,0) node[below,black] {$0$} -- (2,2) node[above,black] {$0$};
\draw[db,invarrow=0.4] (2,0) node[below,black] {$1$} -- (3,2) node[above,black] {$1$};
\draw[db,invarrow=0.5] (3,0) node[below,black] {$1$} -- (5,2) node[above,black] {$1$};
\draw[db,invarrow=0.4] (4,0) node[below,black] {$2$} -- (1,2) node[above,black] {$2$};
\draw[db,invarrow=0.5] (5,0) node[below,black] {$2$} -- (4,2) node[above,black] {$2$};
}
\]

Equivariant puzzles of size $n$ are defined as collections of vertical rhombi
plus a row of $n$ $\Deltatri$s at the bottom,
forming together an equilateral triangle of size $n$.
In the dual picture (which will be displayed in the proof of the next proposition), 
the spectral parameters attached to the lines
are chosen to be $u_1,\ldots,u_n$ from left to right at the bottom,
and therefore $\beta u_1,\ldots,\beta u_n$ on the left and $\alpha u_1,\ldots,\alpha u_n$
on the right.

In the rest of this section, we assume implicitly that
properties~\ref{pty:factor}--\ref{pty:norm} are satisfied (in fact, in
the present formalism, the very definition just given of puzzles
relies on property~\ref{pty:factor}), and that the gauge freedom is
fixed by \eqref{eq:norm}.

\subsection{Consequences of the properties}\label{sec:finalproof}

In this proposition we see that the trivalent vertices are the key to
getting a multiplication rule, with one Schubert class being related
to two others.

\begin{prop}\label{prop:key}
Given $\sigma\in \mathcal S_n$, one has
\[
\begin{tikzpicture}[baseline=5mm]
\draw[thick] (0,0) -- (2,0) -- ++(120:2) -- cycle;
\draw[thick] (0,-0.4) -- (2,-0.4); \draw (0,0) -- (0,-0.4); \draw (2,0) -- (2,-0.4); \node at (1,-0.2) {$\sigma$};
\end{tikzpicture}
=
\begin{tikzpicture}[baseline=5mm]
\draw[thick] (0,0) -- (2,0) -- ++(120:2) -- cycle;
\draw[thick] (150:0.4) -- ++(60:2);
\draw (0,0) -- (150:0.4) (60:2) -- ++(150:0.4);
\path (60:1) ++(150:0.2) node[rotate=60] {$\sigma$};
\begin{scope}[xshift=2cm]
\draw[thick] (30:0.4) -- ++(120:2);
\draw (0,0) -- (30:0.4) (120:2) -- ++(30:0.4);
\path (120:1) ++(30:0.2) node[rotate=-60] {$\sigma$};
\end{scope}
\end{tikzpicture}
\]
where the boundaries are fixed to be
three given strings of length $n$.
\end{prop}

\begin{proof}
The proof is essentially obvious since in the dual graphical representation, the various lines connect
the same locations on the boundaries in the l.h.s.\ and the r.h.s. However, since this proposition is crucial,
we shall prove it in detail.

Clearly, we can limit ourselves to the case that $\sigma$ is a simple
transposition. Once dualized, the required series of equalities looks
as follows (where for the purpose of illustration we have set $n=4$):
\begin{align*}
\text{l.h.s.}
&=\begin{tikzpicture}[baseline=0]
\foreach \i in {1,...,4} \draw[dg,invarrow=0.4] (\i,0) -- node[left=-1mm,pos=0.8,black] {$\ss \beta u_\i$} ++(-2,2);
\foreach \i in {1,...,4} \draw[dr,invarrow=0.4] (\i,0) -- node[right=-1mm,pos=0.8,black] {$\ss \alpha u_\i$} ++(2,2);
\draw[db,arrow] (1,0) -- node[left,pos=0.8,black] {$\ss u_1$} (1,-2);
\draw[db,arrow] (4,0) -- node[right,pos=0.8,black] {$\ss u_4$} (4,-2);
\draw[rounded corners,db,arrow=0.4] (2,0) -- (2,-0.5) -- (3,-1.5) -- node[right,pos=0.5,black] {$\ss u_2$} (3,-2);
\draw[rounded corners,db,arrow=0.4] (3,0) -- (3,-0.5) -- (2,-1.5) -- node[left,pos=0.5,black] {$\ss u_3$}(2,-2);
\end{tikzpicture}
\\
&=
\begin{tikzpicture}[baseline=0]
\foreach \i in {1,3,4} \draw[dg,invarrow=0.4] (\i,0) -- node[left=-1mm,pos=0.8,black] {$\ss \beta u_\i$} ++(-2,2);
\foreach \i in {1,3,4} \draw[dr,invarrow=0.4] (\i,0) -- node[right=-1mm,pos=0.8,black] {$\ss \alpha u_\i$} ++(2,2);
\draw[dg,invarrow=0.4] (3,-1) -- (2,0) -- node[left=-1mm,pos=0.8,black] {$\ss \beta u_2$} ++(-2,2);
\draw[dr,invarrow=0.3,rounded corners] (3,-1) -- ++(0.25,0.25) -- ++(-1,1) -- node[right=-1mm,pos=0.8,black] {$\ss \alpha u_2$} ++(1.75,1.75);
\draw[db,arrow] (1,0) -- node[left,pos=0.8,black] {$\ss u_1$} (1,-2);
\draw[db,arrow] (4,0) -- node[right,pos=0.8,black] {$\ss u_4$} (4,-2);
\draw[rounded corners,db,arrow=0.4] (3,-1) -- node[right,pos=0.5,black] {$\ss u_2$} (3,-2);
\draw[rounded corners,db,arrow=0.6] (3,0) -- (3,-0.5) -- (2,-1.5) -- node[left,pos=0.5,black] {$\ss u_3$}(2,-2);
\end{tikzpicture}
&&\text{using \eqref{eq:qtri}}
\\
&=
\begin{tikzpicture}[baseline=0]
\foreach \i in {1,3,4} \draw[dg,invarrow=0.4] (\i,0) -- node[left=-1mm,pos=0.8,black] {$\ss \beta u_\i$} ++(-2,2);
\foreach \i in {1,3,4} \draw[dr,invarrow=0.4] (\i,0) -- node[right=-1mm,pos=0.8,black] {$\ss \alpha u_\i$} ++(2,2);
\draw[dg,invarrow=0.4] (3,-1) -- (2,0) -- node[left=-1mm,pos=0.8,black] {$\ss \beta u_2$} ++(-2,2);
\draw[dr,invarrow=0.3,rounded corners] (3,-1) -- ++(1,1.5) -- node[left=-1mm,pos=0.9,black] {$\ss \alpha u_2$} ++(0,1.5);
\draw[db,arrow] (1,0) -- node[left,pos=0.8,black] {$\ss u_1$} (1,-2);
\draw[db,arrow] (4,0) -- node[right,pos=0.8,black] {$\ss u_4$} (4,-2);
\draw[rounded corners,db,arrow=0.4] (3,-1) -- node[right,pos=0.5,black] {$\ss u_2$} (3,-2);
\draw[rounded corners,db,arrow=0.6] (3,0) -- (3,-0.5) -- (2,-1.5) -- node[left,pos=0.5,black] {$\ss u_3$}(2,-2);
\end{tikzpicture}
&&\text{using \eqref{eq:qtri2} and \eqref{eq:ybe}}
\\
&=
\begin{tikzpicture}[baseline=0]
\foreach \i in {1,4} \draw[dg,invarrow=0.4] (\i,0) -- node[left,pos=0.8,black] {$\ss \beta u_\i$} ++(-2,2);
\foreach \i in {1,4} \draw[dr,invarrow=0.4] (\i,0) -- node[right,pos=0.8,black] {$\ss \alpha u_\i$} ++(2,2);
\draw[dg,invarrow=0.4,rounded corners] (2.5,-0.5) -- ++(0,1) -- node[left,pos=0.8,black] {$\ss \beta u_3$} ++(-1.5,1.5);
\draw[dr,invarrow=0.3] (2.5,-0.5) -- node[right=-1mm,pos=0.8,black] {$\ss \alpha u_3$} ++(2.5,2.5);
\draw[dg,invarrow=0.4,rounded corners] (3,-1) -- ++(-0.25,1.25) -- (1.75,0.5) -- node[left=-2mm,pos=0.8,black] {$\ss \beta u_2$} (0,2);
\draw[dr,invarrow=0.3,rounded corners] (3,-1) -- ++(1,1.5) -- node[left=-1mm,pos=0.9,black] {$\ss \alpha u_2$} ++(0,1.5);
\draw[db,arrow] (1,0) -- node[left,pos=0.8,black] {$\ss u_1$} (1,-2);
\draw[db,arrow] (4,0) -- node[right,pos=0.8,black] {$\ss u_4$} (4,-2);
\draw[rounded corners,db,arrow=0.4] (3,-1) -- node[right,pos=0.5,black] {$\ss u_2$} (3,-2);
\draw[rounded corners,db,arrow=0.4] (2.5,-0.5) -- (2,-1.5) -- node[left,pos=0.5,black] {$\ss u_3$}(2,-2);
\end{tikzpicture}
&&\text{using \eqref{eq:qtri2}}
\\
&=
\begin{tikzpicture}[baseline=0]
\draw[dg,invarrow=0.4] (1,0) -- node[pos=0.8,left=-2mm,black] {$\ss \beta u_1$} ++(-2,2);
\draw[dg,invarrow=0.4,rounded corners] (2,0) -- ++(-1,1) -- node[pos=0.9,left=-2mm,black] {$\ss \beta u_3$} ++(0,1);
\draw[dg,invarrow=0.3,rounded corners] (3,0) -- ++(-1,1) -- ++(-1.5,0.5) -- node[pos=0.4,left=-2mm,black] {$\ss \beta u_2$} ++(-0.5,0.5);
\draw[dg,invarrow=0.4] (4,0) -- node[pos=0.9,left=-2mm,black] {$\ss \beta u_4$} ++(-2,2);
\draw[dr,invarrow=0.4] (1,0) -- node[pos=0.9,right,black] {$\ss \alpha u_1$} ++(2,2);
\draw[dr,invarrow=0.3,rounded corners] (2,0) -- ++(1,1) -- ++(1.5,0.5) -- node[pos=0.2,right=-1mm,black] {$\ss \alpha u_3$} ++(0.5,0.5);
\draw[dr,invarrow=0.4,rounded corners] (3,0) -- ++(1,1) -- node[pos=0.9,right=-1mm,black] {$\ss \alpha u_2$} ++(0,1);
\draw[dr,invarrow=0.4] (4,0) -- node[pos=0.8,right,black] {$\ss \alpha u_4$} ++(2,2);
\foreach \i/\w in {1/1,2/3,3/2,4/4} \draw[db,arrow] (\i,0) -- node[pos=0.7,right] {$\ss u_\w$} (\i,-1.5);
\end{tikzpicture}
=r.h.s.
&&\text{using \eqref{eq:ybe}}
\end{align*}
\end{proof}

It is interesting to note that if we allowed puzzles to have arbitrary
boundary labels (multinumbers, not just single numbers) then the
multiplication they define would not be associative. 
We need some results about the importance of this single-number sector,
based on the following technical lemma.

\begin{lem}\label{lem:rightward}
  Let $X$ be a valid multinumber, and $\vf_X \in \Lambda^{1+d}$ its 
  associated weight. Then the projection of $\tau\vf_X$ to the $0$th factor 
  of $\Lambda$ points weakly to the right, and strictly to the right 
  unless $X$ is $0$, is $3(2(10))$, or lacks $0$ entirely.
\end{lem}

To put this statement in context, we draw the vectors 
$\vf_0, \tau \vf_0, \tau^2 \vf_0 \in \Lambda$:

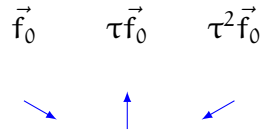
\begin{center}
\begin{tikzpicture}[>=latex]
\matrix[column sep=0.5cm,row sep=0.5cm,cells={scale=2}]
{
\node{$\vf_0$};
&
\node{$\tau \vf_0$};
&
\node{$\tau^2 \vf_0$};
\\
\draw[->,blue] (-0,0.0625) -- ++(-30:0.25);
&
\draw[->,blue] (0,-0.125) -- ++(90:0.25);
&
\draw[->,blue] (0,+0.0625) -- ++(210:0.25);
\\
}; 
\end{tikzpicture}
\end{center}

\begin{proof}
  Multinumbers come in
  six types: $0$, $X0$, $X(Y0)$, $3(2(10))$, $(3(2(10)))0$,
  and those not involving $0$.
  We draw the first five here in their minimal puzzle environments, 
  the better to calculate the contributions (shown here in blue)
  of their $\tau\vf$s 
  to the $0$th factor of $\Lambda$. 
  (The multinumbers not involving $0$ contribute nothing
  in the $0$th factor.) 

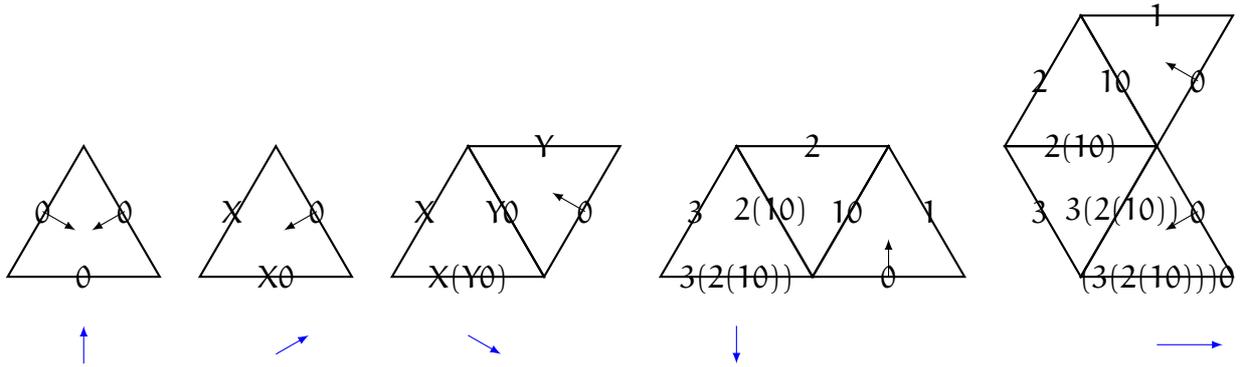
\begin{center}
\begin{tikzpicture}[>=latex]
\matrix[column sep=0.5cm,row sep=0.5cm,cells={scale=2}]
{
\begin{scope}[yshift=-0.866025cm]  \uptri{0}{0}{0}\end{scope}
\draw[->] (NW.center) -- ++(-30:0.25);
  \draw[->] (NE.center) -- ++(210:0.25);
&
\begin{scope}[yshift=-0.866025cm]  \uptri{X}{X0}{0}\end{scope}
\draw[->] (NE.center) -- ++(210:0.25);
&
\begin{scope}[yshift=-0.866025cm]  \uptri{X}{X(Y0)}{}\node at (NE) {$Y0$};\end{scope}
\begin{scope}[xshift=0.5cm]\downtri{0}{Y}{}\end{scope}
  \draw[->] (NW.center) -- ++(150:0.25);
&
\begin{scope}[yshift=-0.866025cm]  \uptri{3}{3(2(10))}{}\node at (NE) {$2(10)$};\end{scope}
\begin{scope}[xshift=0.5cm]\downtri{}{2}{}\end{scope}\node at (NW) {$10$};
  \begin{scope}[xshift=1cm,yshift=-0.866025cm]\uptri{}{0}{1}\end{scope}
  \draw[->] (horiz.center) -- ++(90:0.25);
&
\begin{scope}[yshift=-0.866025cm]  \uptri{}{(3(2(10)))0}{0}\node at (NW) {$3(2(10))$};\end{scope}
\draw[->] (NE.center) -- ++(210:0.25);
  \begin{scope}[xshift=-0.5cm]\downtri{}{2(10)}{3}\end{scope}
  \begin{scope}[xshift=-0.5cm]\uptri{2}{}{}\end{scope}\node at (NE) {$10$};
  \begin{scope}[yshift=0.866025cm]\downtri{0}{1}{}\end{scope}
  \draw[->] (NW.center) -- ++(150:0.25);
\\
\draw[->,blue] (0,-0.125) -- ++(90:0.25);
&
\draw[->,blue] (0,-0.0625) -- ++(30:0.25);
&
\draw[->,blue] (0,0.0625) -- ++(-30:0.25); 
&
\draw[->,blue] (0,0.125) -- ++(-90:0.25); 
&
\draw[->,blue] (0,0) -- ++(0:0.433);
\\
};
\end{tikzpicture}
\end{center}  
\end{proof}
\begin{prop}\label{prop:singlenum}
  Let $P$ be a size $n$ triangle made of puzzle pieces, 
  where the Northwest and Northeast sides have only single-number
  labels $0$, $1$, $2$, $3$.

  If $P$ has the same content on NW and NE,
  then the South side is also single-number labels (with that same content).
  The $120^\circ,240^\circ$ rotations of the statement also hold.
  Conversely, if the contents on the (all single-number) NW and NE sides
  differ, then the S side cannot be all single-number.  
\end{prop}

\begin{proof}
  Let $X(j)$, $j=1,\ldots,n$ be the multinumbers across the South side.
  The Green's theorem argument
  and the assumptions about the NW, NE sides tell us that 
  $\sum_{j=1}^n \vf_{X(j)} \in span(\vf_0,\vf_1,\vf_2,\vf_3)$,
  i.e. not using the $\tau\vf_0,\tau\vf_1,\tau\vf_2,\tau\vf_3$ basis elements.

  This motivates an equivalent, if more general-looking, proposition: 
  if $X(j),$ $j=1,\ldots,n$ is a list of multinumbers such
  that $\sum_{j=1}^n \vf_{X(j)}$ lies in $span(\vf_0,\vf_1,\vf_2,\vf_3)$,
  then the only $X(j)$s involving $0$ actually have $X(j)=0$.
  Once we excise them from the list, subtract $1$ from every
  number in every remaining $X(j)$ (again obtaining valid multinumbers
  and valid puzzle pieces,
  by the functoriality property of multinumbers from \S\ref{ssec:funct}),
  and use induction on $d$, we find that each multinumber $X(j)$ is 
  a single-number.
  

  We consider now the contribution ``$\pi_0(\vf_{X(j)})$'' 
  of each $\vf_{X(j)}$ to the $0$th copy of $\Lambda$ in $\Lambda^{1+d}$. 

  A $d=0$ puzzle has only $0$-labels, so trivially achieves the
  condition we seek, namely that the only multinumbers $X(j)$
  involving $0$s are exactly $0$s.

  In a $d=1$ puzzle, only the labels $0,1,10$ occur.
  Since we know the sum $\sum_{j=1}^n \pi_0(\vf_{X(j)})$
  must be parallel to $\vf_0$ (drawn here as vertical), 
  and by lemma \ref{lem:rightward}
  these vectors $\pi_0(\tau\vf_0),\pi_0(\tau\vf_1)$ point weakly right while 
  $\pi_0(\vf_{10})$ points strictly right, there can be no $\pi_0(\vf_{10})$s.
  For $d=1$ that rules out the $10$ label, leaving only the $0,1$ labels.

  In a $d\leq 2$ puzzle, the only labels including $0$ are $0, X0, 2(10)$,
  and the same rightward-pointing argument applies. Again we learn the
  only $\{X(j)\}$ including $0$ are equal to $0$.

  In a $d=3$ puzzle, this rightward-pointing argument shows only that any label
  involving $0$ must be $0$ or $3(2(10))$. But the same argument applied dually
  to study $3$s in labels shows that any label involving $3$ must be
  $3$ or $((32)1)0$; in particular, not $3(2(10))$.  
  Consequently, neither the $3(2(10))$ nor $((32)1)0$ labels can appear. 
  So once again, any label involving $0$ must be $0$, and we can apply
  the inductive argument from above.

  For the converse, pick $i$ such that there is a different number of
  $i$ labels on the NW and NE sides. Then the projection to the $i$th
  $\Lambda$ factor is not vertical, and cannot be canceled by
  single-number $\vf_X$s from the South side, 
  since each $\vf_i$ is vertical and each $\vf_{j\neq i}$
  projects to $\vec 0$ in the $i$th factor.
\end{proof}

\begin{prop}
  \label{prop:weight}
  \
\begin{enumerate}
\item Single-color crossings preserve single-numbers, i.e.,
\begin{equation}\label{eq:e}
\text{if }
\tikz[baseline=0,xscale=0.5]
{
\draw[invarrow=0.75,d] (-1,-1) node[below] {$\m i$} -- (1,1) node[above] {$\m l$};
\draw[invarrow=0.75,d] (1,-1) node[below] {$\m j$} -- (-1,1) node[above] {$\m k$};
}
\ne0,\text{ then}
\quad
i,j\in \{0,\ldots,d\}
\ 
\Leftrightarrow
\ 
k,l\in \{0,\ldots,d\}
\end{equation}
(where the lines have arbitrary, but identical, colors).
\item 
For $i\in \{0,\ldots,d\}$,
\begin{equation}\label{eq:e2}
\tikz[baseline=0,xscale=0.5]
{
\draw[invarrow=0.75,d] (-1,-1) node[below] {$\m i$} -- (1,1) node[above] {$\m k$};
\draw[invarrow=0.75,d] (1,-1) node[below] {$\m i$} -- (-1,1) node[above] {$\m j$};
}
\ne 0\
\text{ or }
\tikz[baseline=0,xscale=0.5]
{
\draw[invarrow=0.75,d] (-1,-1) node[below] {$\m j$} -- (1,1) node[above] {$\m i$};
\draw[invarrow=0.75,d] (1,-1) node[below] {$\m k$} -- (-1,1) node[above] {$\m i$};
}
\ne0
\quad\Rightarrow\quad
i=j=k
\end{equation}
where again the lines have identical colors.
\item If $\omega$ is a weakly increasing single-number string, then
\begin{equation}\label{eq:empty}
\tikz[scale=1.8,baseline=0.5cm]{\uptri{\lambda}{\omega}{\mu}}\ne 0
\quad\Rightarrow\quad
\lambda=\mu=\omega
\end{equation}
for all single-number strings $\lambda$ and $\mu$, 
i.e., having $\omega$ at the bottom forces it on the other two sides.
Moreover, the puzzle is unique, and along each diagonal (NW/SE or NE/SW)
the labels on the non-horizontal edges are constant.

Similarly,
if $\agemo$ is a weakly {\em decreasing} single-number string, then
\begin{equation}\label{eq:emptyb}
\tikz[scale=1.8,baseline=-0.8cm]{\downtri{\lambda}{\agemo}{\mu}}\ne 0
\quad\Rightarrow\quad
\lambda=\mu=\agemo
\end{equation}
for all single-number strings $\lambda$ and $\mu$.
\end{enumerate}
\end{prop}

The first two of these will derive from the same source:
\begin{lem}
  Let $V$ be a $T$-representation with weights $P$ and weight 
  basis $(\vec v_p)$, and $F$ a face of the convex hull of $P$.
  Let $X \in Hom(V\tensor V,V\tensor V)$ be $T$-equivariant. 
  If $\vec v_p,\vec v_q, \vec v_r, \vec v_s$ are
  basis vectors with weights $p,q\in F$,
  and $X_{\vec v_p \vec v_q}^{\vec v_r \vec v_s} \neq 0$
  or $X^{\vec v_p \vec v_q}_{\vec v_r \vec v_s} \neq 0$, 
  then $r,s \in F$ as well.
\end{lem}

We will of course be applying this when $X$ is an $\check R$-matrix.
On that topic, it is regrettable that ``$R$-matrices'' were not
named ``$X$-matrices'' for their tetravalency;
at least ``$K$-matrices'' (not relevant in this paper, but used in our
\cite{artic73}) were named appropriately.

\begin{proof}
  Since $F$ is a face of the convex hull, it is defined by the
  tightness of some inequality $\langle \vec a, p \rangle \geq c$
  for some $\vec a\in \lie{t}$ generating some one-parameter subgroup $A\leq T$.
  Equivalently, all the $A$-weights in $V$ are $\geq c$, and 
  those on $F$ span the $c$ weight space. 

  Then the only basis vectors $\vec v_m \tensor \vec v_n$ 
  of $A$-weight $2c$ are those with $m,n \in F$, 
  and $X$ must preserve this $A$-weight space of $V\tensor V$,
  giving the claim.  
\end{proof}

\begin{proof}[Proof of proposition~\ref{prop:weight}.]
  We'll apply this lemma when $V$ is a representation $V^{X_{2d}}_\lambda$ of 
  $X_{2d} = A_2,D_4,E_6$, and $P = W_{X_{2d}}\cdot \lambda$;
  our $\lambda$ will be such that its projection to the $A_d$ weight lattice
  is the fundamental weight $\omega_1$.
  Hence the $A_d$-subrepresentation containing the
  high weight vector has the standard basis indexed by $\{0,\ldots,d\}$.
  To get part (1), we take $F$ to be the face containing exactly
  the vertices $\{0,\ldots,d\}$, 
  and use theorem \ref{thm:classif}(\ref{item:AdFace}).
  Shrinking $F$ to its $i$th vertex (then applying the lemma) gets us part (2).
  
  For part (3),
  first observe that the Green's theorem argument shows that the content
  on the three sides must be the same. For convenience we assume that
  the lowest number occurring is $0$, since if it is $i>0$ we can 
  simply subtract $i$ from all numbers appearing and work instead with
  that equally valid puzzle. 

  By the weak increase across the South side, the leftmost label
  is smallest, hence $0$. There are no $\uptri{i}{0}{0i}$ pieces
  for $i>0$ (since $|\vf_{0i}|^2 = 2+2a \neq 2$),
  so the leftmost triangle on the South side must be a $\uptri{0}{0}{0}$.

  Let $j$ be the leftmost label on the Northeast side, and let $P'$
  be the size $n-1$ triangle made by removing the pieces touching
  the Northwest side. (We don't know yet that the newly exposed Northwest 
  labels are single numbers, i.e., that $P'$ is itself a puzzle.)

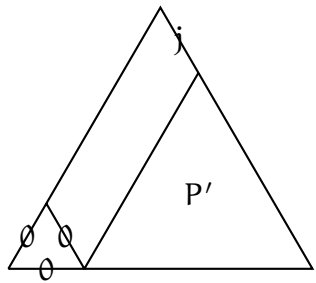
\begin{center}
\begin{tikzpicture}
\draw[thick,black] (0,0) -- node[edgelabel] {$0$} ++(1,0) -- ++(3,0)
-- ++(120:3) -- node[edgelabel] {$j$} ++(120:1) -- ++(240:3) -- node[edgelabel] {$0$} ++(240:1)
(1,0) -- ++(60:3) (1,0) -- node[edgelabel] {$0$} ++(120:1);
\node at (2.5,1) {$P'$};
\end{tikzpicture}
\end{center}

  Assume for contradiction that $j>0$. Rotate $P$ $120^\circ$ counterclockwise
  for ease of comparison to proposition~\ref{prop:singlenum}.
  The NW side of the rotated $P'$ has one more $0$ than the NE side,
  so the vector sum over the NW, NE boundary labels (in the projection
  to the $0$th factor $\Lambda$ of $\Lambda^{1+d}$) points
  strictly rightward. Now we add in the vectors from the South
  boundary labels of $P'$ (as determined in lemma~\ref{lem:rightward}), 
  each of which points weakly rightward,
  and we reach a contradiction with the Green's theorem argument.

  Now that we've established $j=0$, the content of the NE and S side of
  (unrotated) $P'$ 
  is the same, and we can apply proposition~\ref{prop:singlenum} to $P'$.
  We learn that its Northwest side is all single numbers.
  Of course its South side is weakly increasing, so by induction on $n$
  the present proposition applies, i.e. $P'$ is uniquely determined
  by its South side, has fugacity $1$, and its NW side matches its
  South side (both read left to right).

  It remains to determine the pieces in the strip $P\setminus P'$,
  which we do one rhombus at a time from SW to NE,
  having already determined the bottom $\uptri{0}{0}{0}$.
  We will show inductively that each rhombus is of the form 
  $\rh{0}{i}{0}{i}$. A priori, the rhombus bears $\rh{0}{i}{}{j}$
  where $i,j$ are single numbers, so the weight on the NE side is $\vf_{(i0)j}$
  (whether or not that's a valid multinumber). Its norm-square is
  \begin{eqnarray*}
    |\vf_{(i0)j}|^2
    &=& |\vf_{i0}|^2 + |\vf_j|^2 + 2\langle \vf_{i0}, \tau \vf_j \rangle \\
    &=& 4 + 2\langle - \tau \vf_i - \tau^2 \vf_0, \tau \vf_j \rangle 
    = 4 
    - 2\langle \tau \vf_i, \tau \vf_j \rangle 
    - 2\langle  \vf_j, \tau \vf_0 \rangle 
    = 6 - 2
    \begin{cases}
      2 &\text{if $i=j$} \\
      2-a &\text{if $i\neq j$} 
    \end{cases}
  \end{eqnarray*}
  which is $2$ iff $i=j$, in which case (by Green's theorem) the
  NE label is $0$, continuing the induction.  
  Hence the Northwest labels on $P'$ are
  copied across these rhombi to the Northwest side of $P$, after the
  initial $0$, and the fugacities are $1$ as claimed.
\end{proof}


\subsection{Main theorems}\label{ssec:mainthms}
Given a permutation $\sigma\in \mathcal S_n$ and a single-number string $\lambda$ of length $n$, define
\begin{equation}\label{eq:defS}
\mathbf{S}^\lambda_{\sss N}|_\sigma
:=
\begin{tikzpicture}[baseline=0]
\draw[thick] (-1,0.4) -- (1,0.4) (-1,-0.4) -- (1,-0.4); \draw (-1,-0.4) -- (-1,0.4) (1,-0.4) -- (1,0.4); \node at (0,0) {$\sigma$};
\node at (0,0.4) {$\m\lambda$};
\node at (0,-0.4) {$\m\omega$};
\end{tikzpicture}
\qquad
\mathbf{S}^\lambda_{\sss NW}|_\sigma
:=
\begin{tikzpicture}[baseline=0,rotate=60]
\draw[thick] (-1,0.4) -- (1,0.4) (-1,-0.4) -- (1,-0.4); \draw (-1,-0.4) -- (-1,0.4) (1,-0.4) -- (1,0.4); \node[transform shape] at (0,0) {$\sigma$};
\node at (0,0.45) {$\m\lambda$};
\node at (0,-0.4) {$\m\omega$};
\end{tikzpicture}
\qquad
\mathbf{S}^\lambda_{\sss NE}|_\sigma
:=
\begin{tikzpicture}[baseline=0,rotate=-60]
\draw[thick] (-1,0.4) -- (1,0.4) (-1,-0.4) -- (1,-0.4); \draw (-1,-0.4) -- (-1,0.4) (1,-0.4) -- (1,0.4); \node[transform shape] at (0,0) {$\sigma$};
\node at (0,0.45) {$\m\lambda$};
\node at (0,-0.4) {$\m\omega$};
\end{tikzpicture}
\end{equation}
where 
$\omega$ is the unique weakly increasing string with the same content as $\lambda$.

This leads us to the heart of the paper:
\begin{thm}\label{thm:final}
  The following equality holds:
  \[
  \sum_\nu   \tikz[scale=1.8,baseline=0.5cm]{\uptri{\lambda}{\nu}{\mu}}
  \mathbf{S}^\nu_{\sss N}|_\sigma
  =   \mathbf{S}^\lambda_{\sss NW}|_\sigma  \,  \mathbf{S}^\mu_{\sss NE}|_\sigma
  \]
  where the summation is over single-number strings with the same content as
  $\lambda$ and $\mu$.
\end{thm}

\begin{proof} 
The proof can be summarized as the following series of equalities:
\[
\sum_\nu
\begin{tikzpicture}[baseline=5mm]
\draw[thick] (0,0) -- node {$\nu$} (2,0) -- node[xshift=1mm] {$\mu$} ++(120:2) -- node[xshift=-1mm] {$\lambda$} cycle;
\draw[thick] (0,-0.5) -- node {$\omega$} (2,-0.5); \draw (0,0) -- (0,-0.5); \draw (2,0) -- (2,-0.5); \node at (1,-0.25) {$\sigma$};
\end{tikzpicture}
\ =\
\begin{tikzpicture}[baseline=5mm]
\draw[thick] (0,0) -- node {} (2,0) -- node[xshift=1mm] {$\mu$} ++(120:2) -- node[xshift=-1mm] {$\lambda$} cycle;
\draw[thick] (0,-0.5) -- node {$\omega$} (2,-0.5); \draw (0,0) -- (0,-0.5); \draw (2,0) -- (2,-0.5); \node at (1,-0.25) {$\sigma$};
\end{tikzpicture}
\ =\
\begin{tikzpicture}[baseline=5mm]
\draw[thick] (0,0) -- node {$\omega$} (2,0) -- node[xshift=-0.5mm] {} ++(120:2) -- node[xshift=0.5mm] {} cycle;
\draw[thick] (150:0.5) -- node[xshift=-1mm] {$\lambda$} ++(60:2);
\draw (0,0) -- (150:0.5) (60:2) -- ++(150:0.5);
\path (60:1) ++(150:0.25) node[rotate=60] {$\sigma$};
\begin{scope}[xshift=2cm]
\draw[thick] (30:0.5) -- node[xshift=1mm] {$\mu$} ++(120:2);
\draw (0,0) -- (30:0.5) (120:2) -- ++(30:0.5);
\path (120:1) ++(30:0.25) node[rotate=-60] {$\sigma$};
\end{scope}
\end{tikzpicture}
\ =\
\sum_{\tilde\lambda,\tilde\mu}
\begin{tikzpicture}[baseline=5mm]
\draw[thick] (0,0) -- node {$\omega$} (2,0) -- node[xshift=-0.5mm] {$\tilde\mu$} ++(120:2) -- node[xshift=0.5mm] {$\tilde\lambda$} cycle;
\draw[thick] (150:0.5) -- node[xshift=-1mm] {$\lambda$} ++(60:2);
\draw (0,0) -- (150:0.5) (60:2) -- ++(150:0.5);
\path (60:1) ++(150:0.25) node[rotate=60] {$\sigma$};
\begin{scope}[xshift=2cm]
\draw[thick] (30:0.5) -- node[xshift=1mm] {$\mu$} ++(120:2);
\draw (0,0) -- (30:0.5) (120:2) -- ++(30:0.5);
\path (120:1) ++(30:0.25) node[rotate=-60] {$\sigma$};
\end{scope}
\end{tikzpicture}
\ =\
\begin{tikzpicture}[baseline=5mm]
\draw[thick]  (2,0) -- node[xshift=-0.5mm] {$\omega$} ++(120:2) -- node[xshift=0.5mm] {$\omega$} (0,0);
\draw[thick] (150:0.5) -- node[xshift=-1mm] {$\lambda$} ++(60:2);
\draw (0,0) -- (150:0.5) (60:2) -- ++(150:0.5);
\path (60:1) ++(150:0.25) node[rotate=60] {$\sigma$};
\begin{scope}[xshift=2cm]
\draw[thick] (30:0.5) -- node[xshift=1mm] {$\mu$} ++(120:2);
\draw (0,0) -- (30:0.5) (120:2) -- ++(30:0.5);
\path (120:1) ++(30:0.25) node[rotate=-60] {$\sigma$};
\end{scope}
\end{tikzpicture}
\]
where $\omega$ is the unique
weakly increasing string with the same content as $\lambda$ and $\mu$.

Start from the leftmost picture, which is by definition the l.h.s.\ of theorem~\ref{thm:final},
and where the summation is again over single-number strings.
By the first part of proposition~\ref{prop:weight}, the restriction on the summation can be lifted
(any non single-number label will lead to a zero fugacity). As explained in \S\ref{ssec:tensors}
(cf~\eqref{eq:labeledges}), the summation is then implicit in the second picture.

The second equality is simply
proposition~\ref{prop:key} in which we have set the NW side to $\lambda$,
the NE side to $\mu$ and the bottom side to $\omega$.

The third equality, just like the first, is an application
of \eqref{eq:labeledges};
by the first part of proposition~\ref{prop:weight}, 
the internal labels $\tilde\lambda$, $\tilde\mu$
are single-number strings (or more precisely, all other contributions 
to the sum vanish).

The third part of proposition~\ref{prop:weight} then tells us that
$\tilde\lambda=\tilde\mu=\omega$, and that the triangle of the r.h.s.
is made of triangles and rhombi which are all of fugacity $1$
according to property~\ref{pty:norm} and gauge fixing condition
\eqref{eq:norm}.  Once rid of this triangle, we recognize on both
sides the definitions \eqref{eq:defS}, hence the result.
\end{proof}

\junk{
This equality is clearly closely related to the one in proposition~\ref{prop:key},
but the proof will require some additional technical lemmas.
}

One can redo the whole reasoning after ``reversal of
all the arrows''. More precisely, we note that property~\ref{pty:factor}
defines two types of trivalent vertices, one where two edges are incoming,
as was used so far, and one where two edges are outgoing, which we shall use 
now. In property~\ref{pty:ybe}, equations~\eqref{eq:ybe}, \eqref{eq:unit} and
\eqref{eq:equal} are invariant by reversal of arrows, whereas
eq.~\eqref{eq:qtri} turns into eq.~\eqref{eq:qtrirev}.

In the puzzle picture, it is customary to rotate the picture by $180^\circ$
so that lines are still oriented downwards; we then have an obvious
analogue of proposition~\ref{prop:key}, namely
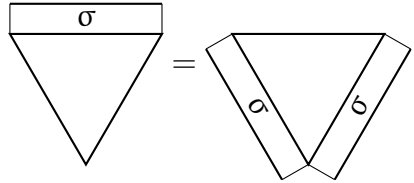
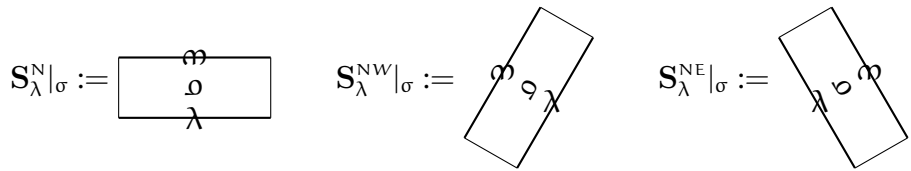
\[
\begin{tikzpicture}[baseline=-5mm,scale=-1]
\draw[thick] (0,0) -- (2,0) -- ++(120:2) -- cycle;
\draw[thick] (0,-0.4) -- (2,-0.4); \draw (0,0) -- (0,-0.4); \draw (2,0) -- (2,-0.4); \node at (1,-0.2) {$\sigma$};
\end{tikzpicture}
=
\begin{tikzpicture}[baseline=-5mm,scale=-1]
\draw[thick] (0,0) -- (2,0) -- ++(120:2) -- cycle;
\draw[thick] (150:0.4) -- ++(60:2);
\draw (0,0) -- (150:0.4) (60:2) -- ++(150:0.4);
\path (60:1) ++(150:0.2) node[rotate=60] {$\sigma$};
\begin{scope}[xshift=2cm]
\draw[thick] (30:0.4) -- ++(120:2);
\draw (0,0) -- (30:0.4) (120:2) -- ++(30:0.4);
\path (120:1) ++(30:0.2) node[rotate=-60] {$\sigma$};
\end{scope}
\end{tikzpicture}
\]

We furthermore define
\[
\mathbf{S}_\lambda^{\sss N}|_\sigma
:=
\begin{tikzpicture}[baseline=0]
\draw[thick] (-1,0.4) -- (1,0.4) (-1,-0.4) -- (1,-0.4); \draw (-1,-0.4) -- (-1,0.4) (1,-0.4) -- (1,0.4); \node at (0,0) {$\hat\sigma$};
\node[rotate=180] at (0,0.4) {$\m\omega$};
\node at (0,-0.4) {$\m\lambda$};
\end{tikzpicture}
\qquad
\mathbf{S}_\lambda^{\sss NW}|_\sigma
:=
\begin{tikzpicture}[baseline=0,rotate=60]
\draw[thick] (-1,0.4) -- (1,0.4) (-1,-0.4) -- (1,-0.4); \draw (-1,-0.4) -- (-1,0.4) (1,-0.4) -- (1,0.4); \node[transform shape] at (0,0) {$\hat\sigma$};
\node[rotate=180] at (0,0.4) {$\m\omega$};
\node at (0,-0.35) {$\m\lambda$};
\end{tikzpicture}
\qquad
\mathbf{S}_\lambda^{\sss NE}|_\sigma
:=
\begin{tikzpicture}[baseline=0,rotate=-60]
\draw[thick] (-1,0.4) -- (1,0.4) (-1,-0.4) -- (1,-0.4); \draw (-1,-0.4) -- (-1,0.4) (1,-0.4) -- (1,0.4); \node[transform shape] at (0,0) {$\hat\sigma$};
\node[rotate=180] at (0,0.4) {$\m\omega$};
\node at (0,-0.35) {$\m\lambda$};
\end{tikzpicture}
\]
where $\lambda$ is read from left to right, 
$\agemo$ is the weakly decreasing
string with the same content as $\lambda$ (i.e., $\omega$ in reverse), 
and $\hat\sigma(i)=\sigma^{-1}(n+1-i)$.
The order of spectral parameters $u_1,\ldots,u_n$ is the increasing one at the bottom,
and therefore $u_{\hat\sigma(1)},\ldots,u_{\hat\sigma(n)}$ at the top.

Using the exact same reasoning as above (noting in particular
the $180^\circ$ symmetry of proposition~\ref{prop:weight}), 
we have the ``dual'' statement:
\begin{thm}\label{thm:dual}
The following equality holds:
\[
\sum_\nu 
\tikz[scale=1.8,baseline=-0.8cm]{\downtri{\,\lambda}{\nu}{\mu}}
\mathbf{S}_\nu^{\sss N}|_\sigma
=
\mathbf{S}_\lambda^{\sss NW}|_\sigma
\,
\mathbf{S}_\mu^{\sss NE}|_\sigma
\]
(where all labels are read as usual from left to right).
\end{thm}

This is to be compared with the primed theorems of \cite{artic68}.
With not much more work, 
we could also obtain analogues of the double primed theorems of \cite{artic68},
but we shall refrain from doing so here.

In the rest of this section, the strategy is to use the
representation theory of quantized affine algebras to define systems of
fugacities which satisfy the properties stated in \S\ref{sec:genproof}. 
The fugacities of rhombi obtained this way will be more general than
needed for the purposes of this paper, and we shall take a limit
($q\to0$ where $q$ is a parameter in these fugacities) to recover the
puzzle rule that we are trying to prove.
(We will interpret geometrically this more general solution
in \cite{artic80}.)

\subsection{$d=1$ and $A_2$}\label{sec:d1}
The case $d=1$ was already investigated in detail in \cite{artic68} (and will be discussed
in a more general setting in \cite{artic80}), so we shall only
sketch it here very informally. We consider the quantized affine algebra $U_q(\mathfrak{a}_2^{(1)})$,
and attach to green and red (resp.\ blue) lines evaluation representations based on
the fundamental representation
$V_{\omega_1}$ (resp.\ the dual fundamental representation $V_{\omega_2}$), of dimension $3$.
All our tensors will then be $U_q(\mathfrak{a}_2^{(1)})$ intertwiners.

As explained in \S\ref{sec:introd1},
in order to take into account the action of $\tau$, we use the following orientation-dependent
encoding of labels of puzzles as weight vectors:
\[
\begin{array}{rlrlrlcccc}
 \multicolumn{2}{c}{\NE}    &\multicolumn{2}{c}{\horiz} & \NW    \\
   \m1  &    &  \m0  &{\color{green!75!black}\bullet}  &  \m{10} &{\color{red!75!black}\bar\bullet}{\color{green!75!black}\bullet} & = & (1 & 0 & 0)\\
 \m0     & {\color{green!75!black}\bar\bullet}&  \m{10} & &  \m1 &  {\color{red!75!black}\bar\bullet}  & = & (0 & 1 & 0)\\
\m{10} & {\color{red!75!black}\bullet}{\color{green!75!black}\bar\bullet}  & \m1  &{\color{red!75!black}\bullet}    &  \m0 &    & = & (0 & 0 & 1)\\
\end{array}
\qquad
\begin{tikzpicture}[>=latex,math mode,baseline=0]
\tikzset{mynode/.style={circle,draw=black,fill=black,inner sep=1pt,outer sep=0pt}}
\coordinate (O) at (0,0) {};
\draw[dotted] (O) -- (90:1) node[mynode,label={[\mcol]right:1}] (1) {};
\draw[dotted] (O) -- (210:1) node[mynode,label={[\mcol]left:0}] (0) {};
\draw[dotted] (O) -- (330:1) node[mynode,label={[\mcol]right:10}] (10) {};
\draw[->] (1) -- node[left,rt] {a} (0);
\draw[->] (0) -- node[below,rt] {a'} (10);
\end{tikzpicture}
\]
(for convenience we also mention the color coding used in \cite{artic68}; the bar indicates
that lines of the corresponding color 
are moving in a direction opposite to the arrows used in this paper).

For weight purposes, it is convenient to temporarily revert the orientation of blue lines,
resulting in a more $\ZZ_3$-invariant setting. The $U$ and $D$ triangles are then
$U_q(\mathfrak{a}_2^{(1)})$ invariants
of a tensor product of 3 evaluation representations of fundamental representations;
the latter, when they
exist (i.e., for specific choices of spectral parameters, as in property~\ref{pty:factor}), 
are simply the 
$U_q(\mathfrak{a}_2)$ antisymmetrizer which $q$-deforms the usual fully skew-symmetric
tensor of $\mathfrak{sl}_3$. Looking up the weight table above, it is not hard to see
that to each permutations of 3 elements corresponds
\[
\begin{matrix} 
123
&
132
&
213
&
231
&
312
&
321\\
\downtri{1}{10}{0}
& \downtri{1}{1}{1}
&\downtri{0}{0}{0}
&\downtri{0}{1}{10}
&\downtri{10}{0}{1}
& \downtri{10}{10}{10}
\end{matrix}
\]
(and similarly for up-pointing triangles); that is,
one of the 5 usual puzzle pieces plus the $K$-piece made of $10$s, as advertised in \S\ref{sec:introd1}. Their fugacity is $-q$ to the power the inversion number of the permutation.
This does not seem to match with the desired fugacity of $1$ for the first 5 pieces;
however, rescaling the SW and SE weight vectors of $10$ by $-q$ and the N weight vector by $-q^{-1}$,
every fugacity becomes up to overall normalization $1$ except that of the $K$-piece which becomes
$-q$. The same reasoning applies to up-pointing triangle with $q\leftrightarrow q^{-1}$.

We are not quite ready yet to take the limit $q\to 0$, because the fugacity of an up-pointing
triangle would diverge. We now use at last the results of \S\ref{ssec:Bmatrix}. 
Multiply the fugacity of every \nablatri\ by $q$ to the power its inversion number (in the
sense of the $B$-matrix \eqref{eq:defB}),
that of every \Deltatri\ by $q$ to the power minus its inversion number.
Essentially,
what lemma~\ref{lem:sympd1} tells us is that this is a reasonable operation in the sense that
its effect can be pushed to the boundary of the triangle. After such a transformation,
the fugacity of every triangle is $1$ except that of \uptri{10}{10}{10} which is $-1$ and
that of \downtri{10}{10}{10} which is $-q^2$. We can now send
$q$ to $0$, leading to the puzzle rule for
the (nonequivariant) $K$-theory of Grassmannians.

More generally, using as 4-valent vertices the natural $R$-matrices (i.e., intertwiners) of
$U_q(\mathfrak{a}_2^{(1)})$ and taking the limit $q\to 0$ appropriately, we recover the
$R$-matrices as defined in \cite{artic68}, and therefore the puzzle
rule for the equivariant $K$-theory of Grassmannians as formulated there.
We shall omit the proof of this claim, 
since it would be a subset of the proof of the $d=2$ case to come. 
See also appendix~\ref{app:cyclic}
for a review of the approach of \cite{artic46,artic68} in the setting of the present paper.

\subsection{$d=2$ and triality in $D_4$}\label{sec:d2}
\subsubsection{$R$-matrices}\label{ssec:Rd2}
To each line (or color of line) 
we now attach a representation of the quantized affine
algebra $U_q(\mathfrak{d}_4^{(1)})$. There are three colors of lines, 
which correspond to the three fundamental representations of $U_q(\mathfrak{d}_4)$ 
that are related by triality (i.e., related to the three nodes $\rt{a'}$, $\rt{b}$, $\rt{b'}$
in the sub-Dynkin diagram $D_4$ in the proof of theorem~\ref{thm:classif}).
We assign to green (resp.\ red, blue)
lines the representation $V_{\omega_1}$ (resp.\ $V_{\omega_2}$, $V_{\omega_3}$)
with weights given by lemma~\ref{lem:greend2} (resp.\ with $\tau^2$, $-\tau$
times these weights). It is known (cf \cite{Hern-affin} and references therein;
see also \cite[p399]{CP-book} for the present case of fundamental representations)
that all three representations can then be extended (affinized) 
to evaluation representations $V_{\omega_i}(u)$ of $U_q(\mathfrak{d}_4^{(1)})$,
where $u$ is the spectral (evaluation) parameter.
Note that these $U_q(\mathfrak{d}_4)$ representations are self-dual, so that
their $U_q(\mathfrak{d}_4^{(1)})$ counterparts only differ from their dual by a shift
of the spectral parameter. \rem[gray]{PZJ restored the $-$ in $-\tau$ because it's true 
more generally}

Next we define $R$-matrices\footnote{The reason our $R$-matrices
  are denoted $\check R$ is to differentiate them from the traditional $R$-matrices $R=P\check R$,
  which are related to each other by permutation $P$ of the factors of the tensor product.} 
(which correspond to crossings of lines of various colors): $\check R_{i,j}$
will be an $U_q(\mathfrak{d}_4^{(1)})$-intertwiner from $V_{\omega_i}(u')\otimes V_{\omega_j}(u'')$ to $V_{\omega_j}(u'')\otimes V_{\omega_i}(u')$.
Let us for example consider representations $V_{\omega_1}(u')$ and $V_{\omega_2}(u'')$,
i.e., red and green lines respectively. 
Then the corresponding $R$-matrix can be written as
\begin{equation}\label{eq:RD4}
\check R_{1,2}
=
\frac{1}{u'-q^2u''}
\left(
(u''-q^4u')
P_{\omega_3}
+
(u'-q^4u'')
P_{\omega_1+\omega_2}
\right)
\end{equation}
where $q$ is another indeterminate. 
Here $P_{\omega_3}$ and $P_{\omega_1+\omega_2}$ are operators of rank $8$ and $56$ respectively, implementing
the two channels of decomposition $V_{\omega_1}\otimes V_{\omega_2}\cong V_{\omega_3}\oplus V_{\omega_1+\omega_2}$ 
for the horizontal subalgebra $U_q(\mathfrak{d}_4)$
(here ``horizontal'' means that its action is independent of the spectral parameter). 
Note that we do not require $\check R_{1,2}$ to have the normalization of the universal
$R$-matrix, only that it commute with the $U_q(\mathfrak{d}_4^{(1)})$ action; 
its normalization is chosen for convenience. \junk{normalization of $P_i$s...}

\junk{The usual formulation of the $R$-matrix is
$R_{3,1}=\mathcal P_{1,3} \check R_{3,1}$, where $\mathcal P_{1,3}$ is the operator
from
$V_{\omega_1}\otimes V_{\omega_3}$
to
$V_{\omega_3}\otimes V_{\omega_1}$
that switches factors of the tensor product.}

In terms of puzzles, recall that $\check R_{1,2}$
parametrizes the fugacities of rhombi, cf eq.~\eqref{eq:rhombus}:
\[
\tikz[baseline=0,scale=1.5]{
\rh{X}{Y}{Z}{W}
}
=
\tikz[baseline=0,xscale=0.5]
{
\draw[invarrow=0.75,dr] (-1,-1) node[below] {$\m X$} -- node[pos=0.75,right,black] {$\ss u''$} (1,1) node[above] {$\m Z$};
\draw[invarrow=0.75,dg] (1,-1) node[below] {$\m Y$} -- node[pos=0.75,left,black] {$\ss u'$} (-1,1) node[above] {$\m W$};
}
=
\check R_{1,2}{}_{XY}^{WZ}
\]

Next, we observe that the $R$-matrix
\eqref{eq:RD4} satisfies 
\begin{equation}\label{eq:UD4}
\check R_{1,2}(u'=q^{2}u,\ u''=q^{-2} u)=-\frac{(1+q^4)(1+q^2)}{q^2} P_{\omega_3}=DU
\end{equation}
\junk{the ugly prefactor is related to the normalization of the $P$s}
where $U: V_{\omega_1}(q^{2}u)\otimes V_{\omega_2}(q^{-2}u) \to V_{\omega_3}(u)$ and $D: V_{\omega_3}(u)\to V_{\omega_2}(q^{-2}u)\otimes V_{\omega_1}(q^{2}u)$ commute with the $U_q(\mathfrak{d}_4^{(1)})$ action. 
This is the analogue of property~\ref{pty:factor}, with $\alpha=q^{-2}$, $\beta=q^2$; that is, we can write
\[
\uptri{X}{Y}{Z}=
\begin{tikzpicture}[baseline=0.5cm]
\draw[invarrow=0.6,dr] (0,0.5) --  node[pos=0.6,right,black] {$\ss \alpha u$} (0.5,1) node[above] {$\m Z$};
\draw[invarrow=0.6,dg] (0,0.5) -- node[pos=0.6,left,black] {$\ss \beta u$} (-0.5,1) node[above] {$\m X$};
\draw[invarrow=0.6,db] (0,0) node[below] {$\m Y$} -- node[pos=0.6,left,black] {$\ss u$} (0,0.5);
\end{tikzpicture}
=
U^{XZ}_Y
\qquad
\downtri{X}{Y}{Z}=
\begin{tikzpicture}[baseline=-0.5cm]
\draw[invarrow=0.6,dr] (-0.5,-1) node[below] {$\m Z$} -- node[pos=0.6,left,black] {$\ss \alpha u$} (0,-0.5);
\draw[invarrow=0.6,dg] (0.5,-1) node[below] {$\m X$} -- node[pos=0.6,right,black] {$\ss \beta u$} (0,-0.5);
\draw[invarrow=0.6,db] (0,-0.5) -- node[pos=0.6,left,black] {$\ss u$} (0,0) node[above] {$\m Y$};
\end{tikzpicture}
=
D^Y_{ZX}
\]

$\check R_{1,2}$, as well as the matrices $U$ and $D$, are given in
appendix~\ref{app:Rd2}.  In particular we can check explicitly that
property~\ref{pty:norm} is satisfied, as well as \eqref{eq:norm}.

The definitions of $\check R_{2,3}, \check R_{3,1}$ are similar\junk{\footnote{%
  It may feel like we should be able to use $\tau$ to define two
  of these from the third. But Schubert calculus only has a 
  natural $\ZZ/3$-invariance in ordinary cohomology, or
  (as in \cite[\S 8]{BuchKGr}) in $K$-theory of Grassmannians,
  so we should not expect this $\ZZ/3$-invariance in $K$-theoretic
  Schubert calculus of $2$-step flag manifolds.
  {this comment is BS. the normalizations are different even in the
    rational case ($H$)}} }
to that of $\check R_{1,2}$, except we choose the normalization
differently:
\begin{equation}\label{eq:Rnorm}
\check R_{i,i+1}
=
\frac{u''-q^4u'}{u'-q^4u''}
P_{\omega_{i+2}}
+
P_{\omega_i+\omega_{i+1}}
,\qquad
i=2,3
\end{equation}
(with indices understood mod $3$).
Next, we consider $\check R_{i+1,i}$, 
$i=1,2,3$;
to ensure that \eqref{eq:unit} is automatically satisfied 
for distinct colors of lines, we define them as
\[
\check R_{i+1,i}(u',u'')=
(\check R_{i,i+1}(u'',u'))^{-1}
\]
where the notation indicates that we've permuted the roles of 
the spectral parameters $u'$ and $u''$ on the r.h.s.
\junk{ instead of invoking triality, it is more convenient to This
  definition ensures that \eqref{eq:unit} is automatically satisfied
  for distinct colors of lines. {or we could use triality and
    check if normalization is same or not} }

Finally, we must define $\check R_{i,i}$:
\begin{equation}\label{eq:RstarD4}
\check R_{i,i}=
\frac{(u''-q^6u')(u''-q^2u')}{(u'-q^6u'')(u'-q^2u'')}
P_0
+\frac{u''-q^2u'}{u'-q^2u''}
P_{\omega_4}
+P_{2\omega_i}
\end{equation}
where the $P_*$ are projectors onto the various irreducible subrepresentations 
of $V_{\omega_i}\otimes V_{\omega_i}$ as $U_q(\mathfrak{d}_4)$-modules.
Note that $\check R_{i,i}$ satisfies 
$\check R_{i,i}(u',u'')\check R_{i,i}(u'',u')=1$, which is nothing
but \eqref{eq:unit} for identical colors of lines; 
in addition $\check R_{i,i}(u'=u'')=1$, which is \eqref{eq:equal}.

These $R$-matrices, with the prescribed choice of bases of $V_{\omega_i}$
($i=1,2,3$) specify the crossings of lines of all possible colors.

Looking at each equation of property~\ref{pty:ybe}, we note that
l.h.s.\ and r.h.s.\ commute with the $U_q(\mathfrak{d}_4^{(1)})$
action, so by Schur's lemma they must be proportional.  In fact, the
Yang--Baxter equation \eqref{eq:ybe} is homogeneous (i.e., independent
of the normalization of $R$-matrices) and is known to be satisfied by
the universal $R$-matrix. As explained above, our normalization also
ensures that unitarity equation \eqref{eq:unit}, as well as
\eqref{eq:equal}, are satisfied. There remain only the bootstrap
equations \eqref{eq:qtri}--\eqref{eq:qtrirev}. 
We leave it to the reader to check that
with the choice of normalization of \eqref{eq:RD4}, \eqref{eq:Rnorm}
and \eqref{eq:RstarD4}, these are indeed satisfied as well. 

\junk{this last paragraph must be removed or incorporated in 3.2:
  Lastly, we discuss property~\ref{pty:empty}. Note first that this
  holds trivially for the standard representation $V^{A_d}_{\omega_1}$ of $A_d$, 
  where the single numbers $\{0,\ldots,d\}$ index {\em all} the weights.
  Now let $V_\lambda$ be an irrep of $X_{2d}$ such that
  $\lambda$, projected to $A_d$'s weight lattice, is $\omega_1$.
  Then the $A_d$-summand of $V^{X_{2d}}_{\lambda}$ containing the
  high weight vector (a) is a copy of $V^{A_d}_{\omega_1}$, and
  (b) its weights are one full face $F$ 
  of the weight polytope of $V^{X_{2d}}_{\lambda}$.
  Consequently, the only tensor products of $V^{X_{2d}}_{\lambda}$ weight vectors 
  occurring on the face $2F$ of 
  the weight polytope of $(V^{X_{2d}}_{\lambda})^{\tensor 2}$
  are tensor products of pairs of weight vectors from $F$,
  which is exactly what (\ref{eq:e}) says.
  ...maybe this should go directly after prop \ref{pty:empty}.
  Also, is there such an $F$, in the $d=4$ case, where $\vf_d$ is
  no longer the high weight vector?}

\junk{issue of gradation: what exactly fixes the gradation of the BLUE
  lines?  the projector (triangle), of course.  and we can't quite
  invoke triality cause of that, annoying. maybe we can, insisting on
  the representation to be the dual one, i.e., transposing the
  $R$-matrix.  to check this in detail, one should write the projector
  with usual action, impose independence of $u$ of coefficients which
  should fix gradation.  }

\subsubsection{The limit $q\to0$}\label{ssec:q0d2}
Since all the properties of \S\ref{sec:genproof} are satisfied, we
can take the equality of theorem~\ref{thm:final} and carefully compute
the leading order as $q\to 0$.

We first discuss the ``rectangles'' in their three orientations, 
that is the $\mathbf{S}^\lambda|_\sigma$.
According to part (1) of proposition~\ref{prop:weight}, we can restrict $\check R_{i,i}$ to the single-number
sector, i.e., in the given basis, to a $9\times9$ matrix. One checks explicitly that these three matrices
are {\em identical}, and equal to
\begin{equation}\label{eq:Runi}
\check R_{\sss A}=
\left(
\begin{array}{ccccccccc}
 1 & 0 & 0 & 0 & 0 & 0 & 0 & 0 & 0 \\
 0 & \frac{(1-q^2)u'}{u'-q^2 u''} & 0 & \frac{q ({u''}-{u'})}{q^2 {u''}-{u'}} & 0 & 0 & 0 & 0 & 0 \\
 0 & 0 & \frac{(1-q^2)u'}{u'-q^2 u''} & 0 & 0 & 0 & \frac{q ({u''}-{u'})}{q^2 {u''}-{u'}} & 0 & 0 \\
 0 & \frac{q ({u''}-{u'})}{q^2 {u''}-{u'}} & 0 & \frac{(1-q^2)u''}{u'-q^2 u''} & 0 & 0 & 0 & 0 & 0 \\
 0 & 0 & 0 & 0 & 1 & 0 & 0 & 0 & 0 \\
 0 & 0 & 0 & 0 & 0 & \frac{(1-q^2)u'}{u'-q^2 u''} & 0 & \frac{q ({u''}-{u'})}{q^2 {u''}-{u'}} & 0 \\
 0 & 0 & \frac{q ({u''}-{u'})}{q^2 {u''}-{u'}} & 0 & 0 & 0 & \frac{(1-q^2)u''}{u'-q^2 u''} & 0 & 0 \\
 0 & 0 & 0 & 0 & 0 & \frac{q ({u''}-{u'})}{q^2 {u''}-{u'}} & 0 & \frac{(1-q^2)u''}{u'-q^2 u''} & 0 \\
 0 & 0 & 0 & 0 & 0 & 0 & 0 & 0 & 1 \\
\end{array}
\right)
\end{equation}
Theorem~\ref{thm:final} can then be rewritten
\begin{equation}\label{eq:final2}
\sum_\nu 
\tikz[scale=1.8,baseline=0.5cm]{\uptri{\lambda}{\nu}{\mu}}
\mathbf{S}^\nu|_\sigma
=
\mathbf{S}^\lambda|_\sigma
\,
\mathbf{S}^\mu|_\sigma
\end{equation}
where the subscript of the $\mathbf{S}$s is irrelevant.

Define the \defn{nilHecke}\/ $R$-matrix to be
\begin{equation}\label{eq:defnil}
(\check R_{\sss nil})_{ij}^{kl}
:=
\begin{cases}
1&(k,l)=(i,j),\ i\le j\\
u''/u'&(k,l)=(i,j),\ i>j\\
1-u''/u'&(k,l)=(j,i),\ i<j\\
0&\text{otherwise}
\end{cases}
\qquad i,j,k,l\in \{0,\ldots,d\}
\end{equation}
where $(i,j)$ is the row-index and $(k,l)$ is the column index.
\rem[gray]{careful that the implicit ordering is $0<1<\cdots<d$ but actually $d$ is the highest weight
vector\dots}

By direct inspection, we have
\begin{lem}\label{lem:D4a}
  As $q\to0$ at fixed $u',u''$, one has
  \[
  (\check R_{\sss A})_{ij}^{kl} = q^{\frac{1}{2}(-B(\vf_i,\vf_j)+B(\vf_k,\vf_l))}
  \left(    (\check R_{\sss nil})_{ij}^{kl}    +O(q^2)  \right)
  \]
  for all $i,j,k,l\in \{0,\ldots,d\}$, 
  noting that $B(\vf_i,\vf_j)=\sign(i-j)$ (cf \eqref{eq:defB}). 
\end{lem}

\begin{cor}\label{cor:schub}
  As at the end of \S\ref{ssec:funct}, let $\sigma \in \mathcal{S}_n$ 
  be a permutation and $\sigma'$ be its image in 
  $\prod_i \mathcal{S}_{p_i} \dom \mathcal{S}_n$, giving a $T$-fixed
  point on $P_-\dom GL_n(\CC)$. Then
  \begin{equation}\label{eq:almost}
    \mathbf{S}^\lambda|_\sigma
    =q^{-\ell(\lambda)} \left( S^\lambda|_{\sigma'}+O(q^2) \right)
  \end{equation}
  where $S^\lambda|_{\sigma'}$ is the restriction to the $\sigma'$ fixed point 
  of the $\lambda$ Schubert class.
\end{cor}

\junk{AK, fill this part? identify with sum over reduced subwords and
  then 4.4 of [KM]. AK: what's the plan, now that \S\ref{sec:pipe}
  exists? unrelated. we don't use that section here, just conclude
  directly by comparing with 4.4 of [KM], which is NOT the same as the
  pipedream formula, cf the comment ``Geometrically, ...'' at the end
  of \S\ref{sec:pipe}}

\begin{proof}
  $\mathbf{S}^\lambda|_\sigma$ is defined using a reduced wiring
  diagram $Q$ for $\sigma$, and extracting a matrix coefficient
  from a product $\prod_Q \check R$ of $\check R$-matrices.
  Each term in that product corresponds to a selection, at
  each crossing in $Q$, of either no crossing $(k,l)=(i,j)$
  or a transposition $(k,l)=(j,i)$; in the latter case we can
  only cross two wires so as to create an inversion
  (since $(\check R_{\ss nil})_{ij}^{jk} = 0$ if the crossing
  would destroy an inversion, namely, if $i>j$).

  As such, the $(\lambda,\sigma)$ matrix entry is a sum over 
  {\em reduced}\/ subwords of $Q$ with product $\lambda$. 
  Note that this is not
  the usual formula \cite[theorem~2.3]{FK-Groth2} that gives an
  alternating sum over subwords with nilHecke product $\lambda$;\footnote{%
    The usual formula is based on 
    M\"obius inversion on a subword complex, and its coefficients
    $\pm 1$ derive from that complex being homeomorphic to a ball.
    The formula used here is subtler, in being based on 
    the subword complex being shellable; see \cite{KM-Schubert2}.
    A similar subtlety will be discussed in the proof of
    proposition~\ref{prop:groth}.}
  rather, it is the same indexing set as in the formula
  \cite[theorem~4.4]{KM-Schubert2}.

  It remains to match up the summands in $\mathbf S^\lambda|_\sigma$
  with those in \cite[theorem~4.4]{KM-Schubert2}.
  The ``absorbable reflections'' of that theorem~4.4 are exactly the
  near-misses of wires that have already crossed once and choose
  to not cross a second time, contributing factors $u''/u'$ 
  to the product as they should.
\end{proof}

Next we proceed similarly with the equivariant puzzle itself.  
\begin{lem}\label{lem:D4b}
  At fixed $u_1,u_2$,
  \begin{multline*}
    \check R_{1,2}(u'=q^{2}u_2,u''=q^{-2}u_1)_{XY}^{WZ}=
    (-q)^{\frac{1}{2}(B(\vf_X,\vf_Y)-B(\vf_W,\vf_Z))}
    \Bigg(
    O(q^2)+
    \\
    \left.
      \begin{cases}
        1\text{ or }u_2/u_1
        &\exists T:\ \downtri{Y}{T}{X}\text{ and }\uptri{W}{T}{Z}\text{ admissible},
        \\
        &\text{the factor of }u_2/u_1\text{ occurring in the cases specified in Theorem~\ref{thm:KTd2}}
        \\
        u_2/u_1-1& \rh{X}{Y}{Z}{W}\text{ is an equivariant piece from Theorem~\ref{thm:HTd2}}
        \\
        0 & \text{otherwise}
      \end{cases}
    \right)
  \end{multline*}
  where an admissible triangle is one which is either of the form of \eqref{eq:d2tri} (up to rotation), 
  or a $K$-piece of theorem~\ref{thm:KTd2}.
  
  Similarly,
  \[
  U^{XZ}_Y = (-q)^{-\frac{1}{2}B(\vf_X,\vf_Z)}
  \left(
    \begin{cases}
      1
      &
      \uptri{X}{Y}{Z}\text{ admissible}
      \\
      0
      &
      \text{otherwise}
    \end{cases}
    \quad+O(q^2)
  \right)
  \]
\end{lem}

\begin{proof}
  This can be checked explicitly on the entries of the $R$-matrix and of $U$, 
  which are listed in appendix~\ref{app:Rd2}, noting that
  $\frac{1}{2}(B(\vf_X,\vf_Y)-B(\vf_W,\vf_Z))$ (resp.\ $-\frac{1}{2}B(\vf_X,\vf_Z)$) 
  is the inversion number of the rhombus (resp.\ triangle),
  given on the pictures inside a circle at the center of it.
\end{proof}

Summing the $\frac{1}{2}B$s in the exponents, the whole puzzle 
has a leading order in $q$ which is (at most) $q^{\sum\{\text{inversion numbers\}}}$, 
where the sum is over all triangles and rhombi
of the puzzle. According to lemma~\ref{lem:sympd1}, this sum computes
the variation of inversion numbers $\ell(\nu)-(\ell(\lambda)+\ell(\mu))$.
The expression of lemma~\ref{lem:D4b} also coincides
with the fugacities from theorem~\ref{thm:KTd2}, including the $(-1)^{\sum\{\text{inversion numbers\}}}$ of $K$-pieces. We conclude that the contribution of a puzzle
is $q^{\ell(\nu)-(\ell(\lambda)+\ell(\mu))}$ times its fugacity as in theorem~\ref{thm:KTd2} plus smaller terms as $q\to 0$.
Combining with corollary~\ref{cor:schub},
we find that both l.h.s.\ and r.h.s.\ of \eqref{eq:final2} are of 
order (at most) $q^{-\ell(\lambda)-\ell(\mu)}$, and that the coefficient
of that order of the equation is exactly the product rule for Schubert classes
as stated in theorem~\ref{thm:KTd2}.

As a consequence of theorem~\ref{thm:dual}, we also obtain for free the puzzle rule for {\em dual}\/ Schubert classes in equivariant $K$-theory
of the 2-step flag manifold, thus concluding the proof of theorem~\ref{thm:KTd2}.

\junk{unfortunately does require a $q\to0$ approach, since some of
the $R$ matrices
diverge already at $d=2$ (even $d=1$)...}

\junk{the trick, in order to not bother with twisting, is to compare the power of $q$ with the one expected from its inversion number}

\junk{careful to not get confused as I did in my mathematica file about rep vs dual. in particular the labeling of blue lines is an involution times what one would naively expect}

\subsection{Interlude: double Grothendieck polynomials}\label{sec:pipe}
As promised, we now justify (under certain assumptions) the notation $\mathbf S^\lambda|_\sigma$ that was used above. In this subsection we are back to
$d\in \{1,2,3\}$.

We assume a further normalization condition:
\begin{pty}\label{pty:norm2}
\begin{equation}\label{eq:norm2}
\tikz[baseline=0,xscale=0.5]
{
\draw[invarrow=0.75,d] (-1,-1) node[below] {$\m i$} -- (1,1) node[above] {$\m i$};
\draw[invarrow=0.75,d] (1,-1) node[below] {$\m i$} -- (-1,1) node[above] {$\m i$};
}
=1, \qquad i\in \{0,\ldots,d\}
\end{equation}
where lines have arbitrary, but identical, colors.
\end{pty}

Note that \eqref{eq:e} and the first part of property~\ref{pty:ybe} imply that in the formula above,
one can fix only the bottom labels (or only the top labels), and then the only way to choose the remaining
labels so the fugacity is nonzero are the ones in the formula.

Given a single-number string $\lambda$ of length $n$,
we now define $\mathbf{S}^\lambda_{\sss N}$ as follows: (on the picture $n=4$)
\junk{really $d$ could be anything {\em for the specializations}; however 
  for the identification to double Grothendieck polynomials, we need $\ge d$}
\begin{equation}\label{eq:defSb}
\mathbf{S}^\lambda_{\sss N}
:=
\begin{tikzpicture}[baseline=1.8cm,scale=0.75]
\foreach\x in {1,...,4}
\draw[db,invarrow=0.9] (\x,0) node[below] {$\m d$} -- node[right=-1mm,pos=0.1] {$\ss u_\x$} ++(0,5);
\foreach\y in {1,...,4}
\draw[db,invarrow=0.9] (0,5-\y) -- node[above=-1mm,pos=0.1] {$\ss t_\y$} ++(5,0) node[right] {$\m d$};
\draw[decorate,decoration=brace] (-0.3,1) -- node[left] {$\m\omega$} (-0.3,4);
\draw[decorate,decoration=brace] (1,5.3) -- node[above] {$\m\lambda$} (4,5.3);
\end{tikzpicture}
\end{equation}
where the string $\omega$ (the unique weakly increasing string with the same content as $\lambda$) is read from top to bottom. $\mathbf{S}^\lambda_{\sss NW}$ and $\mathbf{S}^\lambda_{\sss NE}$ are defined identically, except with green and red lines, respectively. 
All three live in $\QQ(t_1,\ldots,t_n,u_1,\ldots,u_n,\ldots)$. More precisely, 
denoting them collectively by $\mathbf{S}^\lambda$, and 
writing $\omega=0^{p_0}1^{p_1}\ldots d^{p_d}$, with $p_i=n_{i+1}-n_i$
(and $n_0=0$, $n_{d+1}=n$),
we have
\begin{lem}\label{lem:sym}
  $\mathbf{S}^\lambda$ is symmetric in $t_{1+n_i},\ldots,t_{n_{i+1}}$ for
  $i=0,\ldots, d$.
\end{lem}

\begin{proof}
We show invariance by elementary transposition $i,i+1$ where
$\omega_i=\omega_{i+1}$: ($i=2$ on the picture; all lines are of the same color)
\begin{align*}
\mathbf{S}^\lambda&=
\begin{tikzpicture}[baseline=1.8cm,scale=0.75]
\foreach\x in {1,...,4}
\draw[d,invarrow=0.9] (\x,0) node[below] {$\m d$} -- node[right=-1mm,pos=0.1] {$\ss u_\x$} ++(0,5);
\foreach\y/\s in {1,4}
\draw[d,invarrow=0.9] (0,5-\y) -- node[above=-1mm,pos=0.07] {$\ss t_\y$} ++(6,0) node[right] {$\m d$};
\draw[d,invarrow=0.9,rounded corners] (0,5-2) -- node[above=-1mm,pos=0.1] {$\ss t_2$} ++(4.5,0) -- ++(1,-1) -- ++(0.5,0) node[right] {$\m d$};
\draw[d,invarrow=0.9,rounded corners] (0,5-3) -- node[above=-1mm,pos=0.1] {$\ss t_3$} ++(4.5,0) -- ++(1,1) -- ++(0.5,0) node[right] {$\m d$};
\draw[decorate,decoration=brace] (-0.3,1) -- node[left] {$\m\omega$} (-0.3,4);
\draw[decorate,decoration=brace] (1,5.3) -- node[above] {$\m\lambda$} (4,5.3);
\end{tikzpicture}
&&\text{by proposition~\ref{prop:weight} (2) and property~\ref{pty:norm2}}
\\
&=
\begin{tikzpicture}[baseline=1.8cm,scale=0.75]
\foreach\x in {1,...,4}
\draw[d,invarrow=0.9] (\x,0) node[below] {$\m d$} -- node[right=-1mm,pos=0.1] {$\ss u_\x$} ++(0,5);
\foreach\y/\s in {1,4}
\draw[d,invarrow=0.9] (-1,5-\y) -- node[above=-1mm,pos=0.1] {$\ss t_\y$} ++(6,0) node[right] {$\m d$};
\draw[d,invarrow=0.9,rounded corners] (-1,5-2) -- node[above=-1mm,pos=0.5] {$\ss t_2$} ++(0.5,0) -- ++(1,-1) -- ++(4.5,0) node[right] {$\m d$};
\draw[d,invarrow=0.9,rounded corners] (-1,5-3) -- node[above=-1mm,pos=0.5] {$\ss t_3$} ++(0.5,0) -- ++(1,1) -- ++(4.5,0) node[right] {$\m d$};
\draw[decorate,decoration=brace] (-1.3,1) -- node[left] {$\m\omega$} (-1.3,4);
\draw[decorate,decoration=brace] (1,5.3) -- node[above] {$\m\lambda$} (4,5.3);
\end{tikzpicture}
&&\text{by property~\ref{pty:ybe}, \eqref{eq:ybe}}
\\
&=
\begin{tikzpicture}[baseline=1.8cm,scale=0.75]
\foreach\x in {1,...,4}
\draw[d,invarrow=0.9] (\x,0) node[below] {$\m d$} -- node[right=-1mm,pos=0.1] {$\ss u_\x$} ++(0,5);
\foreach\y/\yy in {1/1,2/3,3/2,4/4}
\draw[d,invarrow=0.9] (0,5-\y) -- node[above=-1mm,pos=0.1] {$\ss t_\yy$} ++(5,0) node[right] {$\m d$};
\draw[decorate,decoration=brace] (-0.3,1) -- node[left] {$\m\omega$} (-0.3,4);
\draw[decorate,decoration=brace] (1,5.3) -- node[above] {$\m\lambda$} (4,5.3);
\end{tikzpicture}
&&\text{by proposition~\ref{prop:weight} (2) and property~\ref{pty:norm2}, }\omega_i=\omega_{i+1}
\end{align*}
\end{proof}

The justification of the notation $\mathbf{S}^\lambda|_\sigma$ is then
\begin{lem}\label{lem:restr}
  For each permutation $\sigma\in \mathcal S_n$, we have
  $
  \mathbf{S}^\lambda|_\sigma = \mathbf{S}^\lambda|_{t_i=u_{\sigma^{-1}(i)}
    \ \forall i\in \{1,\ldots,n\} }
  $.
\end{lem}
\begin{proof}
The proof is again an obvious check of connectivity of lines. Let us work it
out on an example, such as $\sigma=4213$, $\sigma^{-1}=3241$. Then 
\begin{align*}
\mathbf{S}^\lambda|_{t_i=u_{\sigma^{-1}(i)}}
&=
\begin{tikzpicture}[baseline=1.8cm,scale=0.75]
\foreach\x in {1,...,4}
\draw[d,invarrow=0.9] (\x,0) node[below] {$\m d$} -- node[right=-1mm,pos=0.1] {$\ss u_\x$} ++(0,5);
\foreach\y/\x in {1/3,2/2,3/4,4/1}
\draw[d,invarrow=0.9] (0,5-\y) -- node[above=-1mm,pos=0.1] {$\ss u_\x$} ++(5,0) node[right] {$\m d$};
\draw[decorate,decoration=brace] (-0.3,1) -- node[left] {$\m\omega$} (-0.3,4);
\draw[decorate,decoration=brace] (1,5.3) -- node[above] {$\m\lambda$} (4,5.3);
\end{tikzpicture}
=
\begin{tikzpicture}[baseline=1.8cm,scale=0.75]
\foreach\y/\x in {1/3,2/2,3/4,4/1}
{
\draw[d,invarrow=0.9,rounded corners] (\x,0) node[below] {$\m d$} -- node[right=-1mm,pos=0.75] {$\ss u_\x$} ++(0,0.5) -- (\x,5-\y) -- (5,5-\y) node[right] {$\m d$};
\draw[d,invarrow=0.9,rounded corners] (0,5-\y) -- node[above=-1mm,pos=0.75] {$\ss u_\x$} ++(0.5,0) -- (\x,5-\y) -- (\x,5);
}
\draw[decorate,decoration=brace] (-0.3,1) -- node[left] {$\m\omega$} (-0.3,4);
\draw[decorate,decoration=brace] (1,5.3) -- node[above] {$\m\lambda$} (4,5.3);
\end{tikzpicture}
&\text{using \eqref{eq:equal}}&
\\
&=
\begin{tikzpicture}[baseline=1.8cm,scale=0.75]
\foreach\y/\x in {1/3,2/2,3/4,4/1}
{
\draw[d,invarrow=0.9,rounded corners,xshift=4cm,yshift=1cm] (\x,0) node[below] {$\m d$} -- node[right=-1mm,pos=0.75] {$\ss u_\x$} ++(0,0.5) -- (\x,5-\y) -- (5,5-\y) node[right] {$\m d$};
\draw[d,invarrow=0.9,rounded corners] (0,5-\y) -- node[above=-1mm,pos=0.75] {$\ss u_\x$} ++(0.5,0) -- (\x,5-\y) -- (\x,5);
}
\draw[decorate,decoration=brace] (-0.3,1) -- node[left] {$\m\omega$} (-0.3,4);
\draw[decorate,decoration=brace] (1,5.3) -- node[above] {$\m\lambda$} (4,5.3);
\end{tikzpicture}
&\text{using \eqref{eq:ybe} and \eqref{eq:unit}}&
  \\
&=
  \begin{tikzpicture}[baseline=1.8cm,scale=0.75]
\foreach\y/\x in {1/3,2/2,3/4,4/1}
{
\draw[d,invarrow=0.9,rounded corners] (0,5-\y) -- node[above=-1mm,pos=0.75] {$\ss u_\x$} ++(0.5,0) -- (\x,5-\y) -- (\x,5);
}
\draw[decorate,decoration=brace] (-0.3,1) -- node[left] {$\m\omega$} (-0.3,4);
\draw[decorate,decoration=brace] (1,5.3) -- node[above] {$\m\lambda$} (4,5.3);
\end{tikzpicture}
  &\hspace{-3cm}\text{using \eqref{eq:norm2} and remark below}&
  \\
&=
\mathbf{S}^\lambda|_\sigma
\end{align*}
where the last picture is the dual depiction of the one defining
$\mathbf{S}^\lambda|_\sigma$, cf~\eqref{eq:defS}.
\end{proof}

The observant reader may have recognized in the definition of $\mathbf S^\lambda$ a pipe dream formula in disguise.
More precisely, we have the following statement:
\begin{prop}\label{prop:groth}
If the $R$-matrix defining crossings in the definition \eqref{eq:defSb} is given
by \eqref{eq:defnil}, i.e., of nilHecke type, then $\mathbf{S}^\lambda$
coincides with the \defn{double Grothendieck polynomial} associated to 
$\lambda$.
\end{prop}
\tikzset{/linkpattern/edge/.style={ultra thick,solid,draw=\linkpatternedgecolor,draw opacity=1}}
\def\linkpatternedgecolor{purple!80!black}%
\begin{proof}
A \defn{pipe dream} is a filling of a $n\times n$ square with 
``plaquettes'' \tikz[baseline=-3pt]{\plaq{a}}
and
\tikz[baseline=-3pt]{\plaq{c}}, where the Southeast triangle is entirely made of
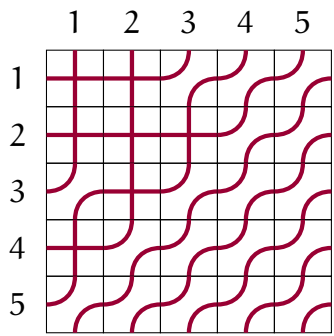
\tikz[baseline=-3pt]{\plaq{a}}, e.g.,
\[
\begin{tikzpicture}[math mode]
\node[loop]
{
         &\node{1};&\node{2};&\node{3};&\node{4};&\node{5};\\
\node{1};&\plaq{c}&\plaq{a}&\plaq{a}&\plaq{a}&\plaq{a}\\
\node{2};&\plaq{c}&\plaq{c}&\plaq{c}&\plaq{a}&\plaq{a}\\
\node{3};&\plaq{a}&\plaq{c}&\plaq{a}&\plaq{a}&\plaq{a}\\
\node{4};&\plaq{c}&\plaq{a}&\plaq{a}&\plaq{a}&\plaq{a}\\
\node{5};&\plaq{a}&\plaq{a}&\plaq{a}&\plaq{a}&\plaq{a}\\
};
\end{tikzpicture}
\]
We say that a pipe dream has inverse connectivity\footnote{To conform with
the conventions of the rest of the paper, we define the connectivity
to be the inverse of the one usually considered in the context
of pipe dreams.} $\sigma\in \mathcal{S}_n$
if the $i$th top mid edge (numbered from left to right) is connected to
the left mid edge numbered (from top to bottom) $\sigma_i$, in the following sense.
One follows the line formed by the plaquettes, with one exception: if ones crosses
another given line multiple times in the process, then all \tikz[baseline=-3pt]{\plaq{c}} 
after the first one
are ignored (for connectivity purposes they're equivalent to a \tikz[baseline=-3pt]{\plaq{a}}). 
Following \cite{KM-Schubert2},
we call such crossings \defn{absorbable}. In the example above, $\sigma=31542$ 
(paying attention to the absorbable crossing at row $3$, column $2$).

Given a single-number string $\lambda\in W_P\dom W$ (cf~\S\ref{ssec:funct}), 
define $\sigma$ to be shortest representative of $\lambda$,
i.e., the unique permutation such that $\lambda_i=\omega_{\sigma(i)}$ and $\lambda_i=\lambda_j$ implies
$\sigma_i<\sigma_j$ for $i<j$.

The \defn{double Grothendieck polynomial} associated to $\lambda$ 
is given by 
\begin{equation}\label{eq:defGroth}
  S^\lambda:=\sum_{\substack{\text{pipe dreams with}\\\text{inverse connectivity $\sigma$}}}
  (-1)^{\#\{\text{absorbable crossings}\}}\prod_{\substack{
      \tikz[baseline=-3pt,scale=0.4]{\plaq{c}}
      \text{ at row }i,\\\text{column }j}}
  (1-t_i/u_j)
\end{equation}
(see e.g. \cite[theorem~2.3]{FK-Groth2}).

We want to show that if the $R$-matrix defining the crossings of the picture
of \eqref{eq:defSb} is given by \eqref{eq:defnil},
then $\mathbf{S}^\lambda=S^\lambda$. We proceed by defining a map $\varphi$
from the pipe dreams of \eqref{eq:defGroth} to 
{\em configurations}\/ of \eqref{eq:defSb}, that is fillings of the
edges of the grid with labels in $\{0,\ldots,d\}$ such that the associated
fugacity is nonzero.

Given a pipe dream, we number the lines starting on the left side (resp.\ bottom side) of the grid according to $\omega$ (resp.\ ``$d$''), thus matching
the left and bottom boundary labels of \eqref{eq:defSb}. We then continue labeling the lines as they propagate inside the grid, in the same sense as the connectivity of a pipe dream;
that is, the labels follow the lines except across absorbable crossings, in which case the labels move Northeast as if the crossing was absent. Finally, we remove all the original lines and replace them with a square grid. On the same example, assuming $d=3$, $\lambda=21321$, 
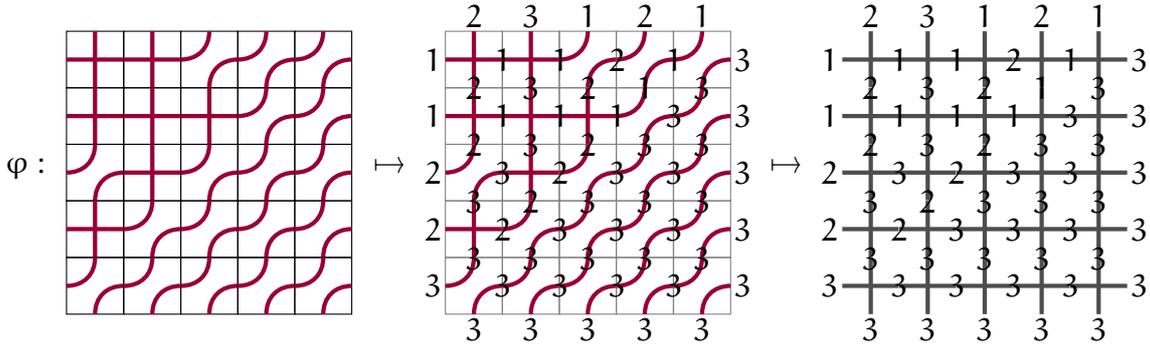
\[
\varphi: 
\begin{tikzpicture}[baseline=(current  bounding  box.center)]
\node[loop]
{
\plaq{c}&\plaq{a}&\plaq{a}&\plaq{a}&\plaq{a}\\
\plaq{c}&\plaq{c}&\plaq{c}&\plaq{a}&\plaq{a}\\
\plaq{a}&\plaq{c}&\plaq{a}&\plaq{a}&\plaq{a}\\
\plaq{c}&\plaq{a}&\plaq{a}&\plaq{a}&\plaq{a}\\
\plaq{a}&\plaq{a}&\plaq{a}&\plaq{a}&\plaq{a}\\
};
\end{tikzpicture}
\mapsto 
\tikzset{bgplaq/.style={draw=gray,fill=\linkpatternboxcolor}}
\begin{tikzpicture}[baseline=(current  bounding  box.center)]
\node[loop]
{
\plaq{c}&\plaq{a}&\plaq{a}&\plaq{a}&\plaq{a}\\
\plaq{c}&\plaq{c}&\plaq{c}&\plaq{a}&\plaq{a}\\
\plaq{a}&\plaq{c}&\plaq{a}&\plaq{a}&\plaq{a}\\
\plaq{c}&\plaq{a}&\plaq{a}&\plaq{a}&\plaq{a}\\
\plaq{a}&\plaq{a}&\plaq{a}&\plaq{a}&\plaq{a}\\
};
\foreach \i/\j/\t in {1/1/1,1/2/3,1/3/2,1/4/1,2/1/1,2/2/1,2/3/1,2/4/3,3/1/3,3/2/2,3/3/3,3/4/3,4/1/2,4/2/3,4/3/3,4/4/3,5/1/3,5/2/3,5/3/3,5/4/3} \node at (loop-\i-\j.east) {$\t$};
\foreach \i/\j/\t in {1/1/2,1/2/3,1/3/2,1/4/1,1/5/3,2/1/2,2/2/3,2/3/2,2/4/3,2/5/3,3/1/3,3/2/2,3/3/3,3/4/3,3/5/3,4/1/3,4/2/3,4/3/3,4/4/3,4/5/3} \node at (loop-\i-\j.south) {$\t$};
\foreach \i/\t in {1/1,2/1,3/2,4/2,5/3} \path (loop-\i-1.west) node[left=-1mm] {$\t$};
\foreach \j in {1,...,5} \path (loop-5-\j.south) node[below=-1mm] {$3$};
\foreach \i in {1,...,5} \path (loop-\i-5.east) node[right=-1mm] {$3$};
\foreach \j/\t in {1/2,2/1,3/3,4/2,5/1} \path (loop-1-\j.north) node[above=-1mm] {$\t$};
\end{tikzpicture}
 \mapsto
\tikzset{bgplaq/.style={draw=white,fill=\linkpatternboxcolor}}
\def\linkpatternedgecolor{gray}%
\begin{tikzpicture}[baseline=(current  bounding  box.center)]
\node[loop]
{
\plaq{c}&\plaq{c}&\plaq{c}&\plaq{c}&\plaq{c}\\
\plaq{c}&\plaq{c}&\plaq{c}&\plaq{c}&\plaq{c}\\
\plaq{c}&\plaq{c}&\plaq{c}&\plaq{c}&\plaq{c}\\
\plaq{c}&\plaq{c}&\plaq{c}&\plaq{c}&\plaq{c}\\
\plaq{c}&\plaq{c}&\plaq{c}&\plaq{c}&\plaq{c}\\
};
\foreach \i/\j/\t in {1/1/1,1/2/3,1/3/2,1/4/1,2/1/1,2/2/1,2/3/1,2/4/3,3/1/3,3/2/2,3/3/3,3/4/3,4/1/2,4/2/3,4/3/3,4/4/3,5/1/3,5/2/3,5/3/3,5/4/3} \node at (loop-\i-\j.east) {$\t$};
\foreach \i/\j/\t in {1/1/2,1/2/3,1/3/2,1/4/1,1/5/3,2/1/2,2/2/3,2/3/2,2/4/3,2/5/3,3/1/3,3/2/2,3/3/3,3/4/3,3/5/3,4/1/3,4/2/3,4/3/3,4/4/3,4/5/3} \node at (loop-\i-\j.south) {$\t$};
\foreach \i/\t in {1/1,2/1,3/2,4/2,5/3} \path (loop-\i-1.west) node[left=-1mm] {$\t$};
\foreach \j in {1,...,5} \path (loop-5-\j.south) node[below=-1mm] {$3$};
\foreach \i in {1,...,5} \path (loop-\i-5.east) node[right=-1mm] {$3$};
\foreach \j/\t in {1/2,2/1,3/3,4/2,5/1} \path (loop-1-\j.north) node[above=-1mm] {$\t$};
\end{tikzpicture}
\]
It is easy to check that the result is a valid configuration of \eqref{eq:defSb}, and
that the top labels reproduce $\lambda$.

Inversely, consider the preimage under $\varphi$ of a configuration of \eqref{eq:defSb}.
The labels around each vertex \tikz[scale=0.4,baseline=0]{\draw[ultra thick] (-1,0) node[left] {$i$}  -- (1,0) node[right] {$l$} (0,-1) node[below] {$j$} -- (0,1) node[above] {$k$};}
can be of three types, according to \eqref{eq:defnil}:
\begin{itemize}
\item $i=k$, $j=l$, $i\le j$: the latter condition means that the two lines arriving from the bottom and left sides have not crossed yet, and they are not either at this vertex; symbolically,
\[
\varphi^{-1}\Big(
\tikz[scale=0.4,baseline=0]{\draw[ultra thick] (-1,0) node[left] {$i$}  -- (1,0) node[right] {$l$} (0,-1) node[below] {$j$} -- (0,1) node[above] {$k$};}
\Big)
=
\,\tikz[baseline=-3pt]{\plaq{a}}
\qquad
i=k,\quad j=l,\quad i\le j
\]

\item Similarly, if $i=l$, $j=k$, $i< j$, the two lines that have not crossed yet do cross at the vertex:
\[
\varphi^{-1}\Big(
\tikz[scale=0.4,baseline=0]{\draw[ultra thick] (-1,0) node[left] {$i$}  -- (1,0) node[right] {$l$} (0,-1) node[below] {$j$} -- (0,1) node[above] {$k$};}
\Big)
=
\,\tikz[baseline=-3pt]{\plaq{c}}
\qquad
i=l,\quad j=k,\quad i< j
\]

\item The nontrivial case is $i=k$, $j=l$, $i> j$, in which case the two lines have crossed somewhere Southwest of the vertex. In this case, the corresponding plaquette
can either be a \tikz[baseline=-3pt]{\plaq{a}}, or an absorbable \tikz[baseline=-3pt]{\plaq{c}}:
\[
\varphi^{-1}\Big(
\tikz[scale=0.4,baseline=0]{\draw[ultra thick] (-1,0) node[left] {$i$}  -- (1,0) node[right] {$l$} (0,-1) node[below] {$j$} -- (0,1) node[above] {$k$};}
\Big)
=
\Big\{\,
\tikz[baseline=-3pt]{\plaq{a}}
,\,
\tikz[baseline=-3pt]{\plaq{c}}
\,\Big\}
\qquad
i=k,\quad j=l,\quad i> j
\]
\end{itemize}
Since identically labeled lines cannot cross each other, lines go straight from the South side to the East side, and similarly the requirement that the top labels form $\lambda$ leads to a connectivity from North to West sides which is given by $\sigma$. Therefore any preimage under $\varphi$ is a pipe dream configuration with the correct connectivity.

Furthermore, since the correspondence is purely local, we can compare fugacities one vertex/plaquette
at a time. In the first two cases, one directly finds fugacities $1$ and $1-t_i/u_j$,
respectively, in both \eqref{eq:defnil} and \eqref{eq:defGroth}. In the third case,
we find $t_i/u_j$ in \eqref{eq:defnil}, whereas \tikz[baseline=-3pt]{\plaq{a}}
(resp.\ \tikz[baseline=-3pt]{\plaq{c}})
contributes $1$ (resp.\ $(-1)\times(1-t_i/u_j)$, since the crossing is absorbable) 
to the sum of \eqref{eq:defGroth}, which also matches.
\end{proof}

Geometrically, this proposition is saying that 
if the $R$-matrix is of the form of \eqref{eq:defnil}, then $\mathbf S^\lambda$ 
is the expression of a Schubert class (in $K_T$) in terms of Chern roots
(first Chern classes or really Bott classes after splitting, 
being in $K$-theory).
Lemma~\ref{lem:restr} then tells us how to obtain from this
expression its restriction at each fixed point. The proof amounts
to providing a direct
connection between pipe-dream-type formul\ae\ for the former, and 
AJS/Billey/Graham--Willems-type formul\ae\ for the latter.

Similar ``dual'' statements can be made by defining
\[
\mathbf{S}_\lambda:=
\begin{tikzpicture}[baseline=1.8cm,scale=0.75]
\foreach\x in {1,...,4}
\draw[d,invarrow=0.9] (\x,0) -- node[right=-1mm,pos=0.1] {$\ss u_\x$} ++(0,5) node[above] {$\m d$};
\foreach\y in {1,...,4}
\draw[d,invarrow=0.9] (0,\y) node[left] {$\m d$} -- node[above=-1mm,pos=0.1] {$\ss t_\y$} ++(5,0);
\draw[decorate,decoration=brace] (5.3,4) -- node[right] {$\m\agemo$} (5.3,1);
\draw[decorate,decoration=brace] (4,-0.3) -- node[below] {$\m\lambda$} (1,-0.3);
\end{tikzpicture}
\]
Geometrically, 
if the $R$-matrix is of the form of \eqref{eq:defnil},
$\mathbf{S}_\lambda$
is the expression of a dual Schubert class (in $K_T$) 
in terms of Chern roots.

\begin{rmk*}
  The connection between the nilHecke algebra and Grothendieck polynomials 
  was first observed in \cite{FK-groth}. There, a different representation 
  is used, which is more appropriate for the {\em full}\/ flag variety. 
  The idea to introduce the ``vertex'' representation we use here is mentioned 
  (at $d=1$: the so-called ``5-vertex model'') in \cite[\S 5.2.5]{hdr}.
\end{rmk*}

\subsection{$d=3$ and $E_6$}\label{sec:d3}
The reasoning is extremely similar to the case of $d=2$, but for
technical reasons, we conclude in a slightly different manner; in
particular we use the results of the previous subsection.

\subsubsection{$R$-matrices}\label{ssec:Rd3}
In order to avoid repeats, we only mention the setup insofar as it differs from the one at $d=2$.
To each line we now attach a representation of the quantized affine
algebra $U_q(\mathfrak{e}_6^{(1)})$: to green and red lines 
we attach the evaluation representation (also called affinization \cite{Hern-affin}) of the $V_{\omega_1}$ 
of lemma~\ref{lem:greend3}, whereas to blue lines we attach the evaluation representation
based on the dual $V_{\omega_6}\cong V^*_{\omega_1}$:
\[
V_1=V_2=V_{\omega_1},\qquad V_3=V_{\omega_6}.
\]

Next we define $R$-matrices. Explicitly, in terms of the projectors 
$P_{\omega_6}$, $P_{2\omega_1}$, $P_{\omega_3}$
that intertwine the action of the horizontal subalgebra $U_q(\mathfrak{e}_6)$, we have
\begin{multline}\label{eq:defR}
\check R_{1,2}=\frac{1}{q^2(u'-q^6u'')(u'-u'')}
\Big(
(u''-q^8u')(u''-q^2u')P_{\omega_6}
+
\\
(u'-q^8u'')(u''-q^2u')
P_{\omega_3}
+
(u'-q^8u'')(u'-q^2u'')
P_{2\omega_1}
\Big)
\end{multline}
All other $R$-matrices involving red and green lines are proportional, since they're based on the same representations.
However, we choose their normalizations differently:
\begin{align}
\check R_{i,i}=
\frac{(u''-q^8u')(u''-q^2u')}{(u'-q^8u'')(u'-q^2u'')}
P_{\omega_6}
+
\frac{u''-q^2u'}{u'-q^2u''}
P_{\omega_3}
+
P_{2\omega_1}
\qquad i=1,2
\end{align}
$\check R_{3,3}$ is defined identically up to the outer automorphism of $E_6$ that switches $\omega_1$ and $\omega_6$.

It is simplest to define $\check R_{i,3}$, $i=1,2$, using the ``crossing symmetry''; namely,
\[
\check R_{i,3}(u'',u')=\mathcal P_{\omega_6,\omega_1}(\mathcal P_{\omega_1,\omega_1}\check R_{1,2}(u',u''))^{T_2}
\qquad i=1,2
\]
where $\mathcal P_{\omega_i,\omega_j}$ is the operator from $V_{\omega_i}\otimes V_{\omega_j}$ to $V_{\omega_j}\otimes V_{\omega_i}$ that switches the factors of the tensor product, and $T_2$ means partial transpose of the second factor of the tensor product.
As usual, we define all remaining $R$-matrices such that
\[
\check R_{i,j}(u'',u')=
(\check R_{j,i}(u',u''))^{-1}
\]
is satisfied for all $i,j=1,2,3$.

As in \eqref{eq:UD4} at $d=2$,
we get the factorization
\[
\check R_{1,2}(u'=q^{-8}u,\ u''=q^{-16}u)
=\frac{(1-q^{16})(1-q^{10})}{q^8(1-q^8)(1-q^2)} P_{\omega_6}=DU
\]
where $U: V_{\omega_1}(q^{-8}u)\otimes V_{\omega_1}(q^{-16}u) \to V_{\omega_6}(u)$ and $D: V_{\omega_6}(u)\to V_{\omega_1}(q^{-16}u)\otimes V_{\omega_1}(q^{-8}u)$ 
commute with the $U_q(\mathfrak{e}_6^{(1)})$ action. 
This is the analogue of property~\ref{pty:factor}, with $\alpha=q^{-16}$, $\beta=q^{-8}$.

It is a tedious but easy exercise to check that all properties~\ref{pty:factor}--\ref{pty:norm2} are satisfied in this setup. We also impose the gauge fixing condition \eqref{eq:norm}.

In order to interpret $\check R_{1,2}$, $U$ and $D$ in terms of puzzle labels, we refer the reader to
the Dynkin diagram (with the correspondence to its usual numbering being
$\rt{c}=1$, $\rt{b'}=2$, $\rt{b}=3$, $\rt{a}=4$, $\rt{a'}=5$, $\rt{c'}=6$)
and the crystal given in \S\ref{ssec:d3puz}:
with our conventions, these determine the labels of $V_1$. Translated into the language of $E_6$,
the rotation $\tau$ introduced in \S\ref{ssec:X} becomes the fourth power of a Coxeter element:
\[
\tau=(s_{b'}s_{c'}s_{a'}s_as_bs_c)^4
\]
where $s_\alpha$ is the reflection associated to the root $\alpha$.
(Note that there could be no such formula at $d=2$, where $\tau$ was
not an inner automorphism.)

The labeling of $V_2$ (resp.\ $V_3$) is then obtained from that of $V_1$ by multiplying weights
by $\tau^2$ (resp.\ $-\tau$).

\begin{ex}
  Consider the puzzle piece \uptri{(32)1}{((32)1)0}{0}. The weights of $(32)1$, $((32)1)0$ and $0$
  {\em considered as labels for vectors in $V_1$} can be read off the crystal; in the basis
  of fundamental weights with the same numbering as above, the highest weight (corresponding to the label $3$)
  is $(1,0,0,0,0,0)$, which we identify with $\vec f_3$ after taking the quotient
  by the null directions of the bilinear form, see theorem~\ref{thm:classif} and the discussion right after.
  Then
  \begin{alignat*}{2}
    \vec f_{(32)1} &= \vec f_3 - 2(\rt a+\rt b+\rt c)-(\rt a'+\rt b'+\rt c') &&= (-1, 0, 0, 0, 1, -1)
    \\
    \vec f_0 &= \vec f_3 -(\rt a+\rt b+\rt c) &&= (0, 1, 0, -1, 1, 0)
    \\
    \vec f_{((32)1)0} &= \vec f_3 - 2(\rt a+\rt b+\rt c)-(\rt a'+\rt b') &&= (-1, 0, 0, 0, 0, 1)
  \end{alignat*}
  where each root is equal in this basis to the corresponding row of the Cartan matrix.

  Now we apply $\tau^2$ and $-\tau$ to $\vec f_0$ and $\vec f_{((32)1)0}$ respectively.
  Each reflection $s_\alpha$ acts on a weight $w$ by
  negating the corresponding entry $w_\alpha\mapsto -w_\alpha$ and adding $w_\alpha$ to every $w_\beta$, where
  $\beta$ is adjacent to $\alpha$ in the Dynkin diagram.
  We compute:
  \begin{align*}
    \tau^2\vec f_0 &= (0, 0, 0, 0, -1, 1)
    \\
    -\tau\vec f_{((32)1)0} &= (-1, 0, 0, 0, 0, 0)
  \end{align*}
  We see that $\vec f_{(32)1}+\tau^2\vec f_0=-\tau\vec f_{((32)1)0}$, in accordance with \eqref{eq:YX}.
\end{ex}

\subsubsection{The limit $q\to 0$}\label{ssec:q0d3}
We start once again from theorem~\ref{thm:final}, and rewrite it using lemma~\ref{lem:restr}:
\[
\sum_\nu 
\tikz[scale=1.8,baseline=0.5cm]{\uptri{\lambda}{\nu}{\mu}}
\mathbf{S}^\nu_{\sss N}
=
\mathbf{S}^\lambda_{\sss NW}
\,
\mathbf{S}^\mu_{\sss NE}
\qquad
\text{in }\QQ(t_1,\ldots,t_n,u_1,\ldots,u_n,q)^{\prod_{i=0}^d \mathcal S_{p_i}}/\mathcal I
\]
where $\prod_{i=0}^d \mathcal S_{p_i}$ acts by permutations of the $t_i$, cf lemma~\ref{lem:sym},
and $\mathcal I$ is the ideal of rational functions that vanish when the $t$s are specialized to a permutation of the $u$s. $\mathcal I$ is generated
by the $P(t_1,\ldots,t_n)-P(u_1,\ldots,u_n)$, where $P$ runs over symmetric polynomials.

We {\em set all $u$s to one}: we obtain
\[
\sum_\nu 
\tikz[scale=1.8,baseline=0.5cm]{\uptri{\lambda}{\nu}{\mu}}|_{u_i=1}\,
\mathbf{S}^\nu_{\sss N}|_{u_i=1}
=
\mathbf{S}^\lambda_{\sss NW}|_{u_i=1}
\,
\mathbf{S}^\mu_{\sss NE}|_{u_i=1}
\qquad
\text{in }\QQ(t_1,\ldots,t_n,q)^{\prod_{i=0}^d \mathcal S_{p_i}}/\mathcal I_1
\]
where $\mathcal I_1$ is the ideal generated by $P(t_1,\ldots,t_n)-P(1,\ldots,1)$, $P\in \QQ[t_1,\ldots,t_n]^{\mathcal S_n}$.

We are now (and only now) ready to take the limit $q\to 0$ as before:
\begin{itemize}
\item According to proposition~\ref{prop:weight} part 1, we can restrict $\check R_{i,i}$ to the single-number
sector, check explicitly that these three matrices
are identical for $i=1,2,3$, and then check that lemma~\ref{lem:D4a} holds without any change at $d=3$.
\item According to proposition~\ref{prop:groth}, this implies that
\[
\mathbf{S}^\lambda=q^{-\ell(\lambda)}(S^\lambda+O(q^2))
\]
where $S^\lambda$ is the double Grothendieck polynomial associated to
$\lambda$. In particular, $\mathbf{S}^\lambda|_{u_i=1}
=q^{-\ell(\lambda)}(S^\lambda|_{u_i=1}+O(q^2))$,
where $S^\lambda|_{u_i=1}$ is the (ordinary)
Grothendieck polynomial associated to $\lambda$.

\item Finally, we examine the (nonequivariant) puzzle, i.e., we take the $q\to0$ limit of $\check R_{3,1}(q^{-16}u,q^8u)=UD$. We find:
\begin{lem} As $q\to 0$,
\begin{align*}
U^{XZ}_Y &= (-q)^{-\frac{1}{2}B(\vf_X,\vf_Z)}
\left(
\begin{cases}
1
&
\uptri{X}{Y}{Z}\text{ admissible}
\\
0
&
\text{otherwise}
\end{cases}
\quad+O(q^2)
\right)
\\
D_{ZX}^Y &= (-q)^{\frac{1}{2}B(\vf_Z,\vf_X)}
\left(
\begin{cases}
1
&
\downtri{X}{Y}{Z}\text{ admissible}
\\
0
&
\text{otherwise}
\end{cases}
\quad+O(q^2)
\right)
\end{align*}
where an admissible triangle is one which is either of the form of
\eqref{eq:d3tri} (up to rotation), or a $K$-piece as listed in 
appendix~\ref{app:d3}.
\end{lem}
The proof is a (computer-assisted) expansion at first nontrivial order
in $q$ of the $27^3$ entries of $U$ and $D$. 
The Macaulay2 code is available upon request.
\end{itemize}
Finally,
\[
\sum_\nu 
\tikz[scale=1.8,baseline=0.5cm]{\uptri{\lambda}{\nu}{\mu}}
S^\nu|_{u_i=1}
=
S^\lambda|_{u_i=1}
\,
S^\mu|_{u_i=1}
\qquad
\text{in }\ZZ[t_1^\pm,\ldots,t_n^\pm]^{\prod_{i=0}^d \mathcal S_{p_i}}/\mathcal I'_1
\]
where we have changed the ambient ring because
the $S^\lambda|_{u_i=1}$ are polynomials in the $t_i$ (in fact, we only
care that they're Laurent polynomials), and
$\mathcal I'_1=\mathcal I_1
\cap \ZZ[t_1^\pm,\ldots,t_n^\pm]^{\prod_{i=0}^d \mathcal S_{p_i}}$.
The triangle now stands for the summation over {\em nonequivariant}\/
puzzles, i.e.,
$\tikz[scale=1.8,baseline=0.5cm]{\uptri{\lambda}{\nu}{\mu}}=(-1)^{\ell(\nu)-\ell(\lambda)-\ell(\mu)}\# \{\text{such puzzles} \}$.
The quotient ring $\ZZ[t_1^\pm,\ldots,t_n^\pm]^{\prod_{i=0}^d \mathcal S_{p_i}}/\mathcal I'_1$ is nothing but the 
$K$-theory ring of the $d$-step flag variety $\{ 0=V_0\le V_1\le \cdots\le V_d\le V_{d+1}=\CC^n \}$ with $\dim V_i=n_i$, $p_i=n_{i+1}-n_i$, $i=0,\ldots,d$,
where $d=3$ here, and the $t$s are the $K$-classes of tautological line bundles (after splitting).
The $S^\lambda|_{u_i=1}$ are then the nonequivariant $K$-theoretic Schubert classes.

As explained at the end of \S\ref{sec:pipe}, the $K$-theoretic dual
Schubert classes can be obtained by $180^\circ$ degree rotation of the
pictures for Schubert classes, and application of theorem~\ref{thm:dual}
leads to a puzzle formula for these dual classes:
\[
\sum_\nu 
\tikz[scale=1.8,baseline=-0.8cm]{\downtri{\lambda}{\nu}{\mu}}
S_\nu|_{u_i=1}
=
S_\lambda|_{u_i=1}
\,
S_\mu|_{u_i=1}
\]
where once again
$\tikz[scale=1.8,baseline=-0.8cm]{\downtri{\lambda}{\nu}{\mu}}=(-1)^{\ell(\nu)-\ell(\lambda)-\ell(\mu)}\# \{\text{such puzzles} \}$.

This concludes the proof of theorem~\ref{thm:d3}.

\appendix
\section{The $d=2$ $R$-matrix}\label{app:Rd2}
These are the nonzero entries of $\check R_{1,2}(u'=q^{2}u_2,u''=q^{-2}u_1)$ as defined in \eqref{eq:RD4} and used in lemma~\ref{lem:D4b}:

\begin{itemize}
\item
$\rh[0]{0}{0}{0}{0}=\rh[0]{10}{0}{10}{0}=\rh[0]{20}{0}{20}{0}=\rh[0]{(21)0}{0}{(21)0}{0}=\rh[0]{0}{1}{0}{1}=\rh[0]{1}{1}{1}{1}=\rh[0]{21}{1}{21}{1}=\rh[2]{(21)0}{1}{(21)0}{1}=\rh[0]{0}{2}{0}{2}=\rh[0]{1}{2}{1}{2}=\rh[0]{2}{2}{2}{2}=\rh[0]{10}{2}{10}{2}=\rh[0]{1}{10}{1}{10}=\rh[2]{10}{10}{10}{10}=\rh[2]{(21)0}{10}{(21)0}{10}=\rh[0]{2(10)}{10}{2(10)}{10}=\rh[0]{2}{20}{2}{20}=\rh[2]{10}{20}{10}{20}=\rh[2]{20}{20}{20}{20}=\rh[2]{2(10)}{20}{2(10)}{20}=\rh[0]{0}{21}{0}{21}=\rh[0]{2}{21}{2}{21}=\rh[2]{20}{21}{20}{21}=\rh[2]{21}{21}{21}{21}=\rh[2]{20}{(21)0}{20}{(21)0}=\rh[0]{21}{(21)0}{21}{(21)0}=\rh[4]{(21)0}{(21)0}{(21)0}{(21)0}=\rh[2]{2(10)}{(21)0}{2(10)}{(21)0}=\rh[2]{1}{2(10)}{1}{2(10)}=\rh[0]{2}{2(10)}{2}{2(10)}=\rh[2]{21}{2(10)}{21}{2(10)}=\rh[4]{2(10)}{2(10)}{2(10)}{2(10)}=1$
\item
$\rh[2]{(21)0}{(21)0}{0}{21}=\frac{\left(1-q^2\right) u_1}{q^2 \left(u_1-q^2 u_2\right)}$
\item
$\rh[1]{(21)0}{1}{10}{2}=\rh[1]{2}{2(10)}{10}{10}=\rh[3]{(21)0}{(21)0}{10}{20}=\rh[1]{1}{2(10)}{0}{21}=\rh[1]{10}{20}{0}{21}=\rh[3]{2(10)}{2(10)}{20}{21}=\rh[3]{(21)0}{(21)0}{1}{2(10)}=\frac{-\left(1-q^2\right) u_1}{q \left(u_1-q^2 u_2\right)}$
\item
$\rh[0]{1}{10}{0}{0}=\rh[0]{2}{20}{0}{0}=\rh[0]{21}{(21)0}{0}{0}=\rh[0]{2}{21}{10}{0}=\rh[0]{2(10)}{10}{20}{0}=\rh[0]{2}{2(10)}{0}{1}=\rh[0]{2}{21}{1}{1}=\rh[0]{10}{0}{1}{1}=\rh[0]{20}{0}{21}{1}=\rh[0]{2(10)}{10}{21}{1}=\rh[0]{(21)0}{0}{1}{2}=\rh[0]{20}{0}{2}{2}=\rh[0]{21}{1}{2}{2}=\rh[0]{2(10)}{10}{2}{2}=\rh[2]{2(10)}{2(10)}{10}{2}=\rh[0]{2}{20}{1}{10}=\rh[0]{21}{(21)0}{1}{10}=\rh[2]{20}{(21)0}{10}{10}=\rh[2]{20}{20}{(21)0}{10}=\rh[2]{21}{2(10)}{(21)0}{10}=\rh[0]{21}{(21)0}{2}{20}=\rh[2]{2(10)}{20}{21}{21}=\rh[2]{10}{20}{1}{2(10)}=\rh[2]{20}{20}{21}{2(10)}=\frac{\left(1-q^2\right) u_1}{u_1-q^2 u_2}$
\item
$\rh[1]{2(10)}{(21)0}{10}{0}=\rh[1]{21}{21}{(21)0}{0}=\rh[1]{2(10)}{20}{(21)0}{0}=\rh[1]{10}{10}{0}{1}=\rh[1]{20}{(21)0}{0}{1}=\rh[1]{2(10)}{(21)0}{1}{1}=\rh[3]{2(10)}{2(10)}{(21)0}{1}=\rh[1]{20}{20}{0}{2}=\rh[1]{21}{2(10)}{0}{2}=\rh[1]{(21)0}{10}{0}{2}=\rh[1]{21}{21}{1}{2}=\rh[1]{2(10)}{20}{1}{2}=\rh[1]{20}{21}{10}{2}=\rh[1]{2(10)}{(21)0}{2}{21}=\rh[1]{20}{(21)0}{2}{2(10)}=\frac{-q \left(1-q^2\right) u_1}{u_1-q^2 u_2}$
\item
$\rh[2]{20}{21}{(21)0}{1}=\frac{q^2 \left(1-q^2\right) u_1}{u_1-q^2 u_2}$
\item
$\rh[2]{(21)0}{1}{20}{21}=\frac{\left(1-q^2\right) u_2}{q^2 \left(u_1-q^2 u_2\right)}$
\item
$\rh[1]{0}{1}{10}{10}=\rh[1]{0}{2}{(21)0}{10}=\rh[1]{0}{2}{20}{20}=\rh[1]{1}{2}{2(10)}{20}=\rh[1]{(21)0}{0}{2(10)}{20}=\rh[1]{10}{2}{20}{21}=\rh[1]{1}{2}{21}{21}=\rh[1]{(21)0}{0}{21}{21}=\rh[1]{0}{1}{20}{(21)0}=\rh[1]{2}{2(10)}{20}{(21)0}=\rh[1]{1}{1}{2(10)}{(21)0}=\rh[1]{2}{21}{2(10)}{(21)0}=\rh[1]{10}{0}{2(10)}{(21)0}=\rh[1]{0}{2}{21}{2(10)}=\rh[3]{(21)0}{1}{2(10)}{2(10)}=\frac{-\left(1-q^2\right) u_2}{q \left(u_1-q^2 u_2\right)}$
\item
$\rh[0]{1}{1}{10}{0}=\rh[0]{2}{2}{20}{0}=\rh[0]{21}{1}{20}{0}=\rh[0]{1}{2}{(21)0}{0}=\rh[0]{2}{2}{21}{1}=\rh[0]{0}{0}{1}{10}=\rh[0]{2}{2}{2(10)}{10}=\rh[0]{20}{0}{2(10)}{10}=\rh[0]{21}{1}{2(10)}{10}=\rh[0]{0}{0}{2}{20}=\rh[0]{1}{10}{2}{20}=\rh[2]{1}{2(10)}{10}{20}=\rh[2]{21}{2(10)}{20}{20}=\rh[2]{(21)0}{10}{20}{20}=\rh[2]{21}{21}{2(10)}{20}=\rh[0]{1}{1}{2}{21}=\rh[0]{10}{0}{2}{21}=\rh[2]{10}{10}{20}{(21)0}=\rh[0]{0}{0}{21}{(21)0}=\rh[0]{1}{10}{21}{(21)0}=\rh[0]{2}{20}{21}{(21)0}=\rh[0]{0}{1}{2}{2(10)}=\rh[2]{(21)0}{10}{21}{2(10)}=\rh[2]{10}{2}{2(10)}{2(10)}=\frac{\left(1-q^2\right) u_2}{u_1-q^2 u_2}$
\item
$\rh[1]{10}{2}{(21)0}{1}=\rh[1]{0}{21}{10}{20}=\rh[3]{1}{2(10)}{(21)0}{(21)0}=\rh[3]{10}{20}{(21)0}{(21)0}=\rh[1]{0}{21}{1}{2(10)}=\rh[1]{10}{10}{2}{2(10)}=\rh[3]{20}{21}{2(10)}{2(10)}=\frac{-q \left(1-q^2\right) u_2}{u_1-q^2 u_2}$
\item
$\rh[2]{0}{21}{(21)0}{(21)0}=\frac{q^2 \left(1-q^2\right) u_2}{u_1-q^2 u_2}$
\item
$\rh[-1]{1}{0}{1}{0}=\rh[-1]{2}{0}{2}{0}=\rh[-1]{21}{0}{21}{0}=\rh[-1]{2(10)}{0}{2(10)}{0}=\rh[-1]{2}{1}{2}{1}=\rh[1]{10}{1}{10}{1}=\rh[1]{20}{1}{20}{1}=\rh[1]{2(10)}{1}{2(10)}{1}=\rh[1]{20}{2}{20}{2}=\rh[1]{21}{2}{21}{2}=\rh[1]{(21)0}{2}{(21)0}{2}=\rh[1]{2(10)}{2}{2(10)}{2}=\rh[1]{0}{10}{0}{10}=\rh[-1]{2}{10}{2}{10}=\rh[1]{20}{10}{20}{10}=\rh[-1]{21}{10}{21}{10}=\rh[1]{0}{20}{0}{20}=\rh[1]{1}{20}{1}{20}=\rh[1]{21}{20}{21}{20}=\rh[3]{(21)0}{20}{(21)0}{20}=\rh[1]{1}{21}{1}{21}=\rh[1]{10}{21}{10}{21}=\rh[3]{(21)0}{21}{(21)0}{21}=\rh[3]{2(10)}{21}{2(10)}{21}=\rh[1]{0}{(21)0}{0}{(21)0}=\rh[1]{1}{(21)0}{1}{(21)0}=\rh[-1]{2}{(21)0}{2}{(21)0}=\rh[3]{10}{(21)0}{10}{(21)0}=\rh[1]{0}{2(10)}{0}{2(10)}=\rh[3]{10}{2(10)}{10}{2(10)}=\rh[3]{20}{2(10)}{20}{2(10)}=\rh[5]{(21)0}{2(10)}{(21)0}{2(10)}=\frac{q \left(u_1-u_2\right)}{u_1-q^2 u_2}$
\end{itemize}

These are the nonzero entries of $U$ and $D$, defined in \eqref{eq:UD4}:
\begin{itemize}
\item
$\uptri[1]{10}{10}{10}=\uptri[1]{10}{20}{(21)0}=\uptri[1]{20}{20}{20}=\uptri[1]{20}{21}{2(10)}=\uptri[1]{21}{2(10)}{20}=\uptri[1]{21}{21}{21}=\uptri[1]{(21)0}{10}{20}=\uptri[1]{(21)0}{1}{2(10)}=\uptri[1]{2(10)}{20}{21}=-\frac{1}{q}$
\item
$\uptri[0]{0}{0}{0}=\uptri[0]{0}{1}{10}=\uptri[0]{0}{2}{20}=\uptri[0]{0}{21}{(21)0}=\uptri[0]{1}{10}{0}=\uptri[0]{1}{1}{1}=\uptri[0]{1}{2}{21}=\uptri[0]{2}{20}{0}=\uptri[0]{2}{21}{1}=\uptri[0]{2}{2}{2}=\uptri[0]{2}{2(10)}{10}=\uptri[0]{10}{0}{1}=\uptri[0]{10}{2}{2(10)}=\uptri[0]{20}{0}{2}=\uptri[0]{21}{(21)0}{0}=\uptri[0]{21}{1}{2}=\uptri[0]{(21)0}{0}{21}=\uptri[0]{2(10)}{10}{2}=\uptri[2]{2(10)}{2(10)}{2(10)}=1$
\item
$\uptri[1]{1}{2(10)}{(21)0}=\uptri[1]{20}{(21)0}{10}=\uptri[1]{2(10)}{(21)0}{1}=-q$
\item
$\uptri[2]{(21)0}{(21)0}{(21)0}=q^2$
\end{itemize}

\begin{itemize}
\item
$\downtri[2]{(21)0}{(21)0}{(21)0}=\frac{1}{q^2}$
\item
$\downtri[1]{1}{2(10)}{(21)0}=\downtri[1]{2(10)}{(21)0}{1}=\downtri[1]{20}{(21)0}{10}=-\frac{1}{q}$
\item
$\downtri[0]{2}{2(10)}{10}=\downtri[2]{2(10)}{2(10)}{2(10)}=\downtri[0]{2}{20}{0}=\downtri[0]{21}{(21)0}{0}=\downtri[0]{2}{21}{1}=\downtri[0]{0}{21}{(21)0}=\downtri[0]{1}{1}{1}=\downtri[0]{21}{1}{2}=\downtri[0]{0}{1}{10}=\downtri[0]{1}{10}{0}=\downtri[0]{2(10)}{10}{2}=\downtri[0]{2}{2}{2}=\downtri[0]{0}{2}{20}=\downtri[0]{1}{2}{21}=\downtri[0]{10}{2}{2(10)}=\downtri[0]{0}{0}{0}=\downtri[0]{10}{0}{1}=\downtri[0]{20}{0}{2}=\downtri[0]{(21)0}{0}{21}=1$
\item
$\downtri[1]{21}{2(10)}{20}=\downtri[1]{20}{20}{20}=\downtri[1]{2(10)}{20}{21}=\downtri[1]{10}{20}{(21)0}=\downtri[1]{21}{21}{21}=\downtri[1]{20}{21}{2(10)}=\downtri[1]{(21)0}{1}{2(10)}=\downtri[1]{10}{10}{10}=\downtri[1]{(21)0}{10}{20}=-q$
\end{itemize}

\section{The $d=3$ $K$-triangles}\label{app:d3}
\def\Ktriup{10/10/10/1, 10/20/(21)0/1, 10/30/(31)0/1, 10/(32)0/((32)1)0/1, 
 10/3(20)/(3(2(10)))0/1, 20/20/20/1, 20/21/2(10)/1, 20/30/(32)0/1, 
 20/31/(32)(10)/1, 21/21/21/1, 21/2(10)/20/1, 21/31/(32)1/1, 
 21/3(10)/(32)0/1, (21)0/1/2(10)/1, (21)0/10/20/1, 
 (21)0/30/(3(21))0/1, (21)0/31/(3(21))(10)/1, 2(10)/20/21/1, 
 2(10)/30/(32)1/1, 30/30/30/1, 30/31/3(10)/1, 30/32/3(20)/1, 
 30/3(21)/3((21)0)/1, 31/31/31/1, 31/3(10)/30/1, 31/32/3(21)/1, 
 (31)0/1/3(10)/1, (31)0/10/30/1, (31)0/2/3((21)0)/1, 
 (31)0/32/3(((32)1)0)/1, 3(10)/30/31/1, 3(10)/32/3(2(10))/1, 
 32/32/32/1, 32/3(20)/30/1, 32/3(21)/31/1, 32/3(2(10))/3(10)/1, 
 32/3(((32)1)0)/(31)0/1, (32)0/2/3(20)/1, (32)0/20/30/1, 
 (32)0/21/3(10)/1, 3(20)/30/32/1, (32)1/2/3(21)/1, (32)1/21/31/1, 
 (32)1/2(10)/30/1, 3(21)/31/32/1, 3(21)/3((21)0)/30/1, 
 ((32)1)0/1/(32)(10)/1, ((32)1)0/10/(32)0/1, ((32)1)0/2/(32)((21)0)/1, 
 (3(21))0/1/3(20)/1, (32)(10)/2/3(2(10))/1, (32)(10)/20/31/1, 
 3((21)0)/30/3(21)/1, 3(2(10))/3(10)/32/1, 3(2(10))/(3(21))(10)/1/1, 
 (3(2(10)))0/10/3(20)/1, (3(21))(10)/1/3(2(10))/1, 
 (32)((21)0)/20/3(21)/1, 3(((32)1)0)/(31)0/32/1, 
 20/(3(21))0/((32)1)0/2, 20/3((21)0)/(3(2(10)))0/2, 
 21/3((21)0)/(3(21))0/2, 21/(3(21))(10)/((32)1)0/2, 
 (21)0/(31)0/((32)1)0/2, (21)0/3(10)/(3(2(10)))0/2, 
 ((32)1)0/20/(3(21))0/2, ((32)1)0/21/(3(21))(10)/2, 
 (3(21))0/21/3((21)0)/2, (3(21))0/31/3(((32)1)0)/2, 
 3((21)0)/31/3(2(10))/2, 3((21)0)/3(10)/3(20)/2, 
 3(2(10))/3((21)0)/31/2, (3(2(10)))0/20/3((21)0)/2, 
 (3(2(10)))0/30/3(((32)1)0)/2, (32)((21)0)/21/3(2(10))/2, 
 (32)((21)0)/2(10)/3(20)/2, 3(((32)1)0)/(32)0/3(21)/2, 
 3(((32)1)0)/(3(21))0/31/2, 3(((32)1)0)/(3(2(10)))0/30/2, 
 ((32)1)0/2(10)/(3(2(10)))0/3, ((32)1)0/((32)1)0/((32)1)0/3, 
 3((21)0)/3((21)0)/3((21)0)/3, 3(((32)1)0)/(32)1/3(2(10))/3, 
 3(((32)1)0)/(32)(10)/3(20)/3, 3(((32)1)0)/(3(21))(10)/3(10)/3, 
 3(((32)1)0)/(32)((21)0)/3((21)0)/4, 
 3(((32)1)0)/3(((32)1)0)/3(((32)1)0)/5}

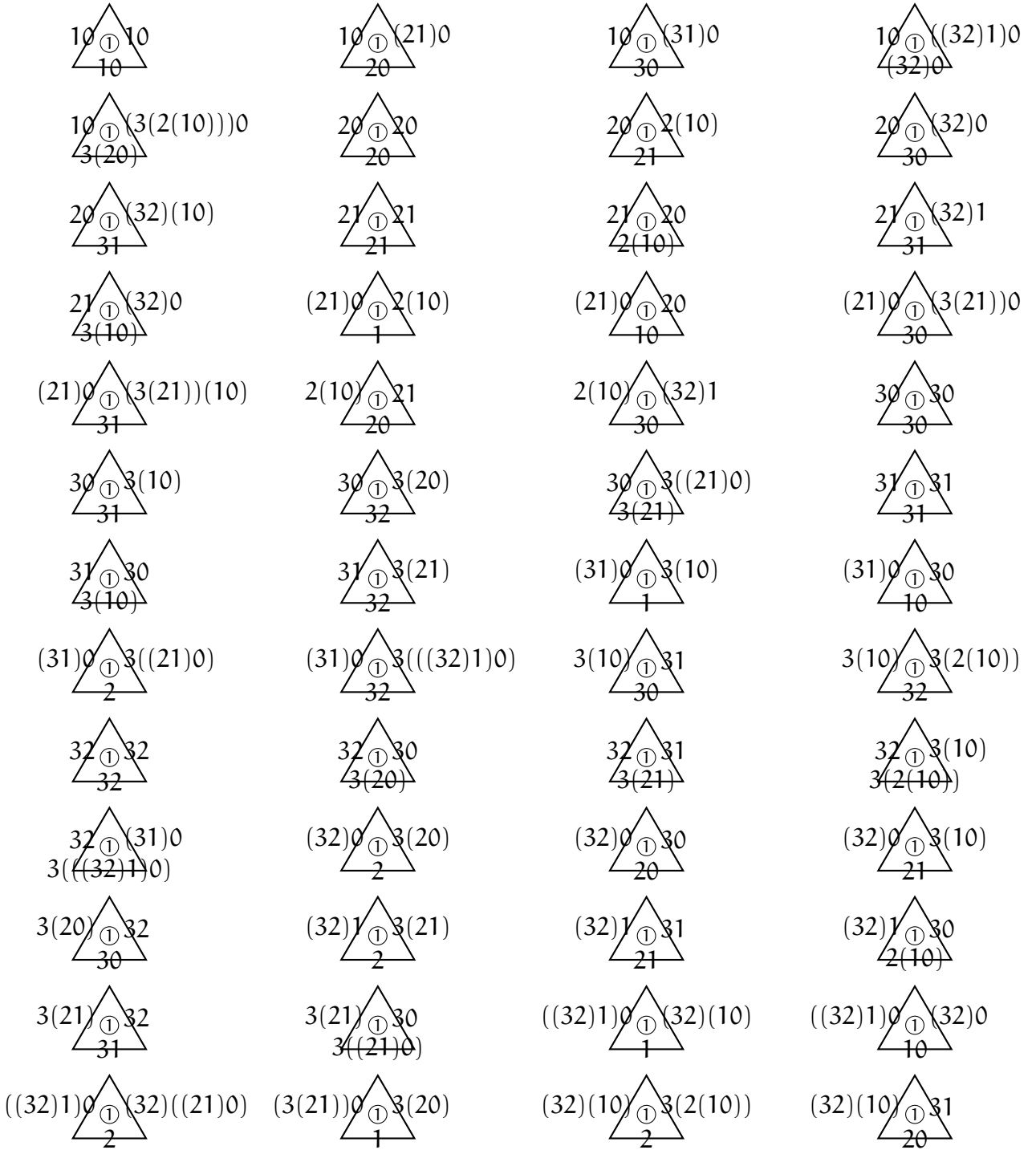
\def\Ktridown{1/2(10)/(21)0/1, 1/3(10)/(31)0/1, 1/3(20)/(3(21))0/1, 
 1/(32)(10)/((32)1)0/1, 1/3(2(10))/(3(21))(10)/1, 2/3(20)/(32)0/1, 
 2/3(21)/(32)1/1, 2/3((21)0)/(31)0/1, 2/3(2(10))/(32)(10)/1, 
 2/(32)((21)0)/((32)1)0/1, 20/(21)0/10/1, 20/3(21)/(32)((21)0)/1, 
 2(10)/(21)0/1/1, 30/(31)0/10/1, 30/(32)0/20/1, 30/(32)1/2(10)/1, 
 30/(3(21))0/(21)0/1, 31/(32)1/21/1, 31/(32)(10)/20/1, 
 31/(3(21))(10)/(21)0/1, 3(10)/(31)0/1/1, 3(10)/(32)0/21/1, 
 (32)0/((32)1)0/10/1, 3(20)/(32)0/2/1, 3(20)/(3(21))0/1/1, 
 3(20)/(3(2(10)))0/10/1, 3(21)/(32)1/2/1, 3(21)/(32)((21)0)/20/1, 
 (3(21))0/(21)0/30/1, (32)(10)/((32)1)0/1/1, 3((21)0)/(31)0/2/1, 
 3(2(10))/(32)(10)/2/1, (3(21))(10)/(21)0/31/1, 
 (32)((21)0)/((32)1)0/2/1, 
 (21)0/(21)0/(21)0/2, 
21/3(2(10))/(32)((21)0)/2, 
2(10)/3(20)/(32)((21)0)/2, (31)0/(31)0/(31)0/2, 
 (31)0/(32)0/(3(21))0/2, (31)0/(32)1/(3(21))(10)/2, 
 (31)0/((32)1)0/(21)0/2, 3(10)/(32)(10)/2(10)/2, 
 3(10)/(3(2(10)))0/(21)0/2, (32)0/(32)0/(32)0/2, 
 (32)0/(32)1/(32)(10)/2, (32)0/(3(21))0/(31)0/2, 
 3(20)/(32)((21)0)/2(10)/2, (32)1/(32)1/(32)1/2, 
 (32)1/(32)(10)/(32)0/2, (32)1/(3(21))(10)/(31)0/2, 
 ((32)1)0/(21)0/(31)0/2, (3(21))0/(31)0/(32)0/2, 
 (3(21))0/((32)1)0/20/2, (32)(10)/(32)0/(32)1/2, 
 (3(2(10)))0/(21)0/3(10)/2, (3(21))(10)/(31)0/(32)1/2, 
 (3(21))(10)/((32)1)0/21/2, (3(21))0/(32)1/(32)((21)0)/3, 
 (32)(10)/(32)(10)/(32)(10)/3, (32)(10)/(3(2(10)))0/(31)0/3, 
 (3(2(10)))0/(31)0/(32)(10)/3, (3(2(10)))0/(32)0/(32)((21)0)/3, 
 (3(2(10)))0/((32)1)0/2(10)/3, (3(21))(10)/(32)(10)/(32)((21)0)/4}

\noindent\foreach[count=\a from 0] \x/\y/\z/\t in \Ktriup%
{%
\tikz[scale=1.2]{
\useasboundingbox (-1.8,-0.2) rectangle (1.8,1);
\uptri[\t]\x\y\z
} %
\pgfmathparse{mod(\a,4)==3 ? "\noexpand\linebreak" : ""}\pgfmathresult%
}%
\foreach[count=\a from 3] \x/\y/\z/\t in \Ktridown%
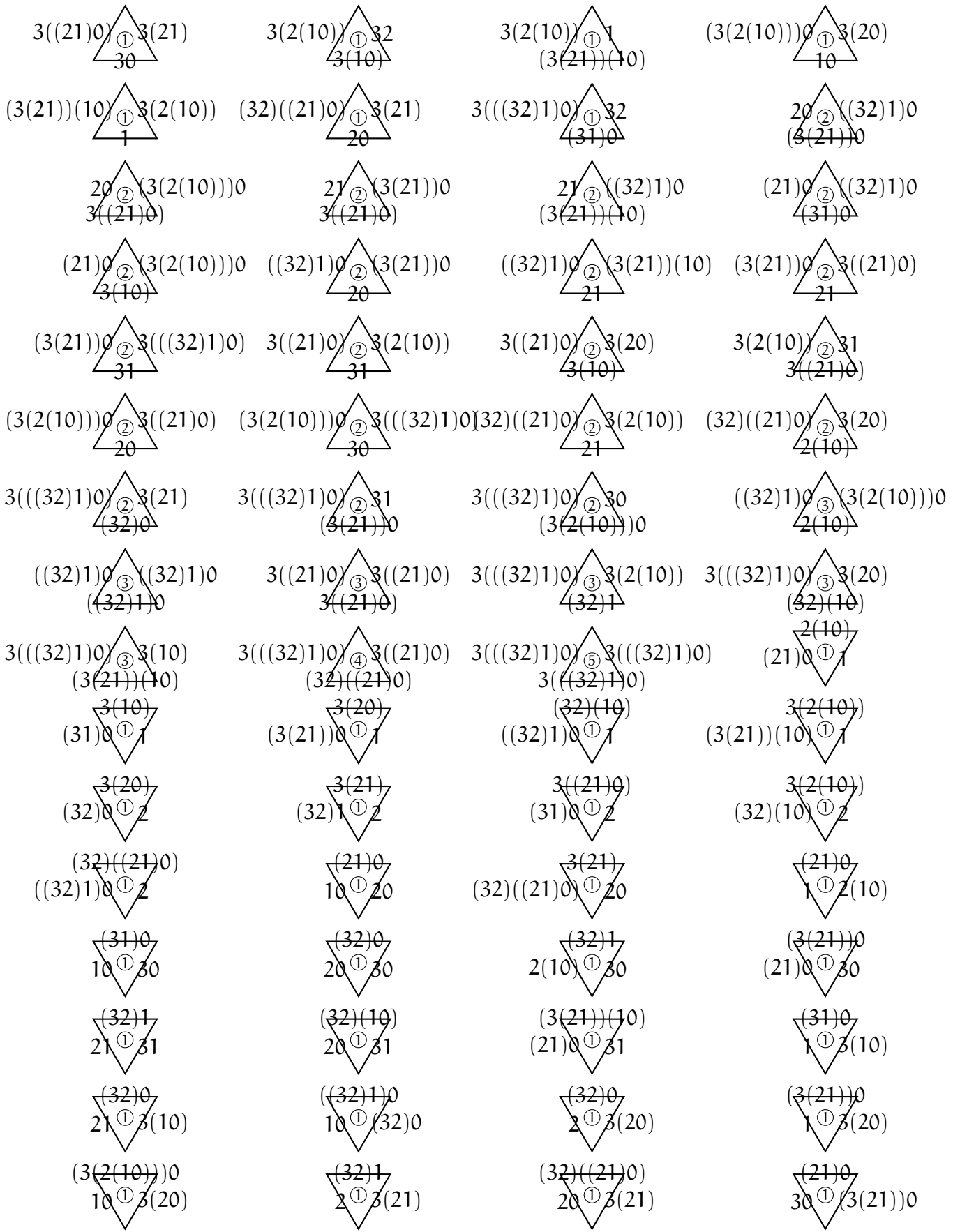
{%
\tikz[scale=1.2]{
\useasboundingbox (-1.8,-1) rectangle (1.8,0.2);
\downtri[\t]\x\y\z
} %
\pgfmathparse{mod(\a,4)==3 ? "\noexpand\linebreak" : ""}\pgfmathresult%
}
\linebreak

\junk{\section{The potential $d=3$ equivariant rhombi}\label{app:d3eq}
\def\eqrh{3/(32)((21)0)/3/(32)((21)0)/-1,3/(32)((21)0)/(21)0/((32)1)0/-1,3/(32)1/2(10)/((32)1)0/-1,3/(3(21))(10)/3/(3(21))(10)/-1,3/(3(21))(10)/10/((32)1)0/-1,3/(32)(10)/3/(32)(10)/-1,3/(3(21))0/3/(3(21))0/-1,3(((32)1)0)/(21)0/3(((32)1)0)/(21)0/-1,3/(32)(10)/20/((32)1)0/-1,3(((32)1)0)/10/3(((32)1)0)/10/-1,3(((32)1)0)/0/3(((32)1)0)/0/-1,3/(3(2(10)))0/3/(3(2(10)))0/-1,3/(32)0/21/((32)1)0/-1,32/(3(21))(10)/32/(3(21))(10)/-1,32/(3(21))0/32/(3(21))0/-1,3/2/3((21)0)/((32)1)0/-1,32/(3(2(10)))0/32/(3(2(10)))0/-1,32/(31)0/32/(31)0/-1,32/21/3(10)/(21)0/-1,32/2(10)/32/2(10)/-1,(32)((21)0)/10/(32)((21)0)/10/-1,(32)((21)0)/0/(32)((21)0)/0/-1,3(21)/(3(2(10)))0/3(21)/(3(2(10)))0/-1,32/1/3(20)/(21)0/-1,3(21)/(31)0/3(21)/(31)0/-1,3/21/3(10)/((32)1)0/-1,3(21)/2(10)/3(21)/2(10)/-1,(32)1/(21)0/(32)1/(21)0/-1,3(21)/20/3(21)/20/-1,(32)1/10/(32)1/10/-1,(3(21))(10)/0/(3(21))(10)/0/-1,3/(21)0/32/((32)1)0/-1,32/10/3(21)/(21)0/-1,(32)1/0/3(2(10))/30/-1,3(21)/0/3(2(10))/10/-1,32/10/32/10/-1,3(21)/0/3(21)/0/-1,3/2(10)/3/2(10)/-1,3/(21)0/3/(21)0/-1,(32)1/0/(32)1/0/-1,3(2(10))/(31)0/3(2(10))/(31)0/-1,3/2(10)/30/((32)1)0/-1,3(2(10))/20/3(2(10))/20/-1,3(2(10))/1/3(2(10))/1/-1,3((21)0)/10/3((21)0)/10/-1,(32)(10)/0/3(2(10))/31/-1,3((21)0)/0/3((21)0)/0/-1,(32)(10)/0/(32)(10)/0/-1,3(20)/((32)1)0/3(20)/((32)1)0/-1,32/0/3(2(10))/(21)0/-1,3/20/31/((32)1)0/-1,3(20)/(21)0/3(20)/(21)0/-1,3(20)/10/3(21)/1/-1,3(20)/10/3(20)/10/-1,3(20)/0/3(2(10))/1/-1,31/((32)1)0/31/((32)1)0/-1,3/1/3(20)/((32)1)0/-1,31/(21)0/32/20/-1,31/(21)0/31/(21)0/-1,31/10/3(21)/20/-1,3/10/3(21)/((32)1)0/-1,31/0/3(2(10))/20/-1,3(10)/0/3(2(10))/21/-1,3/0/3(2(10))/((32)1)0/-1,2/((32)1)0/3/(31)0/-1,2/((32)1)0/2/((32)1)0/-1,2/(21)0/32/(31)0/-1,2/(21)0/2/(21)0/-1,2/1/3(20)/(31)0/-1,21/10/3(21)/(32)0/-1,21/10/21/10/-1,2/10/3(21)/(31)0/-1,21/0/3(2(10))/(32)0/-1,21/0/21/0/-1,2/10/2/10/-1,2(10)/0/3(2(10))/(32)1/-1,2(10)/0/2(10)/0/-1,2/0/3(2(10))/(31)0/-1,1/0/3(2(10))/(3(21))0/-1,1/0/1/0/-1,3/((32)1)0/3/((32)1)0/-2,32/(21)0/32/(21)0/-2,3(21)/10/3(21)/10/-2,3(2(10))/0/3(2(10))/0/-2,32/((32)1)0/32/((32)1)0/-3,3(21)/((32)1)0/3(21)/((32)1)0/-3,3(21)/(21)0/3(21)/(21)0/-3,3(2(10))/((32)1)0/3(2(10))/((32)1)0/-3,3(2(10))/(21)0/3(2(10))/(21)0/-3,3(2(10))/10/3(2(10))/10/-3}
Here are all the $d=3$ rhombi with negative inversion number, up to $180^\circ$ rotation,
and up to order-preserving substitutions of the integers $0,\ldots,d$:

\noindent\foreach[count=\a from 0] \x/\y/\z/\w/\t in \eqrh%
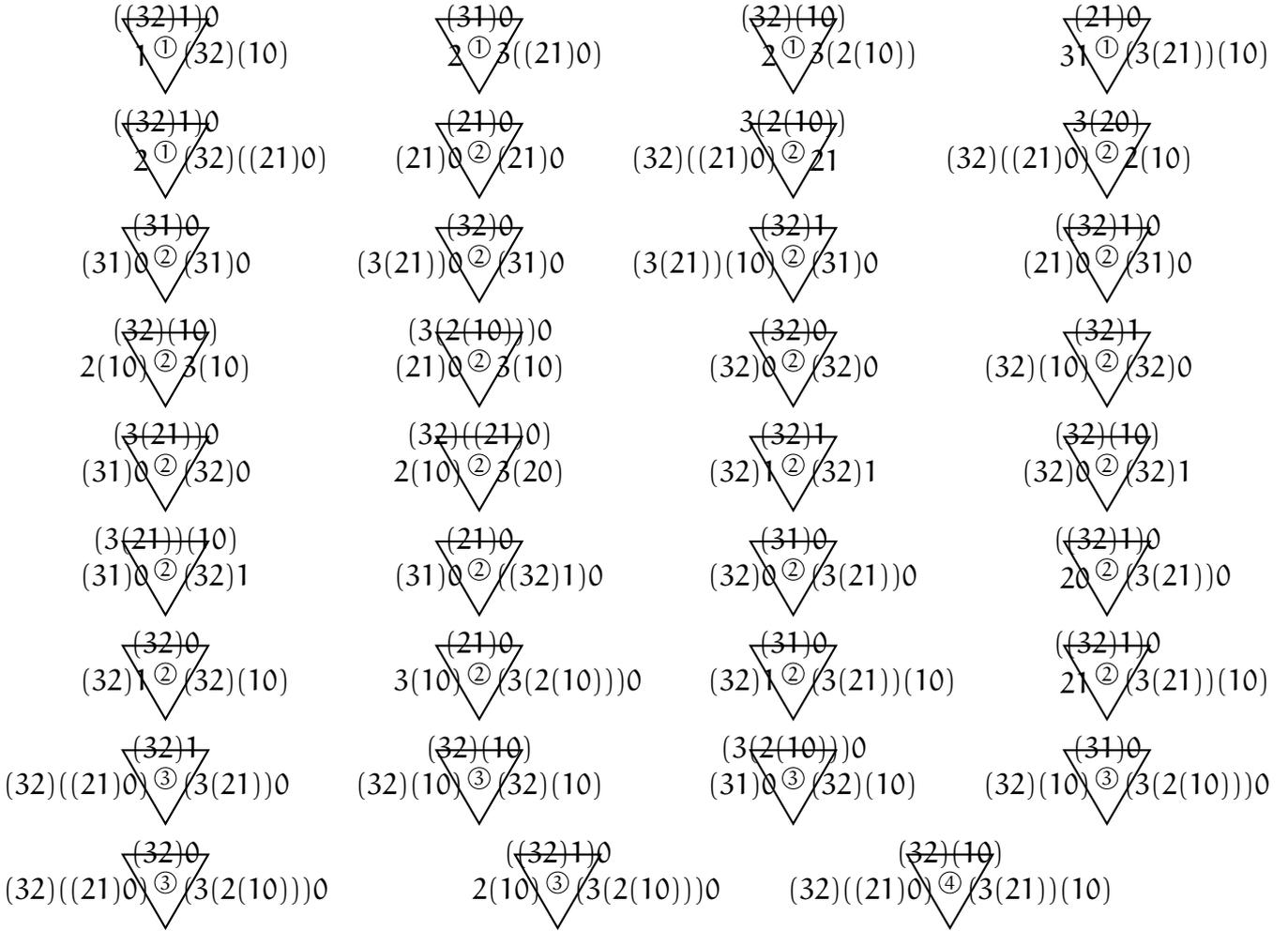
{%
\tikz[scale=1.2]{
\useasboundingbox (-1.8,-0.9) rectangle (1.8,0.9);
\rh[\t]\x\y\z\w
} %
\pgfmathparse{mod(\a,4)==3 ? "\noexpand\linebreak" : ""}\pgfmathresult%
}%
}

\def\hexa#1#2#3#4#5#6{\draw (0:1) -- coordinate[midway,label=right:$\ss #4$] (a) (60:1) coordinate (v) -- coordinate[midway,label=above:$\ss #3$] (b) (120:1) -- coordinate[midway,label=left:$\ss #2$] (c) (180:1) coordinate (u) -- coordinate[midway,label=left:$\ss #1$] (d) (240:1) -- coordinate[midway,label=below:$\ss #6$] (e) (300:1) coordinate (w) -- coordinate[midway,label=right:$\ss #5$] (f) (0:1);
\draw (0,0) -- coordinate[midway] (g) (60:1);
\draw (0,0) -- coordinate[midway] (i) (180:1);
\draw (0,0) -- coordinate[midway] (h) (300:1);
\coordinate (o) at (0,0);
\pic[draw,angle radius=5mm] {angle=d--u--o};\path (u) node[shift={(-30:0.3)}] {$\ss u$};
\pic[draw,angle radius=5mm] {angle=f--w--o};\path (w) node[above=1mm] {$\ss w$};
\pic[draw,angle radius=5mm] {angle=b--v--o};\path (v) node[shift={(210:0.3)}] {$\ss v$};
}
\def\hexb#1#2#3#4#5#6{\draw (0:1) coordinate (v) -- coordinate[midway,label=right:$\ss #4$] (a) (60:1) -- coordinate[midway,label=above:$\ss #3$] (b) (120:1) coordinate (u) -- coordinate[midway,label=left:$\ss #2$] (c) (180:1) -- coordinate[midway,label=left:$\ss #1$] (d) (240:1) coordinate (w) -- coordinate[midway,label=below:$\ss #6$] (e) (300:1) -- coordinate[midway,label=right:$\ss #5$] (f) (0:1);
\draw (0,0) -- coordinate[midway] (i) (0:1);
\draw (0,0) -- coordinate[midway] (h) (120:1);
\draw (0,0) -- coordinate[midway] (g) (240:1);
\coordinate (o) at (0,0);
\pic[draw,angle radius=5mm] {angle=o--u--b};\path (u) node[shift={(-30:0.3)}] {$\ss u$};
\pic[draw,angle radius=5mm] {angle=o--w--d};\path (w) node[above=1mm] {$\ss w$};
\pic[draw,angle radius=5mm] {angle=o--v--f};\path (v) node[shift={(210:0.3)}] {$\ss v$};
}

\section{The cyclic Yang--Baxter equation}\label{app:cyclic}
We start from the following Yang--Baxter equation, in the framework
of \S\ref{sec:genproof}:
\begin{equation}\label{eq:usualYBE}
\begin{tikzpicture}[baseline=0,scale=1.8]
\path (0:1) coordinate (v) -- coordinate[midway] (a) (60:1) -- coordinate[midway] (b) (120:1) coordinate (u) -- coordinate[midway] (c) (180:1) -- coordinate[midway] (d) (240:1) coordinate (w) -- coordinate[midway] (e) (300:1) -- coordinate[midway] (f) (0:1);
\path (0,0) -- coordinate[midway] (i) (0:1);
\path (0,0) -- coordinate[midway] (h) (120:1);
\path (0,0) -- coordinate[midway] (g) (240:1);
\coordinate (o) at (0,0);
\draw[dr,invarrow=0.4,rounded corners] (d) -- node[pos=0.1,right=-1mm] {$\ss u_1$} (h) -- (a);
\draw[dg,arrow=0.4,rounded corners] (c) -- node[pos=0.1,right=-1mm] {$\ss u_3$} (g) -- (f);
\draw[db,arrow=0.4,rounded corners] (b) -- node[pos=0.15,right] {$\ss u_2$} (i) -- (e);
\end{tikzpicture}
=
\begin{tikzpicture}[baseline=0,scale=1.8]
\path (0:1) -- coordinate[midway] (a) (60:1) coordinate (u) -- coordinate[midway] (b) (120:1) -- coordinate[midway] (c) (180:1) coordinate (w) -- coordinate[midway] (d) (240:1) -- coordinate[midway] (e) (300:1) coordinate (v) -- coordinate[midway] (f) (0:1);
\path (0,0) -- coordinate[midway] (g) (60:1);
\path (0,0) -- coordinate[midway] (i) (180:1);
\path (0,0) -- coordinate[midway] (h) (300:1);
\coordinate (o) at (0,0);
\draw[dr,invarrow=0.4,rounded corners] (d) -- node[pos=0.15,above] {$\ss u_1$} (h) -- (a);
\draw[dg,arrow=0.4,rounded corners] (c) -- node[pos=0.15,below] {$\ss u_3$} (g) -- (f);
\draw[db,arrow=0.4,rounded corners] (b) -- node[pos=0.15,right] {$\ss u_2$} (i) -- (e);
\end{tikzpicture}
\end{equation}

We now assume that to each crossing is associated an $R$-matrix coming from
an untwisted quantized affine algebra $U_q(\mathfrak{g}^{(1)})$. Then one can use
the {\em crossing symmetry}\/ to restore the cyclic symmetry of the picture:
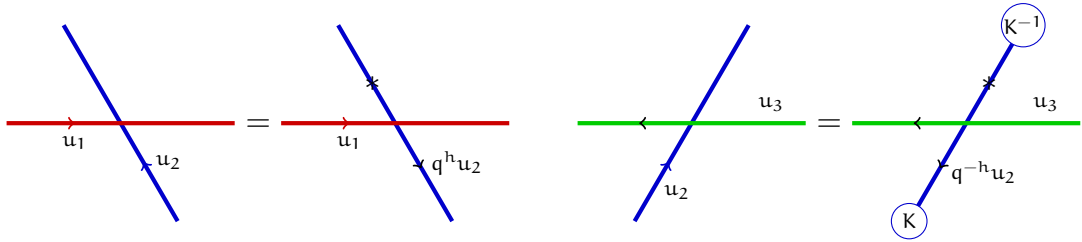
\[
\tikz[baseline=-3pt,scale=1.5]
{
\draw[db,invarrow=0.3] (120:-1) -- node[pos=0.3,right] {$\ss u_2$} (120:1);
\draw[dr,invarrow=0.3] (0:-1) -- node[pos=0.3,below] {$\ss u_1$} (0:1);
}
=
\tikz[baseline=-3pt,scale=1.5]
{
\draw[db,arrow=0.3] (120:-1) -- node[pos=0.7] {$*$} node[pos=0.3,right] {$\ss q^h u_2$} (120:1);
\draw[dr,invarrow=0.3] (0:-1) -- node[pos=0.3,below] {$\ss u_1$} (0:1);
}
\qquad
\tikz[baseline=-3pt,scale=1.5]
{
\draw[db,invarrow=0.3] (60:-1) -- node[pos=0.15,right] {$\ss u_2$} (60:1);
\draw[dg,arrow=0.3] (0:-1) -- node[pos=0.85,above] {$\ss u_3$} (0:1);
}
=
\tikz[baseline=-3pt,scale=1.5]
{
\draw[db,arrow=0.3] (60:-1) node[circle,fill=white,draw=\db,thin,inner sep=2pt] {$\ss K$} -- node[pos=0.25,right] {$\ss q^{-h}u_2$} node[pos=0.7] {$*$} (60:1) node[circle,draw=\db,thin,fill=white,inner sep=0pt] {$\ss K^{-1}$};
\draw[dg,arrow=0.3] (0:-1) -- node[pos=0.85,above] {$\ss u_3$} (0:1);
}
\]
\junk{\[
\tikz[baseline=0,scale=1.8]{\rh{}{}{}{}
\draw[dotted,invarrow=0.3] (-60:0.5) -- node[pos=0.3,above left=-1mm] {$\ss u_1$} ++(60:1);
\draw[dotted,arrow=0.7] (60:0.5) -- node[pos=0.7,above right=-1mm] {$\ss u_2$} node[pos=0.3] {$*$} ++(-60:1);
}
=
\tikz[baseline=0,scale=1.8]{\rh{}{}{}{}
\draw[dotted,invarrow=0.3] (-60:0.5) -- node[pos=0.3,above left=-1mm] {$\ss u_1$} ++(60:1);
\draw[dotted,invarrow=0.7] (60:0.5) -- node[pos=0.7,above right=-1mm] {$\ss q^{12} u_2$} ++(-60:1) node[circle,fill=white,draw=black,solid,inner sep=0pt] {$\ss K^{-1}$};
}
,
\qquad
\tikz[baseline=0,scale=1.8]{\rh{}{}{}{}
\draw[dotted,invarrow=0.3] (-60:0.5) -- node[pos=0.3,above left=-1mm] {$\ss u_1$} node[pos=0.7] {$*$} ++(60:1);
\draw[dotted,arrow=0.7] (60:0.5) -- node[pos=0.7,above right=-1mm] {$\ss u_2$} ++(-60:1);
}
=
\tikz[baseline=0,scale=1.8]{\rh{}{}{}{}
\draw[dotted,arrow=0.3] (-60:0.5) -- node[pos=0.3,above left=-1mm] {$\ss q^{-12}u_1$} ++(60:1);
\draw[dotted,arrow=0.7] (60:0.5) -- node[pos=0.7,above right=-1mm] {$\ss u_2$} ++(-60:1);
}
\]}
where $\ast$ denotes the {\em dual}\/ representation, the circled $K^{\pm 1}$ denotes multiplication by the Cartan element $K=q^{\rho}$, $\rho$ half the sum of positive roots of $\mathfrak{g}$, 
and $h$ is the dual Coxeter number of $\mathfrak{g}$.
Note the conventional choice of dual (rather than predual) representation, hence the apparent discrepancy between
the two equalities.

We can thus rewrite the Yang--Baxter equation \eqref{eq:usualYBE}:
\begin{equation}\label{eq:preweirdYBE}
\begin{tikzpicture}[baseline=0,scale=1.8]
\path (0:1) coordinate (v) -- coordinate[midway] (a) (60:1) -- coordinate[midway] (b) (120:1) coordinate (u) -- coordinate[midway] (c) (180:1) -- coordinate[midway] (d) (240:1) coordinate (w) -- coordinate[midway] (e) (300:1) -- coordinate[midway] (f) (0:1);
\path (0,0) -- coordinate[midway] (i) (0:1);
\path (0,0) -- coordinate[midway] (h) (120:1);
\path (0,0) -- coordinate[midway] (g) (240:1);
\coordinate (o) at (0,0);
\draw[dr,invarrow=0.4,rounded corners] (d) -- node[pos=0.15,left] {$\ss u_1$} (h) -- (a);
\draw[dg,arrow=0.35,rounded corners] (c) -- node[pos=0.15,left] {$\ss u_3$} (g) -- (f);
\draw[db,invarrow=0.35,rounded corners] (b) -- node[pos=0.15,right] {$\ss q^h u_2$} (i) node[circle,thin,draw=\db,fill=white,inner sep=0pt] {$\ss K^{-1}$} -- node[pos=0.3] {$*$} node[pos=0.8,right] {$\ss q^{-h} u_2$} (e) node[circle,fill=white,draw=\db,thin,inner sep=2pt] {$\ss K$};
\end{tikzpicture}
=
\begin{tikzpicture}[baseline=0,scale=1.8]
\path (0:1) -- coordinate[midway] (a) (60:1) coordinate (u) -- coordinate[midway] (b) (120:1) -- coordinate[midway] (c) (180:1) coordinate (w) -- coordinate[midway] (d) (240:1) -- coordinate[midway] (e) (300:1) coordinate (v) -- coordinate[midway] (f) (0:1);
\path (0,0) -- coordinate[midway] (g) (60:1);
\path (0,0) -- coordinate[midway] (i) (180:1);
\path (0,0) -- coordinate[midway] (h) (300:1);
\coordinate (o) at (0,0);
\draw[dr,invarrow=0.4,rounded corners] (d) -- node[pos=0.15,below] {$\ss u_1$} (h) -- (a);
\draw[dg,arrow=0.35,rounded corners] (c)  -- node[pos=0.15,above] {$\ss u_3$} (g) -- (f);
\draw[db,invarrow=0.4,rounded corners] (b) node[circle,thin,draw=\db,fill=white,inner sep=0pt] {$\ss K^{-1}$} -- node[pos=0.25,right] {$\ss q^{-h}u_2$} (i)  node[circle,fill=white,draw=\db,thin,inner sep=2pt] {$\ss K$} -- node[pos=0.3] {$*$} node[pos=0.75,right] {$\ss q^h u_2$} (e);
\end{tikzpicture}
\end{equation}
where something slightly strange happens in that the parameter of the
blue line ``jumps'' from $q^{\pm h}u_2$ to $q^{\mp h}u_2$ as it goes across the central circle.

Observe that
for $\lie{g}=\lie{a}_2,\lie{d}_4,\lie{e_6},\lie{e_8}$,
the dual Coxeter number
$h=3,6,12,30$ is a multiple of $3$.  Furthermore,
for the representations used (implicitly) at $d=1$ in \cite{artic68},
at $d=2$ in \S\ref{sec:d2}, at $d=3$ in \S\ref{sec:d3}, the ratio of
spectral parameters $q^{2h/3}$ is precisely the value at which the
$R$-matrices corresponding to two types of lines turn into projectors
onto the third type of line. It is therefore very tempting to redefine
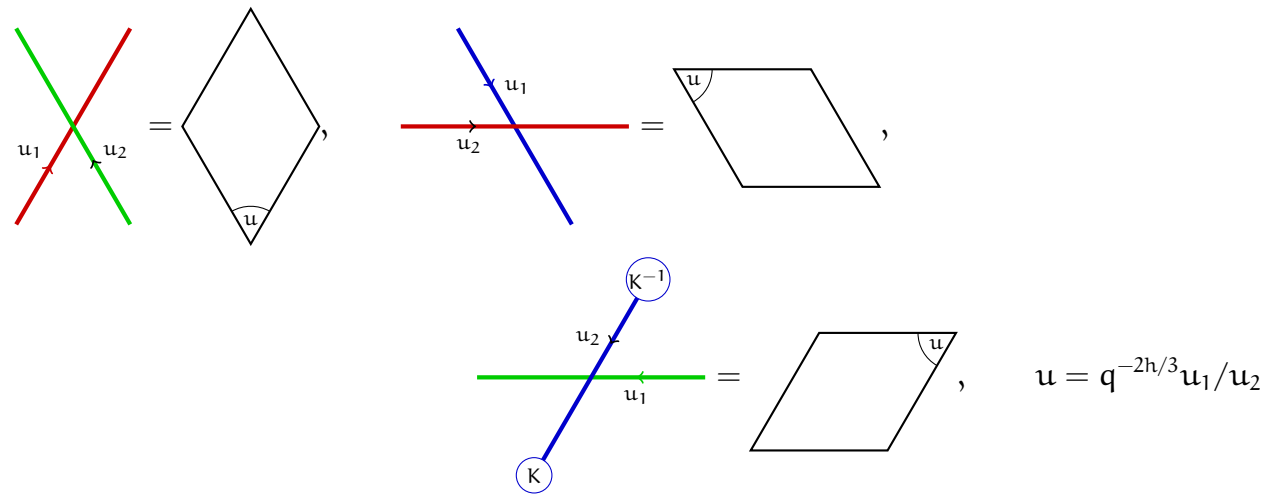
\begin{multline*}
\tikz[baseline=-3pt,scale=1.5]{
\draw[dr,invarrow=0.75] (60:-1) -- node[pos=0.75,right] {$\ss u''$} (60:1);
\draw[dg,arrow=0.25] (-60:-1) -- node[pos=0.25,left] {$\ss u'$} (-60:1);
}
=
\tikz[baseline=-3pt,scale=1.8]{
\rh{}{}{}{}
\pic[draw,angle radius=5mm] {angle=cb--ba--ad};\path (ba) node[above=1mm] {$\ss u$};
},
\qquad
\tikz[baseline=-3pt,scale=1.5,rotate=-120]{
\draw[db,invarrow=0.75] (60:-1) -- node[pos=0.75,left] {$\ss u''$} (60:1);
\draw[dr,arrow=0.25] (-60:-1) -- node[pos=0.25,above] {$\ss u'$} (-60:1);
}
=
\tikz[baseline=-0.866cm,scale=1.8,rotate=-120]{
\rh{}{}{}{}
\pic[draw,angle radius=5mm] {angle=cb--ba--ad};\path (ba) node[shift={(-30:0.3)}] {$\ss u$};
},
\\
\tikz[baseline=-3pt,scale=1.5,rotate=120]{
\draw[dg,invarrow=0.75] (60:-1) -- node[pos=0.75,above] {$\ss u''$} (60:1);
\draw[db,arrow=0.25] (-60:-1)  node[circle,fill=white,draw=\db,thin,inner sep=2pt] {$\ss K$} -- node[pos=0.25,right] {$\ss u'$} (-60:1) node[circle,fill=white,draw=\db,thin,inner sep=0pt] {$\ss K^{-1}$};
}
=
\tikz[baseline=0.866cm,scale=1.8,rotate=120]{
\rh{}{}{}{}
\pic[draw,angle radius=5mm] {angle=cb--ba--ad};\path (ba) node[shift={(210:0.3)}] {$\ss u$};
},
\qquad u=q^{-2h/3}u''/u'
\end{multline*}
where the rhombi are not allowed to be rotated.
Note that by conjugation, one could share $K$ equally between the three
orientations of rhombi, thus restoring $\ZZ/3$-symmetry.
Up to such possible conjugations, the first rhombus will eventually play the role of equivariant rhombus,
where $u$ is the ratio of equivariant parameters.

Setting $u=q^{h/3}u_2/u_1$, $v=q^{h/3}u_3/u_2$, $w=q^{-2h/3} u_1/u_3$
(with $uvw=1$), we find that
Eq.~\eqref{eq:preweirdYBE} becomes 
the ``cyclic'' Yang--Baxter equation:
\begin{equation}\label{eq:weirdYBE}
\tikz[baseline=0,scale=1.8]{\hexb{}{}{}{}{}{}}
=
\tikz[baseline=0,scale=1.8]{\hexa{}{}{}{}{}{}}
\qquad
\forall u,v,w:\ uvw=1
\end{equation}
Compare with \cite[proposition~1]{artic46} (which corresponds to the
rational limit where the spectral parameters are expanded to first
order around $1$) and \cite[proposition~3]{artic68}, both related to
the $d=1$ case in the limit $q\to0$. We could then proceed as in these
papers to derive our puzzle rules for general $d \leq 3$, although
this is quite involved technically.

\gdef\MRshorten#1 #2MRend{#1}%
\gdef\MRfirsttwo#1#2{\if#1M%
MR\else MR#1#2\fi}
\def\MRfix#1{\MRshorten\MRfirsttwo#1 MRend}
\renewcommand\MR[1]{\relax\ifhmode\unskip\spacefactor3000 \space\fi
\MRhref{\MRfix{#1}}{{\scriptsize \MRfix{#1}}}}
\renewcommand{\MRhref}[2]{%
\href{http://www.ams.org/mathscinet-getitem?mr=#1}{#2}}
\bibliographystyle{amsalphahyper}
\bibliography{biblio}

\end{document}